\documentclass[11pt,oneside,
 ]{amsart}
\usepackage[T1]{fontenc}
\usepackage[english]{babel}
\usepackage{fouriernc} 

\usepackage{geometry} 
\usepackage{xcolor} 
\usepackage{tikz} 
\usepackage{hyperref} 
\usepackage{amsmath,amsthm,amssymb}
\usepackage{thmtools}
\usepackage[all,cmtip]{xy}

\usepackage[backend=bibtex,maxnames=10,maxalphanames=10, minalphanames=5,style=alphabetic]{biblatex}
\usepackage{tikz-cd}
\usepackage{quiver}
\usetikzlibrary{calc,babel,spath3}
\usepackage{adjustbox}
\tikzset{curve/.style={settings={#1},to path={(\tikztostart)
    .. controls ($(\tikztostart)!\pv{pos}!(\tikztotarget)!\pv{height}!270:(\tikztotarget)$)
    and ($(\tikztostart)!1-\pv{pos}!(\tikztotarget)!\pv{height}!270:(\tikztotarget)$)
    .. (\tikztotarget)\tikztonodes}},
    settings/.code={\tikzset{quiver/.cd,#1}
        \def\pv##1{\pgfkeysvalueof{/tikz/quiver/##1}}},
    quiver/.cd,pos/.initial=0.35,height/.initial=0}
\tikzset{between/.style n args={2}{/tikz/spath/at end path construction={
    \tikzset{spath/split at keep middle={current}{#1}{#2}}
}}}
\tikzset{tail reversed/.code={\pgfsetarrowsstart{tikzcd to}}}
\tikzset{2tail/.code={\pgfsetarrowsstart{Implies[reversed]}}}
\tikzset{2tail reversed/.code={\pgfsetarrowsstart{Implies}}}
\tikzset{no body/.style={/tikz/dash pattern=on 0 off 1mm}}

\numberwithin{equation}{section} 
\numberwithin{figure}{section} 

\newcommand{\K}{\mathcal{K}}
\newcommand{\talg}{T\mbox{-}alg}
\newcommand{\N}{\mathcal{N}}

\newcommand{\ob}{\text{Ob}}

\newcommand{\M}{\mathcal{M}}
\newcommand{\Nn}{\mathcal{N}}
\newcommand{\Cat}{\mathbf{Cat}}

\newcommand{\id}{\text{id}}
\newcommand{\la}{\langle}
\newcommand{\ra}{\rangle}

\usepackage{cleveref}  
\newcommand{\clevertheorem}[3]{
  \newtheorem{#1}[thm]{#2}
  \crefname{#1}{#2}{#3}
}
\theoremstyle{plain} 
\newtheorem{thm}{Theorem}[section]
\crefname{thm}{Theorem}{Theorems}
\newtheorem*{thm*}{Theorem}
\clevertheorem{prop}{Proposition}{Propositions}
\newtheorem*{prop*}{Proposition}
\clevertheorem{lem}{Lemma}{Lemmas}
\clevertheorem{cor}{Corollary}{Corollaries}
\clevertheorem{conj}{Conjecture}{Conjectures}

\theoremstyle{definition} 
\clevertheorem{defi}{Definition}{Definitions}
\clevertheorem{ex}{Example}{Examples}
\clevertheorem{notn}{Notation}{Notations}

\theoremstyle{remark} 
\clevertheorem{rmk}{Remark}{Remarks}

\makeatletter\let\c@equation\c@thm\makeatother
\makeatletter\let\c@figure\c@thm\makeatother

\crefname{figure}{Figure}{Figures}
\crefname{equation}{Diagram}{Diagrams} 
\crefname{eq}{Display}{Displays}
\crefname{eqn}{Display}{Displays}

\hypersetup{
 pdfkeywords={latex starter,template},
 pdfauthor={Niles Johnson},
}


\title{Coherence for pseudo commutative monads}

\author{Diego Manco}

\address{Department of Mathematics, University  of Western Ontario,
       London, ON, N6A 3K7, Canada.}

\email[]{dmanco@uwo.ca} 

\date{\today}

\begin{document}

\begin{abstract}
We prove that the free algebra functor associated to a symmetric, pseudo commutative 2-monad, from the underlying symmetric monoidal 2-category to the 2-category of algebras and pseudo maps over the 2-monad can be enhanced to a multifunctor. Furthermore, we prove that this multifunctor is pseudo symmetric. Our proof implies coherence results for both symmetric and non-symmetric pseudo commutative 2-monads conjectured by Hyland and Power.
\end{abstract}
\maketitle
\tableofcontents
\section{Introduction}
In category theory, the theory of monads \cite{ML98,BW85,S72} was developed  as a way to describe and study algebraic structures in the spirit of universal algebra, with a monad having an associated category of algebras that we want to understand. Several interesting enhancements have been proposed since they appeared. On the one hand, one can consider monads $T\colon \mathcal{V}\to \mathcal{V}$ on a symmetric monoidal closed category $\mathcal{V}$. When the monad $T$ is strong, it can be considered as a monoidal functor in two ways. If these two ways agree, the monad is said to be commutative (equivalently it is a symmetric monoidal monad) \cite{K70,K71,K72}. Under mild conditions on  $\mathcal{V}$ and $T$ \cite[p. 349]{K74}, \cite[p. 419]{K71}, the category of $T$-algebras (and strict maps), $\talg_s$ is symmetric monoidal closed. Thus, in this case, one has internal homs, a tensor product, and as a biproduct of this one has a multicategory structure which will be our focus. On the other hand, one can consider an enrichment in a 2-category which leads to the concept of 2-monads \cite{L02,BKP89,S07}. In addition to algebras, 2-monads also have pseudo algebras, and between these we can have strict maps and pseudo maps, which are of special interest to us. We will call $\talg$ the category of algebras and pseudo maps of a 2-monad $T$. 

By merging the previous points of view one is led to the definition of a strictly commutative 2-monad over a 2-symmetric monoidal category. At least when $T\colon\Cat\to \Cat$, such a monad generates a symmetric monoidal 2-category $\talg_s$ \cite{K74}. However, to include some important examples of 2-monads $T:\Cat\to \Cat$, such as the monad for symmetric monoidal categories, and because usually one wants to focus in $\talg$ as opposed to $\talg_s,$ Hyland and Power \cite{HP02} defined, after \cite{K74}, the concept of a strong, pseudo commutative 2-monad in a monoidal 2-category $T:\K\to\K.$ These are 2-monads which are commutative only up to coherent isomorphisms in a precise sense. Just like Hyland and Power in \cite{HP02}, we will take $\K$ to be a 2-category with finite products to avoid writing any associators, but we claim that what we do also holds for a general monoidal 2-category. Among other things, it is proven in \cite{HP02} that $\talg$ is a $\Cat$-enriched multicategory, and that it is pseudo closed as a  2-category. In the case that $T$ satisfies the further condition of being symmetric, our first result is that the free algebra construction $T\colon \K\to\talg$ preserves multilinear maps. 
\begin{thm}\normalfont{(\Cref{maintheoremchapter1})}
If $T$ is a symmetric, pseudo commutative, strong 2-monad, $T\colon \K\to \talg$ is a multifunctor.
\end{thm}
When $T$ is a symmetric, pseudo commutative, strong 2-monad, $\talg$ is a symmetric $\Cat$-enriched multicategory \cite{HP02}. The free algebra multifunctor $T:\K\to\talg $ is not symmetric since it doesn't preserve the action of the symmetric group on multilinear maps by swapping inputs. However, we prove that it does so up to coherent isomorphisms. Multifunctors with this property are called pseudo symmetric and they were defined by Yau in his study of inverse $K$-theory \cite{Y23}. A coherence result for these was proven by the author \cite{M23}. So far, the only example in the literature of a pseudo symmetric multifunctor is provided by Mandell's inverse $K$-theory \cite{Mandell10,Y23}.  The following is the main theorem in this paper.
\begin{thm}\normalfont{(\Cref{maintheoremchapter2})}
If $T$ is a symmetric pseudo commutative strong 2-monad, $T\colon \K\to\talg$ is a pseudo symmetric multifunctor.
\end{thm}
Our proof implies a coherence result for pseudo commutative monads originally conjectured by Hyland and Power \cite[Theorem 4]{HP02},  even in the absence of symmetry as we explain in Remark \ref{hpcoherence}.  

Our results apply to all of the following. Examples of 2-monads $T\colon\Cat\to\Cat$ that are symmetric pseudo commutative but not commutative include the monad for symmetric monoidal categories, the monad for permutative categories, the monads associated to the symmetric pseudo commutative operads defined in \cite{CG13,GMMO21}, and  also considered in  \cite{Y24},  chaotic $\Cat$-operads \cite{GMMO21}, and $KZ$-monads \cite{L11}, also known as lax idempotent 2-monads, which include monads whose algebras are categories equipped with a given class of colimits (or limits). Examples of 2-monads that are not symmetric but are pseudo commutative include the monad for braided stric monoidal categories, which has two pseudocommutative structures, neither of which are symmetric \cite{CG13}, as well as the operads in \cite[Theorem 4.4]{CG13}. Although the main theorems don't apply to these, our proofs, and in particular \Cref{hpcoherence} do.

In their definition of the multiplicative equivariant $K$-theory multifunctor in \cite{GMMO21} the free algebra multifunctor for certain symmetric, pseudo commutative, strong 2-monads is considered. Our result can be understood as a step towards proving that this version of equivariant $K$-theory is pseudo symmetric and thus, preserves multiplicative structures in the sense of \cite{M23}. On the other hand, another multiplicative $K$-theory multifunctor was defined by Yau in \cite{Y24}, and our result can play a part in proving that these two machines are equivalent. Notice that one could view our result as a coherence theorem for pseudo symmetric monoidal 2-monads, so our result suggests  the use of the underlying multicategorical structure to express coherence results about general lax symmetric monoidal functors or similar gadgets (where one relaxes symmetry). It is desirable to have a graphic calculus for symmetric pseudo commutative 2-monads that also works for pseudo and lax morphisms of symmetric monoidal 2-categories, our work can contribute to this end.

It is worth mentioning the work of Bourke which explores the question of when one can recover the multicategorical structure in $\talg$ from a symmetric monoidal 2-categorical structure. Specializing to the case $T:\Cat\to \Cat$, Bourke proves that the 2-multicategory $\talg$ is skew \cite{BL20}, and so it admits a skew monoidal 2-category structure \cite{BL18}. In the case that $T\colon \Cat \to \Cat$ is accessible, the multicategory structure in $\talg$ can be seen to arise from a symmetric bicategorical structure on $\talg$ \cite{B17,BL20}. 

\textbf{Outline}: In \Cref{Section3.1}, we define symmetric, pseudo commutative, strong 2-monads $T\colon \to \K$ following \cite{HP02} as well as the $\Cat$-enriched multicategory $\talg.$ When  $T$ is symmetric we define $\talg$ as a symmetric $\Cat$-enriched multicateogry.  In \Cref{Section3.2} we extend the free $T$-algebra 2-functor $T:\K\to \K$ to a non-symmetric multifunctor $T:\K\to \talg$ . We finish by proving that this multifunctor $T:\K\to \talg$ is pseudo symmetric in \Cref{Section3.3}.  The full definitions of pseudo symmetric $\Cat$-enriched multifunctor and symmetric $\Cat$-enriched multicategory will be deferred to \Cref{appendix}.
\section{Symmetric pseudo commutative 2-monads}\label{Section3.1}
We will prove a coherence result for symmetric, pseudo commutative, strong 2-monads. We will assume that $\K$ is a 2-category with finite products which we will denote by $\times$, with $1$ denoting the empty product in $\K$.  We do this following \cite{HP02} as a way of justifying our suppression of the associators from our notation. However, we believe what we do to hold as well in any monoidal 2-category. We will denote by $\rho\colon 1\times {-}\to 1_\K$ and $\lambda\colon {-}\times 1\to 1_\K$ the natural isomorphisms comming from the monoidal structure in $\K$ induced by products.
\begin{defi}{\cite{K70}}
Suppose that $T\colon \K\to\K$ is a 2-functor. A \textit{strength} $t$ on $T$ is the data of a (strict) 2-natural transformation (see \cite{JY23}) with source
\begin{center}
    \begin{tikzcd}
        \K\times \K\ar[r,"1_\K\times T"]&\K\times \K\ar[r,"\times"]&\K,
    \end{tikzcd}
\end{center}
and target
\begin{center}
    \begin{tikzcd}
        \K\times \K\ar[r,"\times"]&\K\ar[r,"T"]&\K.
    \end{tikzcd}
\end{center}
The component of $t$ at $(A,B)\in\ob(\K\times \K),$ will be denoted by $t_{A,B}\colon A\times TB\to T(A\times B)$ or  just $t$ when there is no room for confusion. These data are required to satisfy the following axioms:
\begin{itemize}
    \item \textbf{Unity:} the triangle
    \begin{center}
        \begin{tikzcd}
        1\times TA\ar[rd,"\lambda"']\ar[r,"t_{1,A}"]&T(1\times A)\ar[d,"T\lambda"]\\
        &TA
        \end{tikzcd}
    \end{center}
    commutes for all $A\in \ob(\K).$
    \item \textbf{Associativity:} the triangle
    \begin{center}
        \begin{tikzcd}
            A\times B\times TC\ar[rd,"t_{A\times B,C}"']\ar[r,"1_A\times t_{B,C}"]&A\times T(B\times C)\ar[d,"t_{A,B\times C}"]\\
            &T(A\times B\times C)
        \end{tikzcd}
    \end{center}
    commutes for every $A,B\in\ob(\K).$
\end{itemize}
In this case we say that $T\colon \K\to\K$ is \textit{strong} with strength $t.$
\end{defi}
\begin{rmk}Suppose that $(T,\eta,\mu,t)$ is a strong 2-monad. The following notation is introduced in \cite{HP02}. For $n\geq 2,$ $t_i^n$ will denote the natural isomorphism having as component at $(A_1,\dots,A_n)\in\ob(\K^n),$ the 1-cell 
\begin{center}
    \begin{tikzcd}
        A_1\times \cdots \times A_{i-1}\times TA_i\times A_{i+1}\times  \cdots \times A_n\ar[d,"\cong"']\ar[r,"{t_i^n}_{A_1,\dots,A_n}"] &[10pt]T(A_1\cdots\times A_n)\\
        A_1\times\cdots\times A_{i-1}\times A_{i+1}\times\cdots\times A_n\times TA_i\ar[r,"t"']&T(A_1\times \cdots\times A_n\times A_i ).\ar[u,"T\cong"]
    \end{tikzcd}
\end{center}
We will denote ${t_i^n}_{A_1,\dots,A_n}=t_i$ when there is no room for confusion. Notice that $t=t_2^2.$ In \cite{HP02}, $t_1^2$ is also called $t^*.$ We will write our arrows in terms of $t_1^2$ and $t_2^2$ when possible. We notice that the associativity axiom implies that $t_i^n$ can be written in many ways using the $t_i^k$ for $k<n.$ For example, one can prove by induction that the ${t_{i}^n}_{A_1,\dots, A_n}$ can be written as
\begin{center}
\adjustbox{scale=0.85}{
    \begin{tikzcd}
        A_1\times \cdots \times A_{i-1}\times TA_i\times A_{i+1}\times  \cdots \times A_n \ar[r,"1\times t_1"] &[10pt]A_1\times \cdots\times A_{i-1}\times T(A_i\times \cdots\times A_n)\ar[r,"t_2"]&T(A_1\times\cdots\times A_n),
    \end{tikzcd}}
\end{center}
or as
\begin{center}
\adjustbox{scale=0.85}{
    \begin{tikzcd}
        A_1\times \cdots \times A_{i-1}\times TA_i\times A_{i+1}\times  \cdots \times A_n \ar[r,"t_2\times 1"]&[10pt]T(A_1\times \cdots\times A_i)\times A_{i+1}\cdots\times A_n\ar[r,"t_1"]& T(A_1\times\cdots\times A_n).
    \end{tikzcd}}
\end{center}
\end{rmk}
\begin{defi}\label{strong}
Let $(T,\eta,\mu,t)$ be a 2-monad with $T\colon\K\to\K$. That is, $T$ is a strict 2-functor and $\eta,\mu$ are strict 2-natural transformations satisfying the usual triangle identities (see \cite{JY23}). We say that $(T,\eta,\mu,t)$ is a \textit{strong 2-monad} with strength $t,$ if $T\colon \K\to\K$ is strong with strength $t$ as a 2-functor and $\eta, \mu$ and $t$ are compatible in the sense that, for every $A,B\in\ob(\K),$ the triangle
\begin{center}
    \begin{tikzcd}
        A\times B\ar[r,"1\times \eta"]\ar[rd,"\eta"']&A\times TB\ar[d,"t"]\\
        &T(A\times B)
    \end{tikzcd}
\end{center}
commutes, as well as the square
\begin{center}
    \begin{tikzcd}
        A\times T^2\ar[d,"t"'] B\ar[rr,"1\times \mu"]&&A\times TB\ar[d,"t"]\\
        T(A\times TB)\ar[r,"Tt"']&T^2(A\times B)\ar[r,"\mu"']&T(A\times B).
    \end{tikzcd}
\end{center}
\end{defi}
\begin{defi}{\cite{K70}}
    A strong 2-monad $(T,\eta,\mu,t)$ is called commutative when the following diagram commutes for every $A,B\in\ob(\K):$
    \begin{center}
        \begin{tikzcd}
            TA\times TB\ar[d,"t_2"']\ar[r,"t_1"]&T(A\times TB)\ar[r,"Tt_2"]&T^2(A\times B)\ar[d,"\mu"]\\
            T(TA\times B)\ar[r,"Tt_1"']&T^2(A\times B)\ar[r,"\mu"']&T(A\times B).
        \end{tikzcd}
    \end{center}
    
\end{defi}
\begin{rmk}
Suppose that $(T,\eta,\mu,t)$ is a strong 2-monad. Then, $T$ can be regarded as a monoidal 2-functor in two different ways. In each case, the unitary component is given by $\eta_1\colon 1\to T1.$ The binary components are given by the two 1-cells that form the boundary of the previous diagram. For each of these ways of seeing $T$ as a monoidal 2-functor, $\eta$ is a monoidal 2-natural transformation. It is proven in \cite{K70} that $T$ is commutative if and only if $T$ is a monoidal 2-monad (i.e., $\mu$ is a monoidal 2-natural transformation). 

There are a lot of examples of strong 2-monads which are non-commutative, but that are commutative up to coherent natural isomorphism, these are called pseudo commutative monads and we will defined them next.  The examples include the 2-monads $T\colon \Cat\to \Cat$ given by the free construction for symmetric stric monoidal categories,  symmetric monoidal categories, categories with finite products, categories with finite coproducts, etc. A longer list is included in \cite{HP02}. More examples come from pseudo commutative operads as defined by Corner and Gurski \cite{CG13}, and featured in \cite{GMMO21,Y24}. These are operads whose associated monads are pseudo commutative. Guillou, Merling, May and Osorno \cite{GMMO21} prove that chaotic operads are pseudo commutative.
\end{rmk}
\begin{defi}\label{pseudocomm}{\cite[Def. 5]{HP02}}
    A strong 2-monad $(T,\eta,\mu,t)$ is called \textit{pseudo-commutative} with \textit{pseudocommutativity} $\Gamma$ if there exists an invertible modification with components, for $A,B\in\ob(\K):$
    \begin{center}
        \begin{tikzcd}
            TA\times TB\ar[d,"t_2"']\ar[r,"t_1"]&T(A\times TB)\ar[r,"Tt_2"]&|[alias=D]|T^2(A\times B)\ar[d,"\mu"]\\
            |[alias=R]|T(TA\times B)\ar[r,"Tt_1"]&T^2(A\times B)\ar[r,"\mu"]&T(A\times B),
        \arrow[Rightarrow,from=D,to=R,"\Gamma_{A,B}",shorten >=15mm,shorten <=13mm]
        \end{tikzcd}
    \end{center}
    such that the following axioms are satisfied for all $A,B,C$ objects of $\K$. We will write $\Gamma$ instead of $\Gamma_{A,B}$ when $A$ and $B$ are clear from the context.
    \begin{enumerate}
        \item\label{axiomstrength1} $\Gamma_{A\times B,C}\circ({t_2}_{A,B}\times 1_{TC})={t_2}_{A,B\times C}\circ(1_A\times \Gamma_{B,C}),$ i.e., the following pasting diagram equality holds:
\begin{center}
\begin{tikzcd}
	{A\times TB\times TC} & {T(A\times B)\times TC} & {T(A\times B\times C)} \\[-20pt]
	& {||} \\[-20pt]
	{A\times TB\times TC} & {A\times T(B\times C)} & {T(A\times B\times C).}
	\arrow["{t_2\times 1}", from=1-1, to=1-2]
	\arrow[""{name=0, anchor=center, inner sep=0}, curve={height=-12pt}, from=1-2, to=1-3]
	\arrow[""{name=1, anchor=center, inner sep=0}, curve={height=12pt}, from=1-2, to=1-3]
	\arrow[""{name=2, anchor=center, inner sep=0}, curve={height=-12pt}, from=3-1, to=3-2]
	\arrow[""{name=3, anchor=center, inner sep=0}, curve={height=12pt}, from=3-1, to=3-2]
	\arrow["{t_2}", from=3-2, to=3-3]
	\arrow["\Gamma", between={0.2}{0.8}, Rightarrow, from=0, to=1]
	\arrow["{1\times \Gamma}", between={0.2}{0.8}, Rightarrow, from=2, to=3]
\end{tikzcd}
\end{center}
        \item\label{axiomstrength2} $\Gamma_{A,B\times C}\circ(1_{TA}\times {t_2}_{B,C})=\Gamma_{A\times B,C}\circ({t_1}_{A,B}\times 1_{TC}),$ i.e., the following equality holds:
        \begin{center}
\begin{tikzcd}
	{TA\times B\times TC} & {TA\times T(B\times C)} & {T(A\times B\times C)} \\[-20pt]
	& {||} \\[-20pt]
	{TA\times B\times TC} & {T(A\times B)\times TC} & {T(A\times B\times C).}
	\arrow["{1\times t_2}", from=1-1, to=1-2]
	\arrow[""{name=0, anchor=center, inner sep=0}, curve={height=-12pt}, from=1-2, to=1-3]
	\arrow[""{name=1, anchor=center, inner sep=0}, curve={height=12pt}, from=1-2, to=1-3]
	\arrow["{t_1\times 1}", from=3-1, to=3-2]
	\arrow[""{name=2, anchor=center, inner sep=0}, curve={height=-12pt}, from=3-2, to=3-3]
	\arrow[""{name=3, anchor=center, inner sep=0}, curve={height=12pt}, from=3-2, to=3-3]
	\arrow["\Gamma", between={0.2}{0.8}, Rightarrow, from=0, to=1]
	\arrow["\Gamma", between={0.2}{0.8}, Rightarrow, from=2, to=3]
\end{tikzcd}
        \end{center}
        \item\label{axiomstrength3} $\Gamma_{A,B\times C}\circ(1_{TA}\times {t_1}_{B,C})={t_1}_{A\times B, C}\circ( \Gamma_{A,B}\times 1_C),$ i.e., the following whiskering equality holds:
        \begin{center}
\begin{tikzcd}
	{TA\times TB\times C} &[10pt] {TA\times T(B\times C)} & {T(A\times B\times C)} \\[-20pt]
	& {||} \\[-20pt]
	{TA\times TB\times C} & {T(A\times B)\times C} & {T(A\times B\times C)}
	\arrow["{1\times t_1}", from=1-1, to=1-2]
	\arrow[""{name=0, anchor=center, inner sep=0}, curve={height=-12pt}, from=1-2, to=1-3]
	\arrow[""{name=1, anchor=center, inner sep=0}, curve={height=12pt}, from=1-2, to=1-3]
	\arrow[""{name=2, anchor=center, inner sep=0}, curve={height=-12pt}, from=3-1, to=3-2]
	\arrow[""{name=3, anchor=center, inner sep=0}, curve={height=12pt}, from=3-1, to=3-2]
	\arrow["{t_1}", from=3-2, to=3-3]
	\arrow["\Gamma", between={0.2}{0.8}, Rightarrow, from=0, to=1]
	\arrow["{\Gamma\times 1}", between={0.2}{0.8}, Rightarrow, from=2, to=3]
\end{tikzcd}
        \end{center}
        \item\label{axiometa1} $\Gamma_{A,B}\circ (\eta_A\times 1_{TB})$ is an identity 2-cell. That is, the following whiskering is an identity:
        \begin{center}
\begin{tikzcd}
	{A\times TB} & {TA\times TB} & {T(A\times B)}
	\arrow["{\eta\times 1}", from=1-1, to=1-2]
	\arrow[""{name=0, anchor=center, inner sep=0}, curve={height=-12pt}, from=1-2, to=1-3]
	\arrow[""{name=1, anchor=center, inner sep=0}, curve={height=12pt}, from=1-2, to=1-3]
	\arrow["\Gamma", between={0.2}{0.8}, Rightarrow, from=0, to=1]
\end{tikzcd}
        \end{center}
        \item\label{axiometa2} $\Gamma_{A,B}\circ (1_{TA}\times \eta_{B})$ is an identity 2-cell, that is, the following whiskering is an identity:
        \begin{center}
\begin{tikzcd}
	{TA\times B} & {TA\times TB } & {T(A\times B).}
	\arrow["{1\times \eta}", from=1-1, to=1-2]
	\arrow[""{name=0, anchor=center, inner sep=0}, curve={height=-12pt}, from=1-2, to=1-3]
	\arrow[""{name=1, anchor=center, inner sep=0}, curve={height=12pt}, from=1-2, to=1-3]
	\arrow["\Gamma", between={0.2}{0.8}, Rightarrow, from=0, to=1]
\end{tikzcd}
        \end{center}
        \item\label{axiommu2gamma} The whiskering
        \begin{center}
\begin{tikzcd}
	{T^2A\times TB} & {TA\times TB} & {T(A\times B)}
	\arrow["{\mu\times 1}", from=1-1, to=1-2]
	\arrow[""{name=0, anchor=center, inner sep=0}, curve={height=-12pt}, from=1-2, to=1-3]
	\arrow[""{name=1, anchor=center, inner sep=0}, curve={height=12pt}, from=1-2, to=1-3]
	\arrow["\Gamma", between={0.2}{0.8}, Rightarrow, from=0, to=1]
\end{tikzcd}
        \end{center}
        is equal to the pasting 
        \begin{center}
        \begin{equation}\label{pseudomu1}
            \begin{tikzcd}
                T^2A \times TB\ar[d,"t_2"']\ar[r,"t_1"]&[10pt]|[alias=D2]|T(TA\times TB)\ar[d,"Tt_2"]\ar[r,"Tt_1"]&[-10pt]T^2(A\times TB)\ar[r,"T^2 t_2"]&|[alias=D1]|T^3(A\times B)\ar[d,"T\mu"]\\
                T(T^2 A \times B)\ar[d,"Tt_1"']&|[alias=R1]|T^2(TA\times B)\ar[d,"\mu"]\ar[r,"T^2 t_1"]&T^3(A\times B)\ar[d,"\mu"]\ar[r,"T\mu"]&T^2(A\times B)\ar[d,"\mu"]\\
                |[alias=R2]|T^2(TA\times B)\ar[r,"\mu"']&T(TA\times B)\ar[r,"Tt_1"']&T^2(A\times B)\ar[r,"\mu"']&T(A\times B).
                \arrow[Rightarrow,from=D1,to=R1,"T\Gamma"',shorten=15mm]
                \arrow[Rightarrow,from=D2,to=R2,"\Gamma"',shorten=12mm]
            \end{tikzcd}
            \end{equation}
        \end{center}
        \item\label{axiommu1gamma} The whiskering
        \begin{center}
\begin{tikzcd}
	{TA\times T^2B} & {TA\times TB} & {T(A\times B)}
	\arrow["{1\times \mu}", from=1-1, to=1-2]
	\arrow[""{name=0, anchor=center, inner sep=0}, curve={height=-12pt}, from=1-2, to=1-3]
	\arrow[""{name=1, anchor=center, inner sep=0}, curve={height=12pt}, from=1-2, to=1-3]
	\arrow["\Gamma", between={0.2}{0.8}, Rightarrow, from=0, to=1]
\end{tikzcd}
        \end{center}
        is equal to the pasting 
        \begin{center}
        \begin{equation}\label{pseudomu2}
            \begin{tikzcd}
                TA\times T^2B\ar[d,"t_2"']\ar[r,"t_1"]&T(A\times T^2B)\ar[r,"Tt_2"]&|[alias=D1]|T^2(A\times TB)\ar[d,"\mu"]\\
                |[alias=R1]|T(TA\times TB)\ar[d,"Tt_2"']\ar[r,"Tt_1"']&|[alias=D2]|T^2(A\times TB)\ar[d,"T^2 t_2"]\ar[r,"\mu"']&T(A\times TB)\ar[d,"T t_2"]\\
                T^2(TA\times B)\ar[d,"T^2t_1"']&T^3(A\times B)\ar[d,"T\mu"]\ar[r,"\mu"]&T^2(A\times B)\ar[d,"\mu"]\\
                |[alias=R2]|T^3(A\times B)\ar[r,"T\mu"']&T^2(A\times B)\ar[r,"\mu"']&T(A\times B).
                \arrow[Rightarrow,from=D1,to=R1,"\Gamma"',shorten=15mm]
                \arrow[Rightarrow,from=D2,to=R2,"T\Gamma",shorten=12mm]
            \end{tikzcd}
        \end{equation}
        \end{center}
    \end{enumerate}
\end{defi}
\begin{rmk}
The fact that the source and target of the equal whiskering and pasting diagrams in the previous list of axioms are the same follows from the definition of 2-strong monad. In other words, the pseudo commutativity axioms don't introduce new relations among 1-cells.

The axioms are not independent as it is noted in \cite{HP02}. Any two of the axioms (\ref{axiomstrength1}), (\ref{axiomstrength2}), and (\ref{axiomstrength3}) implies the other, and symmetry (\Cref{Tsymmetric}) introduces further redundancies. 

A modification is more than a mere collection of 2-cells (see \cite{JY22}). For $\Gamma$ to be a modification we need that given $f\colon A\to A'$ and $g\colon B\to B'$ in $\K,$ the following equality of pasting diagrams holds:
\begin{center}
\begin{equation}\label{strongdiagram}
\adjustbox{scale=0.9}{
\begin{tikzcd}
    &[-40pt]TA\times TB\ar[dl,"t_1"']\ar[rr,"Tf\times Tg"]&[-40pt]&[-50pt]TA'\times TB'\ar[dr,"t_2"]\ar[dl,"t_1"']&[-50pt]&[-30pt]&[-30pt]&[-50pt]TA\times TB\ar[dl,"t_1"']\ar[dr,"t_2"]\ar[rr,"Tf\times Tg"]&[-40pt]&[-40pt]TA'\times TB'\ar[rd,"t_2"]&[-40pt]\\
    T(A\times TB)\ar[dd,"Tt_2"']\ar[rr,"T(f\times Tg)"]&&T(A'\times TB')\ar[dd,"Tt_2"{swap,name=D1}]&&T(TA'\times B')\ar[dd,"Tt_1"{name=R1}]&&T(A\times TB)\ar[dd,"Tt_2"{swap,name=D2}]&&T(TA\times B)\ar[rr,"T(Tf\times g)"]\ar[dd,"Tt_1"{name=R2}]&&T(TA'\times B')\ar[dd,"Tt_1"]\\[-10pt]
    &&&&&=\\[-10pt]
    T^2(A\times B)\ar[rd,"\mu"']\ar[rr,"T^2(f\times g)"]&&T^2(A'\times B')\ar[rd,"\mu"']&&T^2(A'\times B')\ar[dl,"\mu"]&&T^2(A\times B)\ar[dr,"\mu"]&&T^2(A\times B)\ar[rr,"T^2(f\times g)"]\ar[dl,"\mu"]&&T^2(A'\times B')\ar[dl,"\mu"]\\
    &T(A\times B)\ar[rr,"T(f\times g)"']&&T(A'\times B')&&&&T(A\times B)\ar[rr,"T(f\times g)"']&&T(A'\times B').
    \arrow[Rightarrow,from=D1,to=R1,"\Gamma",shorten <=5mm, shorten >=5mm]
        \arrow[Rightarrow,from=D2,to=R2,"\Gamma",shorten <=5mm, shorten >=5mm]
\end{tikzcd}}
\end{equation}
\end{center}
\end{rmk}
\begin{ex} We will introduce an example so that our reader has something to compare her intuitions with. Let $T\colon\Cat\to\Cat$ be the monad whose 2-category of algebras is the category of symmetric monoidal categories. For a small category $\mathcal{C},$ $\ob (T\mathcal{C})$ consists of finite sequences of elements of $\ob(C),$ including the empty sequence. The arrows of this category are freely generated by symmetries of the form $\mathcal{C}$ $$(a_{\sigma(1)}a_{\sigma(2)}\ldots a_{\sigma(n))}\xrightarrow{\sigma} (a_1a_2\ldots a_n)$$
for $a_1,\ldots, a_n\in \ob(\mathcal{C}),$ and $\sigma\in \Sigma_n,$ together with maps of the form $$(f_1,\dots, f_n)\colon (a_1\dots a_n)\to (b_1\dots b_n)$$
for $f_i\colon a_i\to b_i$ an arrow of $\mathcal{C}$ for $1\leq i\leq n.$ Thus, every arrow in $T\mathcal{C} $ is of the form $$(a_{\sigma(1)}\cdots a_\sigma(n))\xrightarrow{\sigma}(a_1\ldots a_n)\xrightarrow{(f_1,\dots, f_n)}(b_1\ldots b_n)\xrightarrow{\tau^{-1}}(b_{\tau(1)}\ldots b_{\tau(n)}),$$
for $n\in \mathbb{N},$ $\sigma,\tau \in \Sigma_n,$ and $f_i$ as before. We give $T\mathcal{C}$ the structure of a symmetric monoidal category with symmetric monoidal product concatenation, and unit the empty sequence. We make $T$ into a 2-monad by defining $\eta\colon\mathcal{C}\to T\mathcal{C}$ as sending $a$ to $(a)$, and $\mu\colon T^2\mathcal{C}\to TC$ by erasing parentheses. This 2-monad is strong with respect to the 2-natural transformation $t:\mathcal{A}\to T\mathcal{B}\to T(\mathcal{A}\times \mathcal{B})$ which sends $(a,b_1\ldots b_n)$ to $((a,b_1)(a,b_2)\ldots (a,b_n)).$ Notice that 
\Cref{strongdiagram} does not commute, and so, this monad is not commutative. One of the maps $T\mathcal{A}\times T\mathcal{B}\to T(\mathcal{A}\times \mathcal{B})$ sends the pair $(a_1\ldots a_n,b_1\dots b_m)$ to the sequence $$((a_1,b_1)(a_1,b_2)\ldots (a_1,b_m) (a_2,b_1)(a_2,b_2)\ldots (a_n,b_m) ),$$ while the other maps it to the sequence
$$((a_1,b_1)(a_2,b_1)\ldots (a_n,b_1),(a_1,b_2),(a_2,b_2),\ldots ,(a_n,b_m)).$$
Clearly, the two maps are not equal, but they commute up to a natural isomorphism given by a permutation. Furthermore the two permutations that exchange between one map and the other are inverses of each other, which means that this monad is an example of a symmetric pseudo commutative 2-monad, see \Cref{Tsymmetric}.  
\end{ex}
Following Blackwell, Kelly and Power \cite{BKP02}, we now define, for any 2-monad $T\colon \K\to\K$,  the 2-category $T\mbox{-}Alg$ of $T$-algebras and pseudo morphisms.
\begin{defi}\label{talg}{\cite[Def. 1.2]{BKP02}}
Let $(T,\eta,\mu )$ be a 2-monad. The 2-category $T\mbox{-}Alg$ has strict $T$-algebras as 0-cells. A 1-cell $f$ between $T$-algebras $(A,a\colon TA\to A)$ and $(B,b\colon TB\to B),$ also called a strong morphism of $T$-algebras in \cite{JY22}, consists of a 1-cell $f\colon A\to B$ in $\K,$ together with an invertible 2-cell
\begin{center}
    \begin{tikzcd}
        TA\ar[r,"Tf"]\ar[d,"a"]&|[alias=D1]|TB\ar[d,"b"]\\
        |[alias=R1]|A\ar[r,"f"']&B,
    \arrow[Rightarrow,from=D1,to=R1,"\overline{f}"',shorten=4mm]
    \end{tikzcd}
\end{center}
subject to the following axioms. 
\begin{enumerate}
    \item The equality of pasting diagrams
    \begin{center}
        \begin{tikzcd}
            T^2A\ar[d,"\mu"']\ar[r,"T^2f"]&T^2B\ar[d,"\mu"]&[-20pt]&[-20pt]T^2A\ar[d,"Ta"']\ar[r,"T^2f"]&|[alias=D2]|T^2B\ar[d,"Tb"]\\
            TA\ar[d,"a"']\ar[r,"Tf"]&|[alias=D1]|TB\ar[d,"b"]&=&|[alias=R2]|TA\ar[d,"a"']\ar[r,"Tf"]&|[alias=D3]|TB\ar[d,"b"]\\
            |[alias=R1]|A\ar[r,"f"']&B&&|[alias=R3]|A\ar[r,"f"']&B
            \arrow[Rightarrow,from=D1,to=R1,"\overline{f}"',shorten=4mm]
            \arrow[Rightarrow,from=D2,to=R2,"T\overline{f}"',shorten=4mm]
            \arrow[Rightarrow,from=D3,to=R3,"\overline{f}"',shorten=4mm]
        \end{tikzcd}
    \end{center}
    holds.
    \item The following pasting diagram equals the identity of $f\colon A\to B:$
    \begin{center}
        \begin{tikzcd}
            A\ar[d,"\eta"']\ar[r,"f"]&B\ar[d,"\eta"]\\
            TA\ar[r,"Tf"]\ar[d,"a"']&|[alias=D1]|Tb\ar[d,"b"]\\
            |[alias=R1]|A\ar[r,"f"']&B.
            \arrow[Rightarrow,from=D1,to=R1,"\overline{f}"',shorten=4mm]
        \end{tikzcd}
    \end{center}
\end{enumerate}
A 2-cell in $T\mbox{-}Alg$ between 1-cells $(f,\bar{f}),(g,\bar{g})\colon A\to B$ is a 2-cell $\alpha\colon f\to g$ in $\K$ such that the following diagram commmutes:
\begin{center}
    \begin{tikzcd}
        TA\ar[dd,"a"']\ar[r,bend left=20,"Tf"name=D1]\ar[r,bend right=20,"Tg"{swap,name=R1}]&[15pt]|[alias=D2]|TB\ar[dd,"b"]&[-15pt]&[-15pt]TA\ar[dd,"a"']\ar[r,bend left=20,"Tf"]&[15pt]|[alias=D3]|TB\ar[dd,"b"]\\[-10pt]
        &&=&&\\[-10pt]
        |[alias=R2]|A\ar[r,bend right=20,"g"']&B&&|[alias=R3]|A\ar[r,bend left=20,"f"name=D4]\ar[r,bend right=20,"g"{swap,name=R4}]&B.
        \arrow[Rightarrow,from=D1,to=R1,"T\alpha"',shorten <=1mm,shorten >=1mm]
        \arrow[Rightarrow,from=D2,to=R2,"\overline{g}",shorten <=9mm,shorten >=3mm]
        \arrow[Rightarrow,near start,from=D3,to=R3,"\overline{f}"',shorten <=2mm,shorten >=10mm]
        \arrow[Rightarrow,from=D4,to=R4,"\alpha"',shorten <=1mm,shorten >=1mm]
    \end{tikzcd}
\end{center}
\end{defi}
Hyland and Power \cite{HP02} extend  Blackwell, Kelly and Power's 2-categorical construction to provide a non symmetric $\Cat$-multicategory whose underlying 2-category is $T\mbox{-}Alg.$   When $\K=\Cat$ and $T$ is accesible, Bourke proves \cite{B17,BL20} that the $\Cat$-multicategory structure can be seen to arise from a monoidal bicategory structure on $\talg$. An example of this is the monad for permutative categories, and the corresponding symmetric monoidal 2-category structure on permutative categories is explicitly worked out in \cite{GJO22}.  Guillou, May, Merling and Osorno \cite{GMMO21} specialize Hyland and Power's definition to define a multicategory $\mathcal{O}\mbox{-}Alg$ for $\mathcal{O}$ a pseudo commutative operad. We refer the reader to the \Cref{appendix} for the definition of $\Cat$-multicategory.

To be able to define the multicategory $T\mbox{-}Alg,$ we first need to prove a coherence result.
\begin{defi}\label{partition}
Suppose $(T,\eta,\mu, t)$ is a pseudo-commutative, strong 2-monad, $n\geq 2$ and $1\leq i<j\leq n$. We define a modification from $\mu\circ Ttj\circ t_i$ to $\mu\circ Tt_i\circ t_j$ as follows. Suppose   $A_1,\dots,A_n$ objects of $\K,$ we define the component 2-cell of our modificiation in
\begin{center}
    \begin{tikzcd}
        \K\left(A_1\times \cdots \times A_{i-1}\times TA_i\times A_{i+1}\times\cdots\times A_{j-1}\times  TA_j\times A_{j+1}\times \cdots\times A_n, T(A_1\times\cdots\times A_n)\right)
    \end{tikzcd}
\end{center}
in the following way. In principle there are various ways of doing this. Consider a partition $K$ of the symbols $A_1,\dots,TA_i,\dots,TA_j,\dots,A_n$ into 4 subsets $K_1,K_2,K_3,K_4$ obtained by placing 3 bars in between symbols such that $K_2$ contains $TA_i,$ and $K_3$ contains $TA_j.$ We will represent $K$ in the following way:
$$ \underbrace{\cdots\times \cdots}_{K_1}\mid \underbrace{\cdots\times TA_i\times \cdots}_{K_2}\mid \underbrace{\cdots \times TA_j\times \cdots}_{K_3}\mid\underbrace{ \cdots\times \cdots}_{K_4}.$$
For such a partition $K,$ we can define the 2-cell $\Gamma_{i,j}^K$ as the whiskering
\begin{center}
    \begin{tikzcd}
        A_1\times \cdots \times TA_i\times \cdots\times TA_j\times \cdots\times A_n\ar[d,"="]\\[-10pt]
        \underbrace{\cdots\times \cdots}_{K_1}\mid \underbrace{\cdots\times TA_i\times \cdots}_{K_2}\mid \underbrace{\cdots \times TA_j\times \cdots}_{K_3}\mid\underbrace{ \cdots\times \cdots}_{K_4}\ar[d,"1\times t_{i-|K_1|}^{|K_2|}\times t_{j-|K_1|-|K_2|}^{|K_3|}\times 1"]\\[-10pt]
        \cdots\times \cdots \times T(\cdots\times A_i\times \cdots)\times T(\cdots\times A_j\times \cdots )\times \cdots\times \cdots\ar[d,bend right=50,"{\color{white}a}"{swap,name=D1}]\ar[d,bend left=50,"{\color{white}a}"name=R1]\\
        \underbrace{\cdots\times\cdots  }_{K_1}\times T(\cdots\times A_i\times\cdots\times A_j\times\cdots)\times \underbrace{\cdots\times\cdots}_{K_4}\ar[d,"t_{|K_1|+1}^n"]\\[-10pt]
        T(A_1\times\cdots\times A_n).
        \arrow[Rightarrow,from=D1,to=R1,"1\times \Gamma\times 1",shorten=2mm]
    \end{tikzcd}
\end{center}
\end{defi}
\begin{thm}\label{ij}{\cite[Thm. 5]{HP02}}
Suppose $(T,\eta,\mu,t,\Gamma)$ is a pseudo commutative, strong 2-monad. The three strength axioms imply that given $n\geq 2,$ and $1\leq i<j\leq n,$ any two partitions $K$ and $K'$ as in \Cref{partition} induce the same 2-cell. That is,
$$\Gamma_{i,j}^K=\Gamma_{i,j}^{K'}.$$
\end{thm}
\begin{proof}
Let $K$ be a partition of $A_1,\dots,TA_i,\dots,TA_j,\dots, A_n$ as in \Cref{partition}. The following hold:
\begin{itemize}
    \item[(i)] If $K_1$ ends by $A_p,$ for some $p<i$,   and $K'$ is obtained from $K$ by moving the first bar one spot to the left, then $\Gamma_{i,j}^K=\Gamma_{i,j}^{K'}$ by (\ref{axiomstrength1}) in \Cref{pseudocomm}.
     \item[(ii)] If $K_2$ ends by $A_p$, for some $p$ such that $i<p<j,$ and $K'$ is obtained from $K$ by moving the second bar one spot to the left, then $\Gamma_{i,j}^K=\Gamma_{i,j}^{K'}$ by (\ref{axiomstrength2}) in \Cref{pseudocomm}.
    \item[(iii)] If $K_3$ ends by $A_p,$ for some $p>j,$  and $K'$ is obtained from $K$ by moving the third bar one spot to the left, then $\Gamma_{i,j}^K=\Gamma_{i,j}^{K'}$ by (\ref{axiomstrength3}) in \Cref{pseudocomm}.
    \end{itemize}
    Finally, let $K'$ be any partition as in \Cref{partition}, then $K'$  can be obtained from the partition
    $$K=A_1\times\cdots\times A_{i-1}\mid \times TA_i\times \cdots \times \mid TA_j\times \cdots\times A_n\mid$$
    by making some number of moves (i), (ii) and (iii), and so  $\Gamma_{i,j}^K=\Gamma_{i,j}^{K'}.$
\end{proof}
\begin{defi}
Let $(T,\eta,\mu,t,\Gamma)$ be a pseudo commutative, strong 2-monad, $n\geq 2,$ $1\leq i<j\leq n$ and $A_1,\dots, A_n$ objects of $\K$ we define the unique 2-cell in the previous theorem as $\Gamma_{i,j}.$ That is, if $K$ is any partition as in \Cref{partition}, then $\Gamma_{i,j}=\Gamma_{i,j}^K$
\end{defi}
\begin{rmk}
    To save some space in the following definitions we will denote the product $A_1\times \cdots\times A_{i-1}$ as $A_{<i}.$ When considering a product $A_1\times \cdots\times A_n$ we will also write $A_{>i}=A_{i+1}\times \cdots\times A_n.$ Notice also that the 2-cell $\Gamma_{i,j}$ defined in the previous theorem fits in the following diagram by the $\mu$ axiom for strong monads in \Cref{strong}:
    \begin{center}
       \adjustbox{scale=0.9}{
\begin{tikzcd}
	&[-10pt] {T(A_{<j}\times TA_j\times A_{>j})} & {T^2(A_1\times \cdots \times A_n)} \\
	{A_{<i}\times TA_i\times \cdots\times TA_j\times A_{>j}} & {T(A_{<i}\times TA_i\times A_{>i})} &[-10pt] {T^2(A_1\times \cdots\times A_n)} & {T(A_1\times\cdots\times A_n).}
	\arrow[""{name=0, anchor=center, inner sep=0}, "{Tt_j}", from=1-2, to=1-3]
	\arrow["\mu", from=1-3, to=2-4]
	\arrow["{t_i}", from=2-1, to=1-2]
	\arrow["{t_j}"', from=2-1, to=2-2]
	\arrow[""{name=1, anchor=center, inner sep=0}, "{Tt_i}"', from=2-2, to=2-3]
	\arrow["\mu"', from=2-3, to=2-4]
	\arrow["{\Gamma_{i,j}}", between={0.2}{0.8}, Rightarrow, from=0, to=1]
\end{tikzcd}}
    \end{center}
\end{rmk}
Next, we define the $\Cat$-multicategory $\talg$, whose underlying 2-category is $\talg$ from \Cref{talg}. In \Cref{onecelltalg} we define the 2-cells of $\talg,$ in \Cref{twocelltalg} we define the 2-cells in $\talg,$ and in \Cref{compositiontalg} we define the composition in $\talg.$
\begin{defi}\label{onecelltalg}{\cite[Def. 10]{HP02}} Let $(T,\eta,\mu,t,\Gamma)$ be a pseudo commutative, strong 2-monad. The $n$-ary 1-cells of the $\Cat$-multicategory $T\mbox{-}Alg$ are defined as follows. When $n=0,$ and $B$ is a $T$-algebra, we define the category $T\mbox{-}Alg(-;B)$ as $\K(1,B)$.

Suppose that $(A_i,a_i\colon TA_i\to A_i)$ for $1\leq i \leq n$ and $(B,b\colon TB\to B)$ are $T$-algebras. An $n$-ary 1-cell of $\talg$, $\la A_1 \times \cdots \times A_n\ra\to B$ is the data of a 1-cell $h\colon A_1\times \cdots\times A_n\to  B$ in $\K ,$ together with 2-cells $h_i$ for $1\leq i\leq n$ fitting in the square:
    \begin{center}
    \adjustbox{scale=0.9}{
        \begin{tikzcd}
            A_{<i}\times TA_i\times A_{>i}\ar[d,"1\times a_i\times 1"']\ar[r,"t_i"]&T(A_1\times \cdots \times A_n)\ar[r,"Th"]&|[alias=D1]|TB\ar[d,"b"]\\
            |[alias=R1]|A_1\times\cdots\times A_n\ar[rr,"h"']&&B.
            \arrow[Rightarrow,from=D1,to=R1,"h_i",shorten=22mm]
        \end{tikzcd}}
    \end{center}
    These data have to satisfy the following axioms.
    
        $\bullet$ $\eta$ axiom: The following pasting diagram is the identity of $h\colon A_1\times\cdots\times A_n\to B.$
        \begin{center}
        \begin{equation}
        \adjustbox{scale=0.9}{
            \begin{tikzcd}
            &A_{<i}\times A_i\times A_{>i}\ar[dl,"1\times \eta\times 1"']\ar[d,"\eta"]\ar[r,"h"]&B\ar[d,"\eta"]\\
            A_{<i}\times TA_i\times A_{>i}\ar[r,"t_i"]\ar[d,"1\times a_i\times 1"']&T(A_1\times \cdots\times A_n)\ar[r,"Th"']&|[alias=D1]|TB\ar[d,"b"]\\
            |[alias=R1]|A_1\times \cdots\times A_n\ar[rr,"h"']&&B.
            \arrow[Rightarrow,from=D1,to=R1,"h_i",shorten=20mm]
        \end{tikzcd}}
        \end{equation}
        \end{center}
    $\bullet$ $\mu$ axiom: The pasting diagrams
    \begin{center}
    \begin{equation}\label{mu1}
        \adjustbox{scale=0.9}{\begin{tikzcd}
            A_{<i}\times T^2A_i\times A_{>i}\ar[d,"1\times\mu\times 1"']\ar[r,"t_i"]&T(A_{<i}\times TA_i\times A_{>i})\ar[r,"Tt_i"]&T^2(A_1\times \cdots\times A_n)\ar[d,"\mu"]\ar[r,"T^2h"]&T^2B\ar[d,"\mu"]\\
            A_{<i}\times TA_i\times A_{<i}\ar[rr,"t_i"]\ar[d,"1\times a_i\times 1"']&&|[alias=D1]|T(A_1\times\cdots\times A_n)\ar[r,"Th"]&TB\ar[d,"b"]\\
            |[alias=D2]|A_1\times\cdots\times  A_n\ar[rrr,"h"']&&&B
            \arrow[Rightarrow,from=D1,to=R1,"h_i",shorten=13mm]
        \end{tikzcd}}
        \end{equation}
    \end{center}
    and,
    \begin{center}
    \begin{equation}\label{mu2}
        \adjustbox{scale=0.9}{\begin{tikzcd}
            A_{<i}\times T^2A_i\times A_{>i}\ar[d,"1\times Ta_i\times 1"']\ar[r,"t_i"]&T(A_{<i}\times TA_i\times A_{>i})\ar[d,"T(1\times a_i\times 1)"']\ar[r,"Tt_i"]&T^2(A_1\times\cdots\times A_n)\ar[r,"T^2h"]&|[alias=D1]|T^2B\ar[d,"Tb"]\\
            A_{<i}\times TA_i\times A_{>i}\ar[r,"t_i"]\ar[d,"1\times a_i\times 1"']&|[alias=R1]|T(A_1\times \cdots\times A_n)\ar[rr,"Th"{swap,name=D2}]&&TB\ar[d,"b"]\\
            |[alias=RD]|A_1\times \cdots\times A_n\ar[rrr,"h"']&&&B
            \arrow[Rightarrow,from=D1,to=R1,"Th_i",shorten=15mm]
            \arrow[Rightarrow,from=D2,to=R2,"h_i",shorten < =25mm,shorten > =20mm]
        \end{tikzcd}}
        \end{equation}
    \end{center}
    are equal.
    
    $\bullet$ Coherence: For $i<j,$ the pasting diagrams
    \begin{center}
    \begin{equation}\label{coh1}
        \adjustbox{scale=0.9}{\begin{tikzcd}
            A_{<i}\times TA_i\times \cdots\times TA_{j}\times A_{>j}\ar[r,"t_i"]\ar[d,"1\times a_j\times 1"']&[-15pt]T(A_{<j}\times TA_j\times A_{>j})\ar[d,"T(1\times a_j\times 1)"']\ar[r,"Tt_j"]&[-15pt]|[alias=D1]|T^2(A_1\times \cdots \times A_n)\ar[r,"\mu"]\ar[d,"T^2h"]&[-15pt]T(A_1\times \cdots\times A_n)\ar[d,"Th"]\\
            A_{<i}\times TA_i\times A_{>i}\ar[r,"t_i"]\ar[dd,"1\times a_i\times 1"']&|[alias=R1]|T(A_1\times \cdots\times A_n)\ar[rd,"Th"']&T^2B\ar[d,"Tb"]\ar[r,"\mu"]&TB\ar[dd,"b"]\\[-10pt]
            &&TB\ar[rd,"b"]&\\[-10pt]
            |[alias=R2]|A_1\times \cdots \times A_n\ar[rrr,"h"']&&&B
            \arrow[Rightarrow,from=D1,to=R1,"Th_j",shorten=7mm]
            \arrow[Rightarrow,from=R1,to=R2,"h_i",shorten=8mm]
        \end{tikzcd}}
        \end{equation}
    \end{center}
    and,
    \begin{center}
    \begin{equation}\label{coh2}
        \adjustbox{scale=0.9}{\begin{tikzcd}
            &[-15pt]T(A_{<j}\times TA_j\times A_{>j})\ar[r,"Tt_j"name=D1]&[-15pt]T^2(A_1\times\cdots \times  A_n)\ar[rd,"\mu"]&\\
    A_{<i}\times TA_i\times \cdots\times TA_{j}\times A_{>j}\ar[r,"t_j"]\ar[ru,"t_i"]\ar[d,"1\times a_i\times 1"']&T(A_{<i}\times TA_i\times A_{>i})\ar[d,"T(1\times a_i\times 1)"']\ar[r,"Tt_i"{swap,name=R1}]&|[alias=D2]|T^2(A_1\times \cdots\times A_n)\ar[d,"T^2h"]\ar[r,"\mu"]&T(A_1\times\cdots\times  A_n)\ar[d,"Th"]\\
    A_{<j}\times TA_j\times A_{>j}\ar[r,"t_j"]\ar[dd,"1\times a_j\times 1"']&|[alias=R2]|T(A_1\times \cdots\times A_n)\ar[rd,"Th"']&T^2B\ar[d,"Tb"]\ar[r,"\mu"]&TB\ar[dd,"b"]\\[-10pt]
    &&TB\ar[rd,"b"]&\\[-10pt]
    |[alias=R3]|A_1\times\cdots\times A_n\ar[rrr,"h"']&&&B
    \arrow[Rightarrow,from=D1,to=R1,"\Gamma_{i,j}"',shorten=3mm]
    \arrow[Rightarrow,from=D2,to=R2,"Th_i",shorten=7mm]
    \arrow[Rightarrow,from=R2,to=R3,"h_j",shorten=10mm]
        \end{tikzcd}}
        \end{equation}
    \end{center}
    are equal.
\end{defi}
\begin{defi}\label{twocelltalg}{\cite[Def. 10]{HP02}} Let $(T,\eta,\mu,t,\Gamma)$ be a pseudo commutative, strong 2-monad. We define the 2-cells of $\talg$ as follows. Suppose that $(A_i,a_i\colon TA_i\to A_i)$ and $(B,b\colon TB\to B)$ are $T$-algebras for $1\leq i\leq n,$ and that  $(f,\la f_i\ra)$ and $(g,\la g_i\ra)$ are 1-cells in $\talg ( \la A_1 , \dots ,A_n\ra,B).$ A 2-cell $\alpha\colon f\to g$ in $\talg$ is the datum of a 2-cell in $\K$ of the form
    \begin{center}
\adjustbox{scale=0.9}{\begin{tikzcd}
	{A_1\times \cdots\times A_n} & B,
	\arrow[""{name=0, anchor=center, inner sep=0}, curve={height=-12pt}, from=1-1, to=1-2]
	\arrow[""{name=1, anchor=center, inner sep=0}, curve={height=12pt}, from=1-1, to=1-2]
	\arrow["\alpha", between={0.2}{0.8}, Rightarrow, from=0, to=1]
\end{tikzcd}}
    \end{center}
    subject to the equality, for $i<n,$ of the pasting diagrams
    \begin{center}
    \begin{equation}\label{2cell1}
    \adjustbox{scale=0.9}{
        \begin{tikzcd}
            A_{<i}\times TA_i\times A_{>i}\ar[d,"1\times a_i\times 1"']\ar[r,"t_i"']&T(A_1\times \cdots \times A_n)\ar[r,"Tf"]&|[alias=D1]|TB\ar[d,"b"]\\[25pt]
            |[alias=R1]|A_1\times\cdots\times A_n\ar[rr,bend right=10,"g"{swap,name=R2}]\ar[rr,bend left=10,"f"name=D2]&&B,
\arrow[Rightarrow,from=D1,to=R1,"f_i"',shorten <=20mm,shorten >=22mm]
        \arrow[Rightarrow,from=D2,to=R2,"\alpha",shorten=1mm]
        \end{tikzcd}}
        \end{equation}
    \end{center}
    and
    \begin{center}
    \begin{equation}\label{2cell2}
        \adjustbox{scale=0.9}{\begin{tikzcd}
            A_{<i}\times TA_i\times A_{>i}\ar[d,"1\times a_i\times 1"']\ar[r,"t_i"']&T(A_1\times \cdots \times A_n)\ar[r,bend right,"Tg"{swap,name=R2}]\ar[r,bend left,"Tf"name=D2]&[20pt]|[alias=D1]|TB\ar[d,"b"]\\
            |[alias=R1]|A_1\times\cdots\times A_n\ar[rr,"h"']&&B.
\arrow[Rightarrow,from=D1,to=R1,"g_i",shorten <=30mm,shorten >=22mm]
\arrow[Rightarrow,from=D2,to=R2,"T\alpha",shorten=1mm]
        \end{tikzcd}}
        \end{equation}
    \end{center}
    Vertical composition of 2-cells in $\talg$ is given by vertical composition in $\K.$
\end{defi}
Next we define the $\gamma$ composition in $\talg.$
 \begin{defi}\label{compositiontalg}
Let $(T,\eta,\mu,t,\Gamma)$ be a pseudo commutative, strong 2-monad. For  $(C,c)\in$  $\ob(\talg)$, $n\geq 0,$ $\la B \ra =\la (B_j,b_j)\ra_{j=1}^n\in \ob(\talg)^n,$ $k_j\geq 0$ for $1\leq j\leq n,$   and $\la A_j\ra =$ $\la A_{j,i}\ra_{i=1}^{k_j}$ $\in \ob(\talg)^{k_j}$ for $1\leq j \leq n,$ we define
        \begin{center}\adjustbox{scale=0.9}{
            \begin{tikzcd}
                \talg(\la B\ra;C)\times \prod\limits_{j=1}^n\talg(\la A_j\ra;B_j)\ar[r,"\gamma"]&\talg(\la A\ra;C)
            \end{tikzcd}}
        \end{center}
as follows. Let $(f,f_j)\colon B_1\times \cdots\times B_n\to C$ and $(g_j,g_{ji})\colon A_{j,1}\times \cdots \times A_{j,k_j}\to B_j$ 1-cells of $T$-algebras. We define their $\gamma$ composition as the $\K$ 1-cell
\begin{center}
\adjustbox{scale=0.9}{
    \begin{tikzcd}
        \overline{A_1}\times\cdots\times \overline{A_n}\ar[r,"\prod g_i"]&B_1\times \cdots \times B_n\ar[r,"f"]&C
    \end{tikzcd}}
\end{center}
where $\overline{A_j}$ denotes $\prod_{i=1}^{k_j}A_{j,i}.$  Any number between $s$ with  $1\leq s\leq \sum_{j=1}^nk_j$ can be uniquely written as $s=d+\sum_{t<j}k_t$ where $d<k_j.$ We define $\gamma(\la g_j\ra,f)_s$ as the pasting
\begin{center}
\adjustbox{scale=0.9}{
    \begin{tikzcd}
        &[-80]\overline{A_1}\times\cdots\times \overline{A_{j-1}}\times A_{j,1}\times \cdots\times TA_{j,d}\times \cdots\times A_{j,k_j}\times \overline{A_{j+1}}\times \cdots\times \overline{A_n}\ar[d,"1\times t_d\times 1"']\ar[r,"1\times a_{j,d}\times 1"]\ar[dddl,xshift=-6.2ex,"t_s"']&|[alias=R1]|\overline{A_1}\times \cdots\times \overline{A_n}\ar[dd,"1\times g_j\times 1"]\\
        &\overline{A_1}\times\cdots\times T(\overline{A_j})\times \cdots\times \overline{A_n}\ar[ddl,xshift=-3ex,"t_j"']\ar[d,"1\times Tg_j \times 1"]&\\
        &|[alias=D1]|\overline{A_1}\times \cdots\times TB_j\times \cdots\times \overline{A_n}\ar[d,"g_1\times \cdots\times 1\times \cdots \times g_n"']
        \ar[r,"1\times b_j\times 1"']&\overline{A_1}\times \cdots\times B_j\times \cdots\times \overline{A_n}\ar[d,"g_1\times \cdots\times 1\times \cdots\times g_n"']\\
        T(\overline{A_1}\times \cdots\times \overline{A_n})\ar[rd,"T(g_1\times\cdots\times g_n)"']&B_1\times\cdots\times TB_j\times \cdots\times B_n\ar[r,"1\times b_j\times 1"]\ar[d,"t_j"']&|[alias=R2]|B_1\times \cdots\times B_n\ar[dd,"f"]\\
        &T(B_1\times \cdots\times B_n\ar[d,"Tf"'])&\\
        &|[alias=D2]|TC\ar[r,"c"']&C.
        \arrow[Rightarrow,from=D1,to=R1,"1\times g_{j,d}\times 1"',shorten <=20mm,shorten >=35]
            \arrow[Rightarrow,from=D2,to=R2,"f_j"',shorten <=20mm,shorten >=40]
    \end{tikzcd}}
\end{center}
The multilinear composition for 2-cells is defined in the following way. Suppose that $(f,f_j),$ and $(f',f_j)$ are multilinear 1-cells from  $\la B_1, \dots, B_n\ra$ to $C,$ and $(g_j,g_{ji}),$ and $(g_j',g'_{ji})$ are multilinear 1-cells from $ \la A_{j,1},\dots,A_{j,{k_j}}\ra$ to $B_j$ in $\talg .$ Suppose also that $\alpha\colon f\to f',$ and $ \beta_j\colon g_j\to g_j'$ are 2-cells in $\talg.$ Then, the component 2-cell of $\gamma(\alpha;\beta_1,\dots,\beta_n)$ is the pasting
\begin{center}
    \adjustbox{scale=0.9}{
\begin{tikzcd}
	{\overline{A_1}\times\cdots\times \overline{A_n}} &[10pt] {B_1\times \cdots\times B_n} &[10pt] {C.}
	\arrow[""{name=0, anchor=center, inner sep=0}, "{\prod g_{j}}", curve={height=-12pt}, from=1-1, to=1-2]
	\arrow[""{name=1, anchor=center, inner sep=0}, "{\prod g_{j'}}"', curve={height=12pt}, from=1-1, to=1-2]
	\arrow[""{name=2, anchor=center, inner sep=0}, "f", curve={height=-12pt}, from=1-2, to=1-3]
	\arrow[""{name=3, anchor=center, inner sep=0}, "{f'}"', curve={height=12pt}, from=1-2, to=1-3]
	\arrow["{\prod\beta_{j}}", between={0.2}{0.8}, Rightarrow, from=0, to=1]
	\arrow["\alpha", between={0.2}{0.8}, Rightarrow, from=2, to=3]
\end{tikzcd}}
\end{center}
One can easily check that this composition is well defined.
\end{defi}
By imposing an extra condition on $T$ we can turn $\talg$ into symmetric $\Cat$-multicategory.
\begin{defi}\label{Tsymmetric}
A pseudo commutative, strong 2-monad $(T,\eta,\mu,t,\Gamma)$ is called \emph{symmetric} if for all $A,B$ objects of $\K,$ the following pasting diagram equals the identity of the 1-cell \begin{tikzcd}
    TA\times TB\ar[r,"t_1"]&T(A\times TB)\ar[r,"Tt_2"]&T^2(A\times B)\ar[r,"\mu"]&T(A\times B):
\end{tikzcd}
\begin{center}
    \begin{tikzcd}
        TA\times TB\ar[d,"\cong"']\ar[r,"t_1"]&T(A\times TB)\ar[r,"Tt_2"]&|[alias=D1]|T^2(A\times B)\ar[r,"\mu"]&T(A\times B)\\
        TB\times TA\ar[d,"\cong"']\ar[r,"t_1"]&|[alias=R1]|T(B\times TA)\ar[r,"Tt_2"]&|[alias=D2]|T^2(B\times A)\ar[r,"\mu"]&T(B\times A)\ar[u,"T\cong"']\\
        TA\times TB\ar[r,"t_1"]&|[alias=R2]|T(A\times TB)\ar[r,"Tt_2"]&T^2(A\times B)\ar[r,"\mu"]&T(A\times B).\ar[u,"T\cong"']
        \arrow[Rightarrow,from=D1,to=R1,"\Gamma_{A,B}"',shorten <=1mm,shorten >=1mm]
        \arrow[Rightarrow,from=D2,to=R2,"\Gamma_{B,A}"',shorten <=1mm,shorten >=1mm]
    \end{tikzcd}
\end{center}
\end{defi}
\begin{rmk} As noted in \cite{HP02}, this axiom introduces further redundancies in \Cref{pseudocomm}. Under symmetry, and using the numenclature from \Cref{pseudocomm},  the axioms (\ref{axiomstrength1}), (\ref{axiomstrength2}), and (\ref{axiomstrength3})  are equivalent, the conditions (\ref{axiometa1}) and (\ref{axiometa2}) are equivalent and the two statements (\ref{axiommu1gamma}) and (\ref{axiommu2gamma}) are equivalent.
\end{rmk}
In general, for $1\leq i<j\leq n$ we can define 2-cells $\Gamma_{j,i}$ that are inverses to the $\Gamma_{i,j}$ from \Cref{ij}.
\begin{defi}
Let $(T,\eta,\mu,t,\Gamma)$ be a pseudo commutative, strong 2-monad and $A_1,\dots,A_n$ objects of $\K.$ Let $K$ be a partition as in \Cref{partition}. We define $\Gamma_{j,i}$ as the whiskering
\begin{center}
\adjustbox{scale=0.85}{
\begin{tikzcd}
	{A_1\times \cdots\times TA_i\times \cdots\times TA_j\times \cdots\times A_n} \\[-5pt]
	{\underbrace{\cdots\times \cdots}_{K_1}|\underbrace{\cdots\times TA_i \times \cdots}_{K_2}|\underbrace{\cdots\times TA_j\times\cdots}_{K_3}|\underbrace{\cdots\times\cdots}_{K_4}} \\[-5pt]
	{\underbrace{\cdots\times \cdots}_{K_1}|\underbrace{\cdots\times TA_j \times \cdots}_{K_3}|\underbrace{\cdots\times TA_i\times\cdots}_{K_2}|\underbrace{\cdots\times\cdots}_{K_4}} \\[-5pt]
	{\underbrace{\cdots\times\cdots}_{K_1}\times T(\cdots\times A_i\times \cdots)\times T(\cdots\times A_j\times \cdots)\times \underbrace{\cdots\times \cdots}_{K_4}} \\
	{\underbrace{\cdots\times\cdots}_{K_1}\times T(\cdots\times A_i\times \cdots\times A_j\times \cdots)\times \underbrace{\cdots\times \cdots}_{K_4}} \\[-5pt]
	{T(\cdots\times A_i\times \cdots\times A_j\times \cdots)} \\[-10pt]
	{T(A_1\times \cdots\times A_n).}
	\arrow["{=}", from=1-1, to=2-1]
	\arrow["\cong", from=2-1, to=3-1]
	\arrow["{1\times t_{j-|K_1|}\times t_{i-|K_1|-|K_3|}\times 1}", from=3-1, to=4-1]
	\arrow[""{name=0, anchor=center, inner sep=0}, shift left=5, curve={height=-6pt}, from=4-1, to=5-1]
	\arrow[""{name=1, anchor=center, inner sep=0}, shift right=5, curve={height=6pt}, from=4-1, to=5-1]
	\arrow["{t_{|K_1|+1}}", from=5-1, to=6-1]
	\arrow["\cong", from=6-1, to=7-1]
	\arrow["{1\times \Gamma\times 1}", between={0.2}{0.8}, Rightarrow, from=1, to=0]
\end{tikzcd}}
\end{center}
\end{defi}
\begin{rmk}Notice that $\Gamma_{j,i}$ is independent of the partition by \Cref{ij}. The symmetry axiom can thus be written as $\Gamma_{1,2}=\Gamma_{2,1}^{-1}.$ If we write $\Gamma_{1,2}\colon \omega\to \omega ',$ then the symmetry axiom takes the form
\begin{center}
\begin{tikzcd}
	{TA\times TB} & {T(A\times B)} & {=} & {TA\times TB} & {T(A\times B).}
	\arrow[""{name=0, anchor=center, inner sep=0}, "\omega", curve={height=-24pt}, from=1-1, to=1-2]
	\arrow[""{name=1, anchor=center, inner sep=0}, "\omega"', curve={height=24pt}, from=1-1, to=1-2]
	\arrow[""{name=2, anchor=center, inner sep=0}, "{\omega'}"', from=1-1, to=1-2]
	\arrow[""{name=3, anchor=center, inner sep=0}, "\omega", curve={height=-12pt}, from=1-4, to=1-5]
	\arrow[""{name=4, anchor=center, inner sep=0}, "{\omega }"', curve={height=12pt}, from=1-4, to=1-5]
	\arrow["{\Gamma_{2,1}}"{pos=0.6}, between={0.5}{0.9}, Rightarrow, from=2, to=1]
	\arrow["{\Gamma_{1,2}}", between={0.2}{0.8}, Rightarrow, from=0, to=2]
	\arrow["{1_\omega}", between={0.2}{0.8}, equals, from=3, to=4]
\end{tikzcd}
\end{center}
\end{rmk}
\begin{lem}\label{inverses}Let $(T,\eta,\mu,t,\Gamma)$ be a symmetric, strong, pseudo commutative 2-monad. Let $0\leq i<j\leq n$. Then $\Gamma_{j,i}=\Gamma_{i,j}^{-1}$.
\end{lem}
We notice that since $\Gamma$ is invertible, the inverse pseudo commutativity 2-cells $\Gamma^{-1}$ satisfy analogous properties to those in \Cref{pseudocomm} satisfied by $\Gamma$. When using these properties we will refer the reader to \Cref{pseudocomm}.
Next we define the symmetric $\Cat$-multicategorical structure on $\talg$ for $T$ symmetric. This is the definition of Hyland and Power \cite{HP02}, which agrees with the one given in \cite{GMMO21} for pseudo commutative operads.
\begin{defi}{\cite[Prop. 18]{HP02}}
Let $(T,\eta,\mu,t,\Gamma)$ be a symmetric, pseudo commutative, strong 2-monad. We give $\talg$ the structrue of a symmetric $\Cat$-multicategory by defining the action of the symmetric group. For $A_1,\dots,A_n,B$ objects of $\K,$ and $\sigma\in \Sigma_n,$ define
\begin{center}
\adjustbox{scale=0.9}{
\begin{tikzcd}
    \talg (A_1,\dots,A_n;B)\ar[r,"\sigma"]&\talg (A_{\sigma(1)},\dots, A_{\sigma(n)};B),
\end{tikzcd}}
\end{center}
in the following way. If $(h,h_i)\colon \la A_1, \dots, A_n\ra\to B$ is a 1-cell in $\talg$, we define the 1-cell component of  $h\sigma$ in $\K$ as  $A_{\sigma(1)}\times\cdots\times A_{\sigma(n)}\xrightarrow{\sigma}A_1\times \cdots\times A_n\xrightarrow{h}B.$
We define $(h\sigma)_i$ as the pasting
\begin{center}
\adjustbox{scale=0.9}{
    \begin{tikzcd}
        A_{\sigma(1)}\times\cdots\times TA_{\sigma(i)}\times \cdots \times A_{\sigma(n)}\ar[r,"\sigma"]\ar[d,"1\times a_{\sigma(i)}\times 1"']&[-15pt]A_1\times\cdots\times TA_{\sigma(i)}\times \cdots\times A_n\ar[r,"t_{\sigma(i)}"]\ar[d,"1\times a_{\sigma(i)}\times 1"']&|[alias=D1]|T(A_1\times \cdots\times A_n)\ar[r,"h"]&[-15pt]TB\ar[d,"b"]\\[15pt]
        A_{\sigma(1)}\times \cdots\times A_{\sigma(n)}\ar[r,"\sigma"']&|[alias=R1]|A_1\times \cdots\times A_n\ar[rr,"h"']&&B.

        \arrow[Rightarrow,from=D1,to=R1,"h_{\sigma(i)}^n",shorten <=10mm,shorten >=10mm]
    \end{tikzcd}}
\end{center}
Similarly, for $\alpha\colon f\to g$ 2-cell in $\talg (A_1,\dots,A_n;B)(f,g),$ $\alpha\sigma$ is defined as having component 2-cell
\begin{center}
\begin{tikzcd}
	{A_{\sigma(1)}\times\cdots\times A_{\sigma(n)}} & {A_1\times \cdots\times A_n} & {B.}
	\arrow["\sigma", from=1-1, to=1-2]
	\arrow[""{name=0, anchor=center, inner sep=0}, curve={height=-12pt}, from=1-2, to=1-3]
	\arrow[""{name=1, anchor=center, inner sep=0}, curve={height=12pt}, from=1-2, to=1-3]
	\arrow["\alpha", between={0.2}{0.8}, Rightarrow, from=0, to=1]
\end{tikzcd}
\end{center}
\end{defi}

\begin{rmk}To prove that given a 1-cell $h\colon\la A_1,\dots,A_n\ra\to B$ in $\talg,$ $h\sigma\colon \la A_{\sigma(1)},\dots,A_{\sigma_n}\ra\to B$ is indeed a 1-cell in $\talg$ we need the symmetry axiom. The $\eta$ and $\mu$ axioms for $h\sigma$ follow from the same axioms for $h.$ To prove coherence one can prove that given $0\leq i<j\leq n,$ if   $h$ satisfies coherence, then so does $h\sigma_{i,j}.$ Here, $\sigma_{i,j}\in \Sigma_n$ is the transposition that permutes $i$ and $j.$  Coherence for $h\sigma_{i,j}$ follows from coherence for $h$ together with \Cref{inverses}.
\end{rmk}
\section{The free algebra functor as a non-symmetric multifunctor}\label{Section3.2}
Recall that for $(T,\eta,\mu,t,\Gamma)$ a pseudo commutative, strong 2-monad and $A\in \K,$ $(TA,\mu:T^2A\to TA)$ is a $T$-algebra and can be thought of as the free $T$ algebra generated by $A.$ This defines a 2-functor $T\colon \K\to \talg$ \cite{BKP02} that, as we show in this section, can be extended to non-symmetric multifunctor when  $T$ is symmetric.
\begin{defi}
Let $(T,\eta,\mu,t,\Gamma)$ be a symmetric, pseudo commutative, strong 2-monad. Given $A_1,A_2\in\ob(\K)$ we define the 2-ary 1-cell in $\talg,$ $\omega=\omega_{A_1,A_2}\colon \la TA_1,TA_2\ra\to T(A_1\times A_2)$ as follows. The component 1-cell $\omega_{A_1,A_2}$ is the composite
\begin{center}
\adjustbox{scale=0.9}{
    \begin{tikzcd}
        TA_1\times TA_2\ar[r,"t_1"]&[-10pt]T(A_1\times TA_2)\ar[r,"Tt_2"]&[-10pt]T^2(A_1\times A_2)\ar[r,"\mu"]&[-10pt]T(A_1\times A_2).
    \end{tikzcd}}
\end{center}
We can take the 2-cell $\omega_1$ to be the identity since the following diagram commutes by definition of $t_1$ and naturality of $\mu$:
\begin{center}
\adjustbox{scale=0.9}{
    \begin{tikzcd}
        T^2A_1\times TA_2\ar[r,"t_1"]\ar[d,"\mu\times 1"']&T(TA_1\times TA_2)\ar[r,"T\omega"]&T^2(A_1\times A_2)\ar[d,"\mu"]\\
        TA_1\times TA_2\ar[rr,"\omega"']&&T(A_1\times A_2).
    \end{tikzcd}}
\end{center}
We define $\omega_2$ as being the following pasting diagram:
\begin{center}
\adjustbox{scale=0.9}{
    \begin{tikzcd}
        TA_1\times T^2A_2\ar[dd,"1\times \mu"']\ar[rr,"t_2"]\ar[rd,"\cong"]&[-15pt]&[-15pt]T(TA_1\times TA_2)\ar[d,"T\cong"]\ar[rr,"T\omega"name=D1]&&T^2(A_1\times A_2)\ar[dd,"\mu"]\\
        
        &T^2A_2\times TA_1\ar[d,"\mu\times 1"']\ar[r,"t_1"]&T(TA_2\times TA_1)\ar[r,"T\omega"{swap,name=R1}]&T^2(A_2\times A_1).\ar[ru,"T^2\cong"']\ar[d,"\mu"]&\\

        TA_1\times TA_2\arrow[rd,equal]\ar[r,"\cong"]&TA_2\times TA_1\ar[d,"\cong"]\ar[rr,"\omega"name=D2]&&T(A_2\times A_1)\ar[r,"T\cong"]&T(A_1\times A_2)\\[-5pt]
        
        &TA_1\times TA_2\ar[rr,"\omega"{swap,name=R2}]&&T(A_1\times A_2)\ar[u,"T\cong"]\arrow[ru,equal]&
        \arrow[Rightarrow,near start,from=D1,to=R1,"T\Gamma_{A_1,A_2}",shorten <=5mm,shorten >=5mm]

        \arrow[Rightarrow,from=D2,to=R2,"\Gamma_{A_2,A_1}",shorten <=2mm,shorten >=2mm]
    \end{tikzcd}}
\end{center}
\end{defi}
\begin{rmk}
The previous definition generalizes the definitions of \cite{GMMO21} for the case of  pseudo commutative operads.
\end{rmk}
\begin{lem}\label{omega2} Let $(T,\eta,\mu,t,\Gamma)$ be a symmetric, pseudo commutative, strong 2-monad. For $A_1,A_2$ be objects of $\K$, $\omega_2$ equals the whiskering
\begin{center}
\begin{tikzcd}
	{TA_1\times T^2A_2} & {T(A_1\times TA_2)} & {T^2(A_1\times A_2)} & {T(A_1\times A_2).}
	\arrow[""{name=0, anchor=center, inner sep=0}, "{\omega '}", curve={height=-12pt}, from=1-1, to=1-2]
	\arrow[""{name=1, anchor=center, inner sep=0}, "\omega"', curve={height=12pt}, from=1-1, to=1-2]
	\arrow["{Tt_2}", from=1-2, to=1-3]
	\arrow["\mu", from=1-3, to=1-4]
	\arrow["{\Gamma^{-1}}", between={0.2}{0.8}, Rightarrow, from=0, to=1]
\end{tikzcd}
\end{center}
\end{lem}
\begin{proof}
By applying the strength $\mu$-axiom (\ref{axiommu1gamma}) in \Cref{pseudocomm}, we get that $\omega_2$ equals the pasting
\begin{center}
    \begin{tikzcd}[font=\tiny]
        TA_1\times T^2A_2\ar[rr,"t_2"]\ar[d,"\cong"']&[-15pt]&[-15pt]T(TA_1\times TA_2)\ar[dl,"\cong"']\ar[rr,"T\omega"name=D1]&[-15pt]&[-15pt]T^2(A_1\times A_2)\ar[ddd,"\mu"']\\
        T^2A_2\times TA_1\ar[d,"t_2"]\ar[r,"t_1"]&T(TA_2\times TA_1)\ar[d,"Tt_2"name=D3]\ar[r,"Tt_1"]&|[alias=D2]|T^2(A_1\times TA_1)\ar[r,"T^2t_2"name=R1]&T^3(A_2\times A_1)\ar[d,"T\mu"']&\\
        T(T^2A_2\times A_1\ar[d,"Tt_1"{swap}])&T^2(TA_2\times A_1)\ar[d,"\mu"']\ar[r,"Tt_1"name=R2]&T^2(A_2\times A_1)\ar[r,"T\mu"]\ar[d,"\mu"]&T^2(A_2\times A_1)\ar[d,"\mu"]\ar[ruu,"T^2\cong"']&\\
        |[alias=R3]|T^2(TA_2\times A_1)\ar[r,"\mu"']&T(TA_2\times A_1)\ar[r,"Tt_1"']&T^2(A_2\times A_1)\ar[r,"\mu"']&T(A_2\times A_1)\ar[r,"T\cong"']&T(A_1\times A_2).

        \arrow[Rightarrow,from=D1,to=R1,"T\Gamma_{A_1,A_2}",shorten <=4mm,shorten >=1mm]
        \arrow[Rightarrow,from=D2,to=R2,"T\Gamma_{A_2,A_1}",shorten <=1mm,shorten >=1mm]

        \arrow[Rightarrow,from=D3,to=R3,"\Gamma",shorten <=9mm,shorten >=7mm]
    \end{tikzcd}
\end{center}
By symmetry $\omega_2$ agrees with the whiskering
\begin{center}
\adjustbox{scale=0.9}{
    \begin{tikzcd}
        TA_1\times T^2A_2\ar[r,"\cong"]&[-15pt]T^2A_2\times TA_1\ar[r,bend left,"{\color{white}t}"name=D1]\ar[r,bend right,"{\color{white}t}"'name=R1]&[20pt]T(TA_2\times A_1)\ar[r,"Tt_1"]\ar[d,"T\cong"]&[-10pt]T^2(A_2\times A_1)\ar[r,"\mu"]\ar[d,"T^2\cong"]&[-15pt]T(A_2\times A_1)\ar[d,"T\cong"]\\[-10pt]
        &&T(A_1\times TA_2)\ar[r,"Tt_2"']&T^2(A_1\times A_2)\ar[r,"\mu"']&T(A_1\times A_2)
        \arrow[Rightarrow,from=D1,to=R1,near start,"\Gamma"',shorten <=2mm,shorten >=2mm]
    \end{tikzcd}}
\end{center}
By symmetry, the last whiskering equals the one in the lemma.
\end{proof}
\begin{lem}\label{omegais1cell}
    Let $(T,\eta,\mu,t,\Gamma)$ be a symmetric, pseudo commutative, strong 2-monad, and let $A_1,A_2$ be objects of $\K.$ Then 
    $$\omega\colon \la TA_1, TA_2\ra\to T(A_1\times A_2)$$
    is a 2-ary 1-cell in $\talg$.
\end{lem}
\begin{proof}
First we tackle coherence. By \Cref{omega2}, we can write the \Cref{coh2} for $\omega$ as
\begin{center}
\adjustbox{scale=0.9}{
\begin{tikzcd}
    &T(TA_1\times TA_2)\ar[rr,"Tt_2"name=D1]&[-15pt]&[-15pt]T^2(TA_1\times TA_2)\ar[rd,"\mu"]&\\[-10pt]
    T^2A_1\times T^2A_2\ar[d,"\mu\times 1"']\ar[ru,"t_1"]\ar[r,"t_2"']&|[alias=R1]|T(T^2A_1\times TA_2)\ar[d,"T(\mu\times 1)"]\ar[rr,"Tt_1"']&&T^2(TA_1\times TA_2)\ar[r,"\mu"']\ar[d,"T^2\omega"]&T(TA_1\times TA_2)\ar[d,"Tt_\omega"]\\
    TA_1\times T^2A_2\ar[d,"t_1"']\ar[r,"t_2"]&T(TA_1\times TA_2)\ar[d,"Tt_1"]&&T^3(A_1\times A_2)\ar[d,"T\mu"]\ar[r,"\mu"]&T^2(A_1\times A_2)\ar[ddd,"\mu"]\\
    T(A_1\times T^2A_2\ar[dd,"Tt_2"{swap,name=R2}])&|[alias=D2]|T^2(A_1\times TA_2)\ar[rd,"T^2t_2"]\ar[dd,"\mu"]&&T^2(A_1\times A_2)\ar[rdd,"\mu"]&\\[-20pt]
    &&T^3(A_1\times A_2)\ar[rd,"\mu"]\ar[ru,"T\mu"]&&\\[-20pt]
    T^2(A_1\times TA_2)\ar[r,"\mu"']&T(A_1\times TA_2)\ar[rr,"Tt_2"']&&T^2(A_1\times A_2)\ar[r,"\mu"]&T(A_1\times A_2).
    \arrow[Rightarrow,from=D1,to=R1,"\Gamma_{TA_1,TA_2}",shorten <=4mm,shorten >=4mm]
    \arrow[Rightarrow,from=D2,to=R2,"\Gamma^{-1}",shorten <=4mm,shorten >=4mm]
\end{tikzcd}}
\end{center}
By  (\ref{axiommu2gamma}) in \Cref{pseudocomm},  this equals the pasting
\begin{center}
\adjustbox{scale=0.9}{
\begin{tikzcd}
	&[-5pt] {T(TA_1\times T^2A_2)} &[-5pt] {T^2(TA_1\times TA_2)}&[-5pt]&[-5pt] \\[-10pt]
	{T^2A_1\times T^2A_2} & {T(T^2A_1\times TA_2)} & {T(TA_1\times TA_2)} & {T(TA_1\times TA_2)} & {T^2(A_1\times TA_2)} \\
	{T(TA_1\times T^2A_2)} & {T^2(TA_1\times TA_2)} & {T^3(A_1\times TA_2)} & {T(A_1\times TA_2)} & {T^3(A_1\times A_2)} \\
	{T^2(A_1\times T^2A_2)} & {T^3(A_1\times TA_2)} & {T^2(A_1\times TA_2)} & {T^2(A_1\times A_2)} & {T^2(A_1\times A_2)} \\[-10pt]
	&& {T^3(A_1\times A_2)} & {T^2(A_1\times A_2)} & {T(A_1\times A_2).}
	\arrow["{Tt_2}", from=1-2, to=1-3]
	\arrow["{\Gamma_{TA_1,TA_2}}"'{pos=0.9}, between={0.3}{0.7}, Rightarrow, from=1-3, to=2-2]
	\arrow["\mu", from=1-3, to=2-4]
	\arrow["{t_1}", from=2-1, to=1-2]
	\arrow["{t_2}"', from=2-1, to=2-2]
	\arrow["{t_1}"', from=2-1, to=3-1]
	\arrow["{Tt_1}"', from=2-2, to=2-3]
	\arrow["{\Gamma^{-1}_{TA_1,TA_2}}"{pos=0.1}, between={0.3}{0.7}, Rightarrow, from=2-2, to=3-1]
	\arrow["\mu", from=2-3, to=2-4]
	\arrow["{Tt_1}", from=2-4, to=2-5]
	\arrow["\mu", from=2-5, to=3-4]
	\arrow["{T^2t_2}", from=2-5, to=3-5]
	\arrow["{Tt_2}"', from=3-1, to=3-2]
	\arrow["{Tt_1}"', from=3-1, to=4-1]
	\arrow["\mu"', from=3-2, to=2-4]
	\arrow["{T^2t_1}"', from=3-2, to=3-3]
	\arrow["\mu", from=3-3, to=2-5]
	\arrow["{T\Gamma^{-1}}"{pos=0.4}, between={0.4}{0.6}, Rightarrow, from=3-3, to=4-1]
	\arrow["{T\mu}"', from=3-3, to=4-3]
	\arrow["{Tt_2}", from=3-4, to=4-4]
	\arrow["\mu", from=3-5, to=4-4]
	\arrow["{T\mu}", from=3-5, to=4-5]
	\arrow["{T^2t_2}"', from=4-1, to=4-2]
	\arrow["{T\mu}"', from=4-2, to=4-3]
	\arrow["\mu"', from=4-3, to=3-4]
	\arrow["{T^2t_2}"', from=4-3, to=5-3]
	\arrow["\mu"', from=4-4, to=5-5]
	\arrow["\mu", from=4-5, to=5-5]
	\arrow["\mu", from=5-3, to=4-4]
	\arrow["{T\mu}"', from=5-3, to=5-4]
	\arrow["\mu"', from=5-4, to=5-5]
\end{tikzcd}}
\end{center}
The 2-cells $\Gamma_{TA_1,TA_2}$ and its inverse cancel out, so the previous pasting diagram equals \Cref{coh1} by \Cref{omega2}. Now we tackle the $\eta$ and $\mu$ axioms. For $i=1$ there is nothing to prove since $\omega_1$ is the identity.  The $\eta$ axiom for $i=2$  follows by \Cref{omega2} and  (\ref{axiometa2}) in \Cref{pseudocomm}. Let's prove that $\omega$ satisfies the $\mu$ axiom for $i=2.$ We start from pasting \Cref{mu1}, which by \Cref{omega2} we can express as the whiskering
\begin{center}
\adjustbox{scale=0.9}{\begin{tikzcd}
	{TA_1\times T^3A_2} & {TA_1\times T^2A_2} &[10pt] {T(A_1\times TA_2)} &[10pt] {T(A_1\times A_2).} & {}
	\arrow["{1\times \mu}", from=1-1, to=1-2]
	\arrow[""{name=0, anchor=center, inner sep=0}, curve={height=-12pt}, from=1-2, to=1-3]
	\arrow[""{name=1, anchor=center, inner sep=0}, curve={height=12pt}, from=1-2, to=1-3]
	\arrow["{\mu\circ Tt_2}", from=1-3, to=1-4]
	\arrow["{\Gamma^{-1}}", between={0.2}{0.8}, Rightarrow, from=0, to=1]
\end{tikzcd}}
    \end{center}
    By (\ref{axiommu1gamma}) in \Cref{pseudocomm} the previous diagram equals the following:
\begin{center}
\adjustbox{scale=0.9}{
    \begin{tikzcd}
        &[5pt]TA_1\times T^3A_2\ar[dl,"1\times T\mu"']\ar[d,"t_1"']\ar[r,"t_2"]&[-15pt]T(TA_1\times T^2A_2)\ar[d,"Tt_1"']\ar[r,"Tt_2"]&[-15pt]|[alias=D1]|T^2(TA_1\times TA_2)\ar[r,"T^2t_1"]&[-15pt]T^3(A_1\times TA_2)\ar[d,"T\mu"]\\
        TA_1\times T^2A_2\ar[d,"t_1"']&T(A_1\times T^3A_2\ar[dl,"T(1\times T\mu)"']\ar[d,"Tt_2"{swap,name=R2}])&|[alias=D2]|T^2(A_1\times T^2A_2)\ar[d,"\mu"']\ar[r,"T^2t_2"{swap,name=R1}]&T^3(A_1\times TA_2)\ar[r,"T\mu"']\ar[d,"\mu"]&T^2(A_1\times TA_2)\ar[d,"\mu"]\\
        T(A_1\times T^2A_2)\ar[rdd,"Tt_2"']&T^2(A_1T^2A_2)\ar[dd,"T^2(1\times \mu)"]\ar[r,"\mu"']&T(A_1\times T^2A_2)\ar[r,"Tt_2"']\ar[dd,"T(1\times \mu)"]&T^2(A_1\times TA_2)\ar[r,"\mu"]\ar[d,"T^2t_2"]&T(A_1\times TA_2)\ar[d,"Tt_2"]\\[-12pt]
        &&&T^3(A_1\times A_2)\ar[r,"\mu"']\ar[d,"T\mu"]&T^2(A_1\times A_2\ar[d,"\mu"])\\[-12pt]
        &T^2(A_1\times TA_2)\ar[r,"\mu"']&T(A_2 \times TA_2)\ar[r,"Tt_2"']&T^2(A_1\times A_2)\ar[r,"\mu"']&T(A_1\times A_2).
        \arrow[Rightarrow,from=D1,to=R1,"T\Gamma^{-1}"',shorten <=3mm,shorten >=3mm]
        \arrow[Rightarrow,from=D2,to=R2,"\Gamma^{-1}",shorten <=3mm,shorten >=3mm]
    \end{tikzcd}}
\end{center}
Here we decorated the diagram coming from \Cref{pseudomu2} with some extra commutative squares that do not change the pasting. By using that $\Gamma^{-1}$ is a modification, we get that the previous pasting equals
\begin{center}
\adjustbox{scale=0.9}{
    \begin{tikzcd}
        &[-10pt]TA_1\times T^3A_2\ar[dl,"1\times T\mu"']\ar[r,"t_2"]&[-10pt]T(TA_1\times T^2A_2)\ar[ld,"T(1\times \mu)"']\ar[d,"Tt_1"']\ar[r,"Tt_2"]&[-10pt]|[alias=D1]|T^2(TA_1\times TA_2)\ar[r,"T^2t_1"]&[-10pt]T^3(A_1\times TA_2)\ar[d,"T\mu"]\\
        TA_1\times T^2A_2\ar[r,"t_2"]\ar[d,"t_1"']&T(TA_1\times TA_2)\ar[d,"Tt_1"']&T^2(A_1\times T^2A_2)\ar[d,"\mu"]\ar[dl,"T^2(1\times \mu)"']\ar[r,"T^2t_2"{swap,name=R1}]&T^3(A_1\times TA_2)\ar[r,"T\mu"']\ar[d,"\mu"]&T^2(A_1\times TA_2)\ar[d,"\mu"]\\
        T(A_1\times T^2A_2)\ar[dd,"Tt_2"{swap,name=R2}]&|[alias=D2]|T^2(A_1\times TA_2)\ar[dd,"\mu"]&T(A_1\times T^2A_2)\ar[ddl,"T(1\times \mu)"]\ar[r,"Tt_2"']&T^2(A_1\times TA_2)\ar[r,"\mu"]\ar[d,"T^2t_2"]&T(A_1\times TA_2)\ar[d,"Tt_2"]\\[-12pt]
        &&&T^3(A_1\times A_2)\ar[r,"\mu"']\ar[d,"T\mu"]&T^2(A_1\times A_2)\ar[d,"\mu"]\\[-12pt]
        T^2(A_1\times TA_2\ar[r,"\mu"'])&T(A_1\times TA_2)\ar[rr,"Tt_2"']&&T^2(A_1\times A_2)\ar[r,"\mu"']&T(A_1\times A_2).
        \arrow[Rightarrow,from=D1,to=R1,"T\Gamma^{-1}"',shorten <=3mm,shorten >=3mm]
        \arrow[Rightarrow,from=D2,to=R2,"\Gamma^{-1}",shorten <=5mm,shorten >=5mm]
    \end{tikzcd}}
\end{center}
The last pasting equals \Cref{mu2} for $\omega$ by two applications of \Cref{omega2} and a change in 1-cells. The previous diagram has the correct source 1-cell since the 1-cells $\mu\circ (T\mu )\circ (T^2t_2)$ and $\mu\circ (Tt_2)\circ\mu \colon T^2(A_1\times TA_2)\to T(A_1\times A_2)$ are equal. 
        
        

\end{proof}
\begin{lem}{(Associativity of $\omega$)}\label{omegaass}
    Let $(T,\eta,\mu,t,\Gamma)$ be a symmetric, pseudo commutative, strong 2-monad. Let $A,B,C$ be objects of $\K,$ then
    $$\gamma(\omega_{A,B\times C};1_A,\omega_{B,C})=\gamma(\omega_{A\times B,C};\omega_{A,B},1_C),$$
    that is, the following multicategorical diagram commutes
    \begin{center}
        \begin{tikzcd}
            \la TA,TB,TC\ra\ar[d,"\la 1_A{,}\omega_{B,C}\ra"']\ar[r,"\la\omega_{A,B}{,}1_C\ra"]&\la T(A\times B),TC\ra\ar[d,"\omega_{A\times B,C}"]\\
            \la TA,T(B\times C)\ra\ar[r,"\omega_{A,B\times C}"']&T(A\times B\times C).
        \end{tikzcd}
    \end{center}
\end{lem}
\begin{proof}
First of all by associativity of $t,$ the strength axioms, the monad axioms and naturality of various 2-natural transformations, the corresponding 1-cells $TA\times TB\times TC\to T(A\times B\times C)$ are equal. We must show that the 2-cell constrains are equal, i.e., $\gamma(\omega;\omega,1)_i=\gamma(\omega ;1,\omega)_i$ for $i=1,2,3.$ For $i=1$ this follows since both $\gamma(\omega;\omega,1)_1$ and $\gamma(\omega;1,\omega)_1$ are identities. For $i=2,$ we have from \Cref{compositiontalg} and by \Cref{omega2} that $\gamma(\omega;\omega,1)_2$ is the 2-cell 
\begin{center}
   \adjustbox{scale=0.85}{
\begin{tikzcd}
	{TA\times T^2B\times TC} &[20pt] {T(A\times TB)\times TC} &[-15pt] {T^2(A\times B)\times TC} &[-15pt] {T(A\times B)\times TC} &[-15pt] {T(A\times B\times TC)} \\
	& {T(A\times TB\times TC)} & {T(T(A\times B)\times TC)} & {T^2(A\times B\times TC)} & {T(A\times B\times TC),}
	\arrow[""{name=0, anchor=center, inner sep=0}, curve={height=-12pt}, from=1-1, to=1-2]
	\arrow[""{name=1, anchor=center, inner sep=0}, curve={height=12pt}, from=1-1, to=1-2]
	\arrow["{Tt_2\times 1}"', from=1-2, to=1-3]
	\arrow["{t_1}"', from=1-2, to=2-2]
	\arrow["{\mu\times 1}", from=1-3, to=1-4]
	\arrow["{t_1}", from=1-3, to=2-3]
	\arrow["{t_1}", from=1-4, to=1-5]
	\arrow["{\mu\circ Tt_3}", from=1-5, to=2-5]
	\arrow["{T(t_2\times 1)}"', from=2-2, to=2-3]
	\arrow["{Tt_1}"', from=2-3, to=2-4]
	\arrow["\mu"', from=2-4, to=1-5]
	\arrow["{\Gamma^{-1}\times 1}"{pos=0.1}, between={0.2}{0.8}, Rightarrow, from=0, to=1]
\end{tikzcd}}
\end{center}
where $\mu\circ Tt_3\circ t_1=\omega\colon T(A\times B)\times TC\to T(A\times B\times C).$ We can then apply (\ref{axiomstrength3}) in \Cref{pseudocomm},  to get
\begin{center}
\adjustbox{scale=0.85}{
\begin{tikzcd}
	{TA\times T^2B\times TC} & {TA\times T(TB\times TC)} & {T(A\times TB\times TC)} &[60pt] {T(A\times B\times C).}
	\arrow["{1\times t_1}", from=1-1, to=1-2]
	\arrow[""{name=0, anchor=center, inner sep=0}, curve={height=-12pt}, from=1-2, to=1-3]
	\arrow[""{name=1, anchor=center, inner sep=0}, curve={height=12pt}, from=1-2, to=1-3]
	\arrow["{\mu\circ Tt_3\circ \mu \circ Tt_1\circ T(t_2\times 1)}", from=1-3, to=1-4]
	\arrow["{\Gamma^{-1}}", between={0.2}{0.8}, Rightarrow, from=0, to=1]
\end{tikzcd}}
\end{center}
Since $\mu \circ Tt_3 \circ\mu \circ Tt_1 \circ T(t_2\times 1)$ equals $\mu\circ Tt_2\circ T(1\times \mu)\circ T(1\times Tt_2)\circ T(1\times t_1)$ as 1-cells $T(A\times TB\times TC)\to T(A\times B\times C)$ by associativity of $t,$ strength axioms for $\mu$ and monad axioms for $\mu,$ we can write the previous pasting as
\begin{center}
\adjustbox{scale=0.85}{
    \begin{tikzcd}
        TA\times T^2B\times TC\ar[r,"1\times t_1"]&[-10pt]TA\times T(TB\times TC)\ar[ddl,"1\times Tt_1"']\ar[r,"t_2"]\ar[d,"t_1"']&[-15pt]T(TA\times TB\times TC)\ar[r,"Tt_1"]&[-15pt]|[alias=D1]|T^2(A\times TB\times TC)\ar[d,"\mu"]\\
        &|[alias=R1]|T(A\times T(TB\times TC))\ar[r,"Tt_2"{swap}]\ar[d,"T(1\times Tt_1)"]&T^2(A\times TB\times TC)\ar[r,"\mu"']\ar[d,"T^2(1\times t_1)"]&T(A\times TB\times TC)\ar[d,"T(1\times t_1)"]\\
        TA\times T^2(B\times TC)\ar[d,"1\times T^2t_2"']\ar[r,"t_1"]&T(A\times T^2(B\times TC))\ar[r,"Tt_2"]\ar[d,"T(1\times T^2t_2)"]&T^2(A\times T(B\times TC))\ar[r,"\mu"]\ar[d,"T^2(1\times Tt_2)"]&T(A\times T(B\times TC))\ar[d,"T(1\times Tt_2)"]\\
        TA\times T^3(B\times C)\ar[d,"1\times T\mu"']\ar[r,"t_1"]&T(A\times T^3(B\times C))\ar[r,"Tt_2"]\ar[d,"T(1\times T\mu)"]&T^2(A\times T^2(B\times C))\ar[r,"\mu"]\ar[d,"T^2(1\times \mu)"]&T(A\times T^2(B\times C))\ar[d,"T(1\times \mu)"]\\
        TA\times T^2(B\times C)\ar[r,"t_1"']&T(A\times T^2(B\times C))\ar[r,"Tt_2"']&T^2(A\times T(B\times C))\ar[r,"\mu"']&T(A\times T(B\times C))\ar[r,"\mu \circ Tt_2"']&[-15pt]T(A\times B\times C).
        \arrow[Rightarrow,from=D1,to=R1,"\Gamma^{-1}"',shorten <=12mm,shorten >=12mm]
    \end{tikzcd}}
\end{center}
Since $\Gamma^{-1}$ is a modification, our diagram equals
\begin{center}
\adjustbox{scale=0.85}{
\begin{tikzcd}
	{TA\times T^2B\times TC} &[80pt] {TA\times T^2(B\times C)} & {T(A\times T(B\times C))} & {T(A\times B\times C).}
	\arrow["{(1\times T\mu)\circ (1\times T^2t_2)\circ (1\times Tt_1)\circ (1\times t_1)}", from=1-1, to=1-2]
	\arrow[""{name=0, anchor=center, inner sep=0}, curve={height=-12pt}, from=1-2, to=1-3]
	\arrow[""{name=1, anchor=center, inner sep=0}, curve={height=12pt}, from=1-2, to=1-3]
	\arrow["{\mu \circ Tt_2}", from=1-3, to=1-4]
	\arrow["{\Gamma^{-1}}", between={0.2}{0.8}, Rightarrow, from=0, to=1]
\end{tikzcd}}
\end{center}
By an application of \Cref{omega2}, the previous whiskering is precisely $\gamma(\omega;1,\omega)_2$.\\
Let's now prove that $\gamma(\omega;\omega,1)_3=\gamma(\omega;1,\omega)_3.$ \Cref{compositiontalg} and an application of \Cref{omega2} give us that $\gamma(\omega;\omega,1)_3$ is the whiskering
\begin{center}
\adjustbox{scale=0.85}{
\begin{tikzcd}
	{TA\times TB\times T^2C} &[20pt] {T^2(A\times B)\times T^2C} &[-15pt] {T(A\times B)\times T^2C} & {T(A\times B\times TC)} &[-10pt] {T(A\times B\times C).}
	\arrow["{(Tt_2\times 1)\circ(t_1\times 1)}", from=1-1, to=1-2]
	\arrow["{\mu\times 1}", from=1-2, to=1-3]
	\arrow[""{name=0, anchor=center, inner sep=0}, curve={height=-12pt}, from=1-3, to=1-4]
	\arrow[""{name=1, anchor=center, inner sep=0}, curve={height=12pt}, from=1-3, to=1-4]
	\arrow["{\mu\circ Tt_3}", from=1-4, to=1-5]
	\arrow["{\Gamma^{-1}}", between={0.2}{0.8}, Rightarrow, from=0, to=1]
\end{tikzcd}}
\end{center}
By an application of (\ref{axiommu2gamma}) in  \Cref{pseudocomm},  we then have that $\gamma(\omega;\omega,1)_3$ equals
\begin{center}
\begin{equation}\label{delta3}
   \adjustbox{scale=0.85}{\begin{tikzcd}
        TA\times TB\times T^2C\ar[r,"(Tt_2\times 1)(t_1\times 1)"]&[10pt]T^2(A\times B)\times T^2C\ar[r,"t_2"]\ar[d,"t_1"']&[-15pt]|[alias=D1]|T(T^2(A\times B)\times TC)\ar[r,"Tt_1"]&[-15pt]T^2(T(A\times B)\times TC)\ar[d,"\mu"]\\
        &T(T(A\times B)\times T^2C)\ar[d,"t_1"]\ar[r,"Tt_2"{swap,name=R1}]&T^2(T(AB)\times TC)\ar[r,"\mu"']\ar[d,"T^2t_1"]&T(T(A\times B)\times TC)\ar[d,"Tt_1"]\\
        &T^2(A\times B\times T^2C)\ar[d,"Tt_2"{swap,name=R2}]&|[alias=D2]|T^3(A\times B\times TC)\ar[r,"\mu"']\ar[d,"T\mu"]&T^2(A\times B\times TC)\ar[d,"\mu"]\\
        &T^3(A\times B\times TC)\ar[r,"T\mu"']&T^2(A\times B\times TC)\ar[r,"\mu"]&T(A\times B\times TC)\ar[r,"\mu\circ Tt_3"']&[-15pt]T(A\times B\times C).
        \arrow[Rightarrow,from=D1,to=R1,"\Gamma^{-1}"',shorten <=4mm,shorten >=4mm]
        \arrow[Rightarrow,from=D2,to=R2,"T\Gamma^{-1}",shorten <=6mm,shorten >=6mm]
    \end{tikzcd}}
    \end{equation}
\end{center}
Now, by \Cref{compositiontalg} and two applications of \Cref{omega2}, $\gamma(\omega;1,\omega)_3$ is the vertical composition of the whiskering
\begin{center}
    \begin{equation}\label{gamma3-2}
    \adjustbox{scale=0.85}{
\begin{tikzcd}
	{TA\times TB\times T^2C} &[20pt] {TA\times T(B\times TC)} &[60pt] {T(A\times B\times C),}
	\arrow[""{name=0, anchor=center, inner sep=0}, curve={height=-12pt}, from=1-1, to=1-2]
	\arrow[""{name=1, anchor=center, inner sep=0}, curve={height=12pt}, from=1-1, to=1-2]
	\arrow["{\mu\circ Tt_2\circ t_1\circ (1\times \mu)\circ(1\times Tt_2)}", from=1-2, to=1-3]
	\arrow["{1\times \Gamma^{-1}}",pos={0.1}, between={0.2}{0.8}, Rightarrow, from=0, to=1]
\end{tikzcd}}
    \end{equation}
\end{center}
with the whiskering
\begin{center}
    \begin{equation}\label{gamma3-1}
     \adjustbox{scale=0.85}{
\begin{tikzcd}
	{TA\times TB\times TC} &[90pt] {TA\times T^2(B\times C)} & {T(A\times T(B\times C))} &[-10pt]{T(A\times B\times C).}
	\arrow["{(1\times T\mu)\circ(1\times T^2t_2)\circ(1\times Tt_1)\times (1\times t_2)}", from=1-1, to=1-2]
	\arrow[""{name=0, anchor=center, inner sep=0}, curve={height=-12pt}, from=1-2, to=1-3]
	\arrow[""{name=1, anchor=center, inner sep=0}, curve={height=12pt}, from=1-2, to=1-3]
	\arrow["{\mu\circ Tt_2}", from=1-3, to=1-4]
	\arrow["{\Gamma^{-1}}", between={0.2}{0.8}, Rightarrow, from=0, to=1]
\end{tikzcd}}
    \end{equation}
\end{center}
We will show that \Cref{gamma3-1} equals the whiskering
\begin{center}
\begin{equation}\label{3.2.4}
   \adjustbox{scale=0.85}{
\begin{tikzcd}
	{TA\times TB\times T^2C} &[20pt] {T^2(A\times B)\times T^2C} & {T(T(A\times B)\times TC)} &[20pt] {T(A\times B\times C)}
	\arrow["{(Tt_2\times 1)\circ(t_1\times 1)}", from=1-1, to=1-2]
	\arrow[""{name=0, anchor=center, inner sep=0}, curve={height=-12pt}, from=1-2, to=1-3]
	\arrow[""{name=1, anchor=center, inner sep=0}, curve={height=12pt}, from=1-2, to=1-3]
	\arrow["{\mu\circ Tt_3\circ \mu\circ Tt_1}", from=1-3, to=1-4]
	\arrow["{\Gamma^{-1}}", between={0.2}{0.8}, Rightarrow, from=0, to=1]
\end{tikzcd}}
\end{equation}
\end{center}
comming from \Cref{delta3}, as well as an analogous statement for \Cref{gamma3-2}. We can decorate \Cref{3.2.4} with some extra commutative squares
\begin{center}
    \adjustbox{scale=0.85}{\begin{tikzcd}
        TA\times TB\times T^2 C\ar[r,"t_1\times 1"']&[-20pt] T(A\times TB)\times T^2C\ar[r,"t_2"]
    \ar[d,"Tt_2\times 1"']&[-20pt] T(T(A\times TB)\times TC)\ar[r,"Tt_1"]\ar[d,"T(Tt_2 \times 1)"]&[-20pt]T^2(A\times TB\times TC)\ar[d,"T^2(t_2\times 1)"']\ar[rd,"\mu"]&\\
         &T^2(A\times B)\times T^2C\ar[r,"t_2"]\ar[d,"t_1"']&T(T^2(A\times B)\times TC)\ar[r,"Tt_1"]&|[alias=D1]|T^2(T(A\times B)\times TC)\ar[d,"\mu"]&[-20pt]T(A\times TB\times TC)\ar[ld,"T(t_2\times 1)"]\\  
        &|[alias=R1]|T(T(A\times B)\times T^2C)\ar[r,"Tt_2"{swap}]&T^2(T(AB)\times TC)\ar[r,"\mu"']&T(T(A\times B)\times TC)\ar[r,"\mu\circ T t_3\circ\mu\circ Tt_1"']&T(A\times B\times TC),
\arrow[Rightarrow,from=D1,to=R1,"\Gamma^{-1}"',shorten <=14mm,shorten >=14mm]
    \end{tikzcd}}
\end{center}
so that we can apply the fact that $\Gamma^{-1}$ is a modification to get
\begin{center}
\adjustbox{scale=0.85}{\begin{tikzcd}
	{TA\times TB\times T^2C} & {T(A\times TB)\times T^2C} &[30pt] {T(A\times TB\times TC)} &[50pt] {T(A\times B\times C).}
	\arrow["{t_1\times 1}", from=1-1, to=1-2]
	\arrow[""{name=0, anchor=center, inner sep=0}, curve={height=-12pt}, from=1-2, to=1-3]
	\arrow[""{name=1, anchor=center, inner sep=0}, curve={height=12pt}, from=1-2, to=1-3]
	\arrow["{\mu \circ Tt_3\circ \mu\circ Tt_1\circ T(t_2\times 1)}", from=1-3, to=1-4]
	\arrow["{\Gamma^{-1}}", between={0.2}{0.8}, Rightarrow, from=0, to=1]
\end{tikzcd}}
\end{center}
By (\ref{axiomstrength2}) in \Cref{pseudocomm} this pasting becomes
\begin{center}
\begin{equation}\label{3.4.5}
\adjustbox{scale=0.85}{\begin{tikzcd}
	{TA\times TB\times T^2C} & {TA\times T(TB\times TC)} &[20pt] {T(A\times TB\times TC)} &[40pt] {T(A\times B\times C).}
	\arrow["{1\times t_2}", from=1-1, to=1-2]
	\arrow[""{name=0, anchor=center, inner sep=0}, curve={height=-12pt}, from=1-2, to=1-3]
	\arrow[""{name=1, anchor=center, inner sep=0}, curve={height=12pt}, from=1-2, to=1-3]
	\arrow["{\mu\circ Tt_3\circ \mu\circ Tt_1\circ T(t_2\times 1)}", from=1-3, to=1-4]
	\arrow["{\Gamma^{-1}}", between={0.2}{0.8}, Rightarrow, from=0, to=1]
\end{tikzcd}}
\end{equation}
\end{center}
Next, we decorate diagram (\ref{3.4.5}) with some commutative squares without altering the pasting. We are using that $\mu\circ Tt_3\circ\mu\circ Tt_1\circ T(t_2\times 1)$ is equal to
$\mu\circ Tt_2\circ T(1\times \mu )\circ T(1\times Tt_2)\circ T(1\times t_1)$ as 1-cells from $T(A\times TB\times TC)$ to $ T(A\times B\times C).$ 
\begin{center}
\adjustbox{scale=0.85}{
    \begin{tikzcd}
        TA\times TB\times T^2C\ar[r,"1\times t_2"]&[-15pt]TA\times T(TB\times TC)\ar[ddl,"1\times Tt_1"']\ar[r,"t_2"]\ar[d,"t_1"{swap}]&[-15pt]T(TA\times TB\times TC)\ar[r,"Tt_1"]&[-15pt]|[alias=D1]|T^2(A\times TB\times TC)\ar[d,"\mu"]\\
        
        &|[alias=R1]|T(A\times T(TB\times TC))\ar[d,"T(1\times Tt_1)"]\ar[r,"Tt_2"']&T^2(A\times TB\times TC)\ar[d,"T^2(1\times t_1)"]\ar[r,"\mu"']&T(A\times TB\times TC)\ar[d,"T(1\times t_1)"]\\

        TA\times T^2(B\times TC)\ar[d,"1\times T^2t_2"']\ar[r,"t_1"']&T(A\times T^2(B\times TC))\ar[d,"T(1\times T^2t_2)"]\ar[r,"Tt_2"']&T^2(A\times T(B\times TC))\ar[r,"\mu"']\ar[d,"T^2(1\times Tt_2)"]&T(A\times T(B\times TC)\ar[d,"T(1\times Tt_2)"]
        \\

        TA\times T^3(B\times C)\ar[d,"1\times T\mu"']\ar[r,"t_1"']&T(A\times T^2(B\times C))\ar[r,"Tt_2"']\ar[d,"T(1\times T\mu)"']&T^2(A\times T^2(B\times C))\ar[r,"\mu"']\ar[d,"T^2(1\times \mu)"']&T(A\times T^2(B\times C))\ar[d,"T(1\times \mu)"]\\
        
        TA\times T^2(B\times C)\ar[r,"t_1"']&T(A\times T(B\times C))\ar[r,"Tt_2"']&T^2(A\times T(B\times C))\ar[r,"\mu"']&T(A\times T(B\times C))\ar[r,"\mu \circ Tt_2"']&[-10pt]T(A\times B\times C).
        \arrow[Rightarrow,from=D1,to=R1," \Gamma^{-1}",shorten <=12mm,shorten >=12mm]
    \end{tikzcd}}
\end{center}
We then get \Cref{gamma3-1} since $\Gamma^{-1}$ is a modification. To finish the proof we just have to show that \Cref{gamma3-2} equals the following whiskering coming from \Cref{delta3}
\begin{center}
   \adjustbox{scale=0.85}{
\begin{tikzcd}
	{TA\times TB\times T^2C} &[-10pt] {T(A\times TB)\times T^2C} &[-10pt] {T^2(A\times B)\times T^2C} \\
	& {T(A\times TB\times T^2C)} & {T(T(A\times B)\times T^2 C)} & {T^2(A\times B\times TC)} &[-10pt] {T(A\times B\times C).}
	\arrow["{t_1\times 1}", from=1-1, to=1-2]
	\arrow["{{t_1}_{A,TB\times T^2C}}"', from=1-1, to=2-2]
	\arrow["{Tt_2\times 1}", from=1-2, to=1-3]
	\arrow["{t_1}", from=1-2, to=2-2]
	\arrow["{t_1}", from=1-3, to=2-3]
	\arrow["{T(t_2\times 1)}"', from=2-2, to=2-3]
	\arrow[""{name=0, anchor=center, inner sep=0}, curve={height=-12pt}, from=2-3, to=2-4]
	\arrow[""{name=1, anchor=center, inner sep=0}, curve={height=12pt}, from=2-3, to=2-4]
	\arrow["{\mu\circ Tt_3\circ \mu}"', from=2-4, to=2-5]
	\arrow["{T\Gamma^{-1}}", between={0.2}{0.8}, Rightarrow, from=0, to=1]
\end{tikzcd}}
\end{center}
We can apply (\ref{axiomstrength1}) in \Cref{pseudocomm} to get
\begin{center}
\adjustbox{scale=0.85}{
\begin{tikzcd}
	{TA\times TB\times T^2C} & {T(A\times TB\times T^2C)} &[40pt] {T(A\times T(B\times TC))} &[10pt] {T(A\times B\times C).}
	\arrow["{t_1}", from=1-1, to=1-2]
	\arrow[""{name=0, anchor=center, inner sep=0}, curve={height=-12pt}, from=1-2, to=1-3]
	\arrow[""{name=1, anchor=center, inner sep=0}, curve={height=12pt}, from=1-2, to=1-3]
	\arrow["{\mu \circ Tt_3\circ\mu\circ Tt_2}", from=1-3, to=1-4]
	\arrow["{T(1\times \Gamma^{-1})}", between={0.2}{0.8}, Rightarrow, from=0, to=1]
\end{tikzcd}}
\end{center}
Since $t_1$ is a 2-natural transformation, we get that the last pasting equals the whiskering
\begin{center}
    \adjustbox{scale=0.85}{
\begin{tikzcd}
	{TA\times TB\times T^2C} &[20pt] {TA\times T(B\times TC)} & {T(A\times T(B\times TC))} &[15pt] {T(A\times B\times C).}
	\arrow[""{name=0, anchor=center, inner sep=0}, curve={height=-12pt}, from=1-1, to=1-2]
	\arrow[""{name=1, anchor=center, inner sep=0}, curve={height=12pt}, from=1-1, to=1-2]
	\arrow["{t_1}", from=1-2, to=1-3]
	\arrow["{\mu\circ Tt_3\circ\mu\circ Tt_2}", from=1-3, to=1-4]
	\arrow["{1\times \Gamma^{-1}}"{pos=0.1}, between={0.2}{0.8}, Rightarrow, from=0, to=1]
\end{tikzcd}}
\end{center}
This equals (\ref{gamma3-2}) since $\mu\circ Tt_3\circ \mu \circ Tt_2\circ t_1=\mu\circ Tt_2 \circ t_1\circ (1\times \mu)\circ (1\times Tt_2)$ as 1-cells from $TA\times T(B\times TC)$ to $T(A\times B\times C).$
\end{proof}
\begin{lem} Let $(T,\eta,\mu,t,\Gamma)$ be a symmetric, pseudo commutative, strong 2-monad, then $\omega$ is 2-natural in the following sense: 

\begin{enumerate}
    \item For $f_1\colon A_1\to B_1$ and $f_2\colon A_2\to B_2$ in $\K,$ $$\gamma(T(f_1\times f_2);\omega_{A_1,A_2})=\gamma(\omega_{B_1,B_2};Tf_1,Tf_2).$$
    That is, the following multicategorical diagram commutes:
\begin{center}
    \begin{tikzcd}
        \la TA_1, TA_2\ra\ar[d,"\la Tf_1{,} Tf_2\ra"']\ar[r,"\omega"]&T(A_1\times A_2)\ar[d,"T(f_1\times f_2)"]\\
        \la TB_1,TB_2\ra \ar[r,"\omega"']&T(B_1\times B_2).
    \end{tikzcd}
\end{center} 
    \item For 2-cells $\alpha_1\colon f_1\to g_1$ in $\K(A_1,B_1)$ and $\alpha_2\colon f_2\to g_2$ in $\K(A_2,B_2)$ 
    $$\gamma(T(\alpha_1\times \alpha_2);1_{\omega_{A_1,A_2}})=\gamma(1_{\omega_{B_1,B_2}};T\alpha_1,T\alpha_2).$$
    That is, the multicategorical pasting
    \begin{center}
    \begin{tikzcd}
        \la TA_1, TA_2\ra\ar[d,bend left=50,"\la Tg_1{,} Tg_2\ra"name=R1]\ar[r,"\omega"]\ar[d,bend right=50,"\la Tf_1{,} Tf_2\ra"{swap,name=D1}]\ar[r,"\omega"]&T(A_1\times A_2)\ar[d,"T(g_1\times g_2)"]\\[20pt]
        \la TB_1,TB_2\ra\ar[r,"\omega"']&T(B_1\times B_2).
        \arrow[Rightarrow,from=D1,to=R1,"\la T\alpha_1 {,} T\alpha_2\ra",shorten <=2mm,shorten >=2mm]
    \end{tikzcd}
\end{center} 
equals
\begin{center}
    \begin{tikzcd}
        \la TA_1, TA_2\ra\ar[r,"\omega"]\ar[d,"\la Tf_1{,} Tf_2\ra"']&T(A_1\times A_2)\ar[d,bend right=60,"T(f_1\times f_2)"{swap,name=D1}]\ar[d,bend left=60,"T(g_1\times g_2)"name=R1]\\[20pt]
        \la TB_1,TB_2\ra \ar[r,"\omega"']&T(B_1\times B_2).
        \arrow[Rightarrow,from=D1,to=R1,"T(\alpha_1\times \alpha_2 )",shorten <=2mm,shorten >=2mm]
    \end{tikzcd}
\end{center} 
    \end{enumerate}
\end{lem}
\begin{proof}
For part $(1),$ the corresponding 1-cells of $\gamma(T(f_1\times f_2);\omega)$ and $\gamma(\omega;Tf_1,Tf_2)$ are equal since $\omega\colon TA_1\times TA_2\to T(A_1\times A_2)$ equals $\mu\circ Tt_2\circ t_1,$ a composition of 2-natural transformations.  The 1-cells $\gamma(T(f_1\times f_2);\omega)_1$ and $\gamma(\omega;Tf_1,Tf_2)_1$ are equal since both are identity 1-cells, with $Tf_1$ and $Tf_2$ being strict maps of $T$-algebras.  Let's show that $\gamma(T(f_1\times f_2);\omega)_2=\gamma(\omega;Tf_1,Tf_2)_2.$ By a double application of \Cref{pseudocomm}, $\gamma(\omega;Tf_1,Tf_2)_2$ equals
\begin{center}
    \adjustbox{scale=0.9}{
\begin{tikzcd}
	{TA_1\times T^2A_2} & {TB_1\times T^2B_2} & {T(B_1\times TB_2)} & {T(B_1\times B_2).}
	\arrow["{Tf_1\times T^2f_2}", from=1-1, to=1-2]
	\arrow[""{name=0, anchor=center, inner sep=0}, curve={height=-12pt}, from=1-2, to=1-3]
	\arrow[""{name=1, anchor=center, inner sep=0}, curve={height=12pt}, from=1-2, to=1-3]
	\arrow["{\mu\circ Tt_2}", from=1-3, to=1-4]
	\arrow["{\Gamma^{-1}}", between={0.2}{0.8}, Rightarrow, from=0, to=1]
\end{tikzcd}}
\end{center}
Since $\Gamma^{-1}$ is a modification, and because $\mu$ and $Tt_2$ are 2-natural this whiskering can be writen as
\begin{center}
    \adjustbox{scale=0.9}{
\begin{tikzcd}
	{TA_1\times T^2A_2} && {T(A_1\times TA_2)} & {T(A_1\times A_2)} & {T(B_1\times B_2).}
	\arrow[""{name=0, anchor=center, inner sep=0}, curve={height=-12pt}, from=1-1, to=1-3]
	\arrow[""{name=1, anchor=center, inner sep=0}, curve={height=12pt}, from=1-1, to=1-3]
	\arrow["{\mu \circ Tt_2}", from=1-3, to=1-4]
	\arrow["{T(f_1\times f_2)}", from=1-4, to=1-5]
	\arrow["{\Gamma^{-1}}", between={0.2}{0.8}, Rightarrow, from=0, to=1]
\end{tikzcd}}
\end{center}
An application of \Cref{omega2} gives us that this is exactly $\gamma(T(f_1\times f_2),\omega)_2.$ Part $(2)$ follows from the 2-naturality of $t_1,Tt_2$ and $\mu$.
\end{proof}
\begin{defi}
    Let $(T,\eta,\mu,t,\Gamma)$ be a symmetric, pseudo commutative, strong 2-monad. For $A_1,\dots,A_n\in \ob (\K ),$ we define $$\omega_n=\omega_{A_1,\dots, A_n}\colon \la TA_1, \dots, TA_n\ra \to T(A_1\times \cdots\times A_n)$$
    in $\talg$ by recursion in the following way:
    \begin{itemize}
        \item For $n=0,$ $\omega_0\colon 1\to T1$ is $\eta_1\colon 1\to T1.$
        \item For $n=1$ $\omega_1\colon  TA_1\to TA_1$ is the identity $1_{TA_1}.$
        \item For $n=2,$ $\omega_2$ is $\omega_{A_1,A_2}\colon \la TA_1,TA_2\ra\to T(A_1\times TA_2)$ from \Cref{omegais1cell}.
        \item For $n\geq 3$ $\omega_n=\gamma(\omega_2;\omega_{n-1},\omega_1).$
    \end{itemize}
\end{defi}
\begin{cor}
Let $(T,\eta,\mu,t,\Gamma)$ be a symmetric, pseudo commutative, strong 2-monad, and $A_1,\dots , A_n$ objects of $\K.$ For $n\geq 3,$
$$\omega_n=\gamma(\omega_2;\omega_{n-1},\omega_1)=\gamma(\omega_2;\omega_1,\omega_{n-1}).$$
\end{cor}
It follows by a straightforward induction that $\omega_n$ is natural in the following sense. 
\begin{lem}\label{omeganatural}Let $(T,\eta,\mu,t,\Gamma)$ be a symmetric, pseudo commutative, strong 2-monad. For any $n,$ $\omega_n$ is natural in the following sense:
\begin{enumerate}
    \item[(1)]Suppose $f_i\colon A_i\to B_i$ are 1-cells in $\K$ for $1\leq i\leq n,$ then $$\gamma(T(f_1\times\cdots \times f_n);\omega_{A_1,\dots,A_n})=\gamma(\omega_{B_1,\dots ,B_n};Tf_1,\dots,Tf_n).$$ That is, the following multicategorical diagram commutes:
\begin{center}
    \begin{tikzcd}
        \la TA_1,\dots, TA_n\ra\ar[d,"\la Tf_1{,}\dots{,} Tf_n \ra"']\ar[r,"\omega"]&T(A_1\times\cdots\times A_n)\ar[d,"T(f_1\times\cdots \times f_n)"]\\
        \la TB_1,\dots,TB_n \ra\ar[r,"\omega"']&T(B_1\times \cdots\times B_n).
    \end{tikzcd}
\end{center} 
    \item[(2)] Suppose $\alpha_i\colon f_i\to g_i$ are 2-cells in $\K(A_i,B_i)$ for $1\leq i \leq n$ Then 
    $$\gamma(T(\alpha_1\times \cdots\times\alpha_n);1_{\omega_{A_1,\dots,A_n}})=\gamma(1_{\omega_{B_1,\dots,B_n}};T\alpha_1,\dots,T\alpha_n).$$
    That is, the multicategorical whiskering
    \begin{center}
\begin{tikzcd}
	{\la TA_1,TA_2\ra} & {T(A_1\times A_2)} \\[10pt]
	{\la TB_1,TB_2\ra} & {T(B_1\times B_2).}
	\arrow["\omega", from=1-1, to=1-2]
	\arrow[""{name=0, anchor=center, inner sep=0}, "{\la Tg_1,Tg_2\ra}", shift left=5, curve={height=-6pt}, from=1-1, to=2-1]
	\arrow[""{name=1, anchor=center, inner sep=0}, "{\la Tf_1,Tf_2\ra}"', shift right=5, curve={height=6pt}, from=1-1, to=2-1]
	\arrow["{T(g_1\times g_2)}", from=1-2, to=2-2]
	\arrow["\omega"', from=2-1, to=2-2]
	\arrow["{\la T\alpha_1,T\alpha_2\ra}", between={0.2}{0.8}, Rightarrow, from=1, to=0]
\end{tikzcd}
\end{center} 
equals the whiskering
\begin{center}
    \begin{tikzcd}
        \la TA_1,\dots, TA_n\ra\ar[r,"\omega"]\ar[d,"\la Tf_1{,} \dots{,}Tf_n\ra"']&T(A_1\times\cdots\times A_n)\ar[d,bend right=65,"T(f_1\times \cdots\times f_n)"{swap,name=D1}]\ar[d,bend left=65,"T(g_1\times\cdots\times g_n)"name=R1]\\[20pt]
        \la TB_1,\dots,TB_n\ra\ar[r,"\omega"']&T(B_1\times \cdots\times B_n).
        \arrow[Rightarrow,from=D1,to=R1,"T(\alpha_1\times \cdots\times \alpha_n )",shorten <=2mm,shorten >=2mm]
    \end{tikzcd}
\end{center} 
\end{enumerate}
\end{lem}
Next, we define the free algebra $\Cat$-multifunctor $T:\K\to \talg$.
\begin{defi}
    Let $(T,\eta,\mu,t,\Gamma )$ be a symmetric, pseudo commutative, strong 2-monad. We define the multifunctor $T\colon \K\to \talg$ as follows:
    \begin{itemize}
        \item $T$ is already defined on objects and since  $(TA,\mu:T^2A\to TA)$ is a $T$-algebra for $A\in\ob(\K).$
        \item For $n=0,$ $T\colon \K (1,A)\to \talg(1,TA)$ is defined as the composition
        \begin{center}
\begin{tikzcd}
	{\K (1,A)} & {\K (T1, TA)} & {\talg (1,TA),}
	\arrow["T", from=1-1, to=1-2]
	\arrow["T"', curve={height=12pt}, from=1-1, to=1-3]
	\arrow["{\eta_1*}", from=1-2, to=1-3]
\end{tikzcd}
            \end{center}
        where $\eta_1\colon 1\to T1,$ and $\talg (1,TA)=\K (1,TA).$
        \item For $n=1,$ we define $T\colon \K(A,B)\to \talg (TA,TB)$ as sending $f\colon A\to B$ to $TF\colon TA\to TB$ with $(Tf)_1$ being an identity. Similarly, a 2-cell $\alpha\colon f\to g$ in $\K (A,B)$ is sent to $T\alpha$.
        \item For $n\geq 2,$ we define $T\colon \K(A_1\times\cdots\times A_n,B)\to \talg(TA_1,\cdots,TA_n;TB)$ as the composition
        \begin{center}
\adjustbox{scale=0.9}{\begin{tikzcd}
	{\K(A_1\times \cdots\times A_n,B)} & {\talg (T(A_1\times\cdots\times A_n);TB) } & {\talg (TA_1,\ldots, TA_n;TB).}
	\arrow["T", from=1-1, to=1-2]
	\arrow["T"', curve={height=12pt}, from=1-1, to=1-3]
	\arrow["{\omega_n^*}", from=1-2, to=1-3]
\end{tikzcd}}
            \end{center}
    \end{itemize}
\end{defi}
\begin{thm}\label{maintheoremchapter1}
Let $(T,\eta,\mu,t,\Gamma)$ be a symmetric, pseudo commutative, strong 2-monad. Then $T\colon \K\to \talg$ is a non-symmetric $\Cat$-multifunctor.
\end{thm}
\begin{proof}
It is clear from the definition that $T$ preserves identities. Preservation of $\gamma$ by $T$ follows at once from \Cref{omeganatural} and \Cref{omegaass}.
\end{proof}
\section{Pseudo symmetry of the free algebra multifunctor}\label{Section3.3}
Next, we will define the pseudo symmetry isomorphisms. We do this in a recursive way, starting with the non trivial element of $\Sigma_2$. From here on $\sigma_i$ will denote the transposition in $\Sigma_n$ that permutes $i$ and $i+1.$
\begin{defi}
    Let $(T,\eta,\mu,t,\Gamma)$ be a symmetric, pseudo commutative, strong 2-monad. For $A,B\in \K$ we define $\omega'\colon \la TA,TB\ra\to T(A\times B)$ as the image of $\omega$ trhough the composition
    \begin{center}
        \begin{tikzcd}[font=\small]
            \talg(TB,TA;T(B\times A))\ar[r,"\sigma_1"]&[-17pt]\talg (TA,TB;T(B\times A))\ar[r,"T\cong_*"]&[-17pt]\talg (TA,TB;T(A\times B)).
        \end{tikzcd}
        \end{center}
\end{defi}
\begin{lem}\label{omega'}
    For $(T,\eta,\mu,t,\Gamma)$ a symmetric, pseudo commutative, strong 2-monad and $\omega'$ as in the previous definition, its component is 
        \begin{center}
            \begin{tikzcd}
                TA\times TB\ar[r,"t_2"]&T(TA\times TB)\ar[r,"Tt_1"]&T^2(A\times B)\ar[r,"\mu"]&T(A\times B).
            \end{tikzcd}
        \end{center}
        The 2-cell $\omega'_1$ equals
        \begin{center}
\begin{tikzcd}
	{T^2A\times TB} & {T(TA\times B)} & {T^2(A\times B)} & {T(A\times B),}
	\arrow[""{name=0, anchor=center, inner sep=0}, curve={height=-12pt}, from=1-1, to=1-2]
	\arrow[""{name=1, anchor=center, inner sep=0}, curve={height=12pt}, from=1-1, to=1-2]
	\arrow["{Tt_1}", from=1-2, to=1-3]
	\arrow["\mu", from=1-3, to=1-4]
	\arrow["\Gamma", between={0.2}{0.8}, Rightarrow, from=0, to=1]
\end{tikzcd}
        \end{center}
        while $\omega '_2$ is an identity 1-cell.
\end{lem}
\begin{proof}
Since $\omega_1$ is an identity 2-cell we get that $\omega'_2$ is as well. On the other hand, by \Cref{omega2}, $\omega'_1$ can be written as
\begin{center}
\adjustbox{scale=0.9}{\begin{tikzcd}
	{T^2A\times TB} & {TB\times T^2A} &[15pt] {T(B\times TA)} &[10pt] {T(A\times B).}
	\arrow["\cong", from=1-1, to=1-2]
	\arrow[""{name=0, anchor=center, inner sep=0}, curve={height=-12pt}, from=1-2, to=1-3]
	\arrow[""{name=1, anchor=center, inner sep=0}, curve={height=12pt}, from=1-2, to=1-3]
	\arrow["{T\cong\circ \mu \circ Tt_2}", from=1-3, to=1-4]
	\arrow["{\Gamma^{-1}}", between={0.2}{0.8}, Rightarrow, from=0, to=1]
\end{tikzcd}}
\end{center}
By naturality of $\mu$ and definition of $t_1$ we can write this whiskering as
\begin{center}\adjustbox{scale=0.9}{\begin{tikzcd}
	{T^2A\times TB} & {TB\times T^2A} & {T(B\times TA)} & {T(TA\times B)} & {T(A\times B).}
	\arrow["\cong", from=1-1, to=1-2]
	\arrow[""{name=0, anchor=center, inner sep=0}, curve={height=-12pt}, from=1-2, to=1-3]
	\arrow[""{name=1, anchor=center, inner sep=0}, curve={height=12pt}, from=1-2, to=1-3]
	\arrow["{T\cong}", from=1-3, to=1-4]
	\arrow["{\mu\circ Tt_1}", from=1-4, to=1-5]
	\arrow["{\Gamma^{-1}}", between={0.2}{0.8}, Rightarrow, from=0, to=1]
\end{tikzcd}}
\end{center}
By \Cref{Tsymmetric} we have that $\omega'_1$ agrees with the whiskering in the statment of the lemma.
\end{proof}
\begin{lem} Let $(T,\eta,\mu,t,\Gamma)$ be a symmetric, pseudo commutative, strong 2-monad. For $A,B$ objects of $\K$, there is a 2-cell
\begin{center}
\begin{tikzcd}
	{\la TA,TB\ra} &[10pt] {T(A\times B),}
	\arrow[""{name=0, anchor=center, inner sep=0}, "\omega", curve={height=-12pt}, from=1-1, to=1-2]
	\arrow[""{name=1, anchor=center, inner sep=0}, "{\omega '}"', curve={height=12pt}, from=1-1, to=1-2]
	\arrow["{\Gamma_{A,B}}", between={0.2}{0.8}, Rightarrow, from=0, to=1]
\end{tikzcd}
    \end{center}
in the multicategory $\talg$ with component 2-cell $\Gamma_{A,B}.$
\end{lem}
\begin{proof}
We need to prove that \Cref{2cell1,2cell2} for $\Gamma_{A,B}$ are equal for $i=1,2.$ For $i=1$ \Cref{2cell1} takes the form
\begin{center}
\adjustbox{scale=0.9}{
\begin{tikzcd}
	{T^2A\times TB} & {TA\times TB} &[10pt] {T(A\times B)}
	\arrow["{\mu\times 1}", from=1-1, to=1-2]
	\arrow[""{name=0, anchor=center, inner sep=0}, curve={height=-12pt}, from=1-2, to=1-3]
	\arrow[""{name=1, anchor=center, inner sep=0}, curve={height=12pt}, from=1-2, to=1-3]
	\arrow["\Gamma", between={0.2}{0.8}, Rightarrow, from=0, to=1]
\end{tikzcd}}
\end{center}
Now, by \Cref{omega'},  \Cref{2cell2} for $i=1$ agrees exactly with \Cref{pseudomu1} and we are done by (\ref{axiommu2gamma}) in \Cref{pseudocomm}. For $i=2$ \Cref{2cell1} is, by an application of (\ref{axiommu1gamma}) in \Cref{pseudocomm},  and \Cref{omega2},
\begin{center}
\adjustbox{scale=0.9}{
            \begin{tikzcd}
                TA\times T^2B\ar[d,"t_2"']\ar[rrd,bend left=20, "\omega'_{A,B}"name=D1]\ar[rrd,"\omega_{A,B}"{swap,name=R1}]&&\\
                |[alias=R3]|T(TA\times TB)\ar[d,"Tt_2"']\ar[r,"Tt_1"']&|[alias=D2]|T^2(A\times TB)\ar[d,"T^2 t_2"]\ar[r,"\mu"']&T(A\times TB)\ar[d,"T t_2"]\\
                T^2(TA\times B)\ar[d,"T^2t_1"']&T^3(A\times B)\ar[d,"T\mu"]\ar[r,"\mu"]&T^2(A\times B)\ar[d,"\mu"]\\
                |[alias=R2]|T^3(A\times B)\ar[r,"T\mu"']&T^2(A\times B)\ar[r,"\mu"']&T(A\times B).
                \arrow[Rightarrow,from=D1,to=R1,"\Gamma_{A,TB}^{-1}"',shorten <=3mm,shorten >=3mm]
                \arrow[Rightarrow,from=D2,to=R2,"T\Gamma_{A,B}",shorten=12mm]
                \arrow[Rightarrow,from=R1,to=R3,"\Gamma_{A,TB}"',shorten <=2mm,shorten >=2mm]
            \end{tikzcd}}
        \end{center}
The 2-cells $\Gamma_{A,TB}$ and its inverse cancel out to give \Cref{2cell2} for $i=2.$
\end{proof}
Next, we define $T_{\sigma_i}$ for $\sigma_i\in \Sigma_2.$ From now on we will use the notation in the \Cref{appendix}, specially in \Cref{pseudosymm}. We startby defining $T_{\sigma_1}$ for $\sigma_1\in \Sigma_2$. 
\begin{defi}
 Let $A,B$ and $C$ be objects of $\K.$ We define the natural transformation
\begin{center}
    \begin{tikzcd}
        \K(A\times B,C)\ar[r,"T"]\ar[d,"\sigma_1"']&|[alias=R1]|\talg (TA,TB;TC)\ar[d,"\sigma_1"]\\
        |[alias=D1]|\K(B\times A,C)\ar[r,"T"{swap}]&\talg (TB,TA;C),
        \arrow[Rightarrow,from=D1,to=R1,"T_{\sigma_1}",shorten <=5mm,shorten >=4mm]
    \end{tikzcd}
\end{center}
as having component $T_{\sigma_1;f}$ for $f\colon A\times B\to C,$ the whiskering in the multicategory $\talg$
\begin{center}
\begin{tikzcd}
	{\la TB,TA\ra} &[15pt] {T(B\times A)} & {T(A\times B)} & TC
	\arrow[""{name=0, anchor=center, inner sep=0}, "\omega", curve={height=-12pt}, from=1-1, to=1-2]
	\arrow[""{name=1, anchor=center, inner sep=0}, "{\omega '}"', curve={height=12pt}, from=1-1, to=1-2]
	\arrow["{T\cong}", from=1-2, to=1-3]
	\arrow["Tf", from=1-3, to=1-4]
	\arrow["{\Gamma_{B,A}}", between={0.2}{0.8}, Rightarrow, from=0, to=1]
\end{tikzcd}
\end{center}
The fact that $T_{\sigma_1}$ is in fact a natural transformation follows from the exchange property in the 2-category $\K$.
\end{defi}
\begin{defi}\label{pseudosymmisos}
Let $(T,\eta,\mu,t,\Gamma)$ be a symmetric, pseudo commutative, strong 2-monad, and $\sigma_i\in\Sigma_n.$ the transposition that interchanges $i$ and $i+1$ in $\Sigma_n$ for $n\geq 3.$ We define the natural transformation $T_{\sigma_i}$
\begin{center}
    \begin{tikzcd}[font=\small]
        \K(A_1\times\cdots\times A_n,C)\ar[r,"T"]\ar[d,"\sigma_i"']&[-10pt]|[alias=R1]|\talg (TA_1,\dots, TA_n;TC)\ar[d,"\sigma_i"]\\
        |[alias=D1]|\K(A_1\times \cdots\times A_{i+1}\times A_i\times \cdots\times A_n,C)\ar[r,"T"{swap}]&\talg (TA_1,\dots,TA_{i+1},TA_{i},\dots,TA_n;C)
        \arrow[Rightarrow,from=D1,to=R1,"T_{\sigma_i}",shorten <=8mm,shorten >=9mm]
    \end{tikzcd}
\end{center}
as follows. For $f\colon A_1\times \cdots\times A_n\to C$ the 2-cell $T_{\sigma_i ;f}$ is 
\begin{center}
\adjustbox{scale=0.9}{
    \begin{tikzcd}
    TA_1\times\cdots\times TA_{i+1}\times TA_i\times \cdots\times TA_n\ar[r,bend left=20,"\omega \times\omega\times\omega"name=D1]\ar[r,bend right=15,"\omega\times\omega'\times\omega"{swap,name=R1}]&[20pt]T(A_1\times \cdots\times A_{i-2})\times T(A_{i+1}\times A_i)\times T(A_{i+2}\times\cdots\times A_n)\ar[d,"\omega"]\\[-12pt]
    & T(A_1\times \cdots\times A_{i+1}\times A_i\times \cdots \times  A_n)\ar[d,"T\cong"]\\[-12pt]
    &T(A_1\times \cdots\times A_i\times A_{i+1}\times \cdots \times A_n)\ar[d,"Tf"]\\[-12pt]
    &TC
    \arrow[Rightarrow,near start,from=D1,to=R1,"1\times\Gamma_{A_{i+1},A_i}\times 1"',shorten <=5mm,shorten >=5mm]
    \end{tikzcd}}
\end{center}
The fact that this is well defined comes from the associativity of  $\omega$ (\Cref{omegaass}), and the fact that $T_{\sigma_i}$ is in fact a natural transformation follows from the exchange rule in $\K$.
\end{defi}
Next, we prove that this defines $T_{\sigma}$ for every $\sigma\in \Sigma_n$ and every $n$ by using that the symmetric group $\Sigma_n$ is generated by the consecutive transpositions $\sigma_1,\dots,\sigma_{n-1}$ subject to the relations:
\begin{itemize}
    \item[(a)] $\sigma_i^2=\id$
    \item[(b)] $\sigma_i\sigma_j=\sigma_j\sigma_i$ for $|i-j|>1.$
    \item[(c)] $\sigma_i\sigma_{i+1}\sigma_i=\sigma_{i+1}\sigma_i\sigma_{i+1}.$
\end{itemize}
The relations between the different $T_{\sigma_i}$ will follow from relations between 2-cells in $\talg$ which can be proven in $\K.$ The relations in $\K$ can be proven even when $T$ is not symmetric, except for the relation induced by $\sigma_i\sigma_i=\id.$ The following follows (in a way, it is equivalent to) symmetry for $T.$
\begin{lem}
Suppose that $(T,\eta,\mu,t,\Gamma )$ is a symmetric, pseudo commutative, strong 2-monad. Then the following pasting diagram is the identity:
\begin{center}
\adjustbox{scale=0.9}{
    \begin{tikzcd}
        \K(A_1\times\cdots\times \times A_n,C)\ar[r,"T"]\ar[d,"\sigma_i"']&[-15pt]|[alias=R1]|\talg (TA_1,\dots, TA_n;TC)\ar[d,"\sigma_i"]\\
        |[alias=D1]|\K(A_1\times \cdots\times A_{i+1}\times A_i\times \cdots\times A_n,C)\ar[r,"T"{swap}]\ar[d,"\sigma_{i}"']&|[alias=R2]|\talg (TA_1,\dots,TA_{i+1},TA_{i},\dots,TA_n;C)\ar[d,"\sigma_{i}"]\\
         |[alias=D2]|\K(A_1\times\cdots\times A_n,C)\ar[r,"T"']&\talg (TA_1,\dots, TA_n;TC).
        \arrow[Rightarrow,from=D1,to=R1,"T_{\sigma_i}",shorten <=11mm,shorten >=12mm]
        \arrow[Rightarrow,from=D2,to=R2,"T_{\sigma_i}"',shorten <=11mm,shorten >=12mm]
    \end{tikzcd}}
\end{center}
\end{lem}
The following holds in the absence of symmetry.
\begin{lem}\label{yangbaxter2}
Suppose that $(T,\eta,\mu,t,\Gamma)$ is a pseudo commutative, strong 2-monad. Then the pasting diagram
\begin{center}
\adjustbox{scale=0.9}{
    \begin{tikzcd}
    TA_1\times TA_2\times TA_3\times TA_4\ar[d,"1\times \cong"']\ar[r,"1\times \omega"]&|[alias=D1]|TA_1\times TA_2\times T(A_3\times A_4)\ar[r,"\omega"]&T(A_1\times A_2\times A_3\times A_4)\\
    |[alias=R1]|TA_1\times TA_2\times TA_4\times T A_3\ar[dd,"\cong\times 1"']\ar[rd,"\omega\times 1"]\ar[r,"1\times \omega"]&TA_1\times TA_2\times T(A_4\times A_3)\ar[r,"\omega"]\ar[u,"1 \times T\cong"']&T(A_1\times A_2\times A_4\times A_3)\ar[u,"T\cong"']\\[-10pt]
    &|[alias=D2]|T(A_1\times A_2)\times TA_4\times TA_3\ar[ru,"\omega"]&\\
    |[alias=R2]|TA_2\times TA_1\times TA_4\times TA_3\ar[r,"\omega\times 1"']&T(A_2\times A_1)\times TA_4\times TA_3\ar[r,"\omega"]\ar[u,"T\cong\times 1"]&T(A_2\times A_1\times A_4\times A_3),\ar[uu,"T\cong"']
    \arrow[Rightarrow,from=D1,to=R1,"1\times\Gamma",shorten <=9mm,shorten >=12mm]
        \arrow[Rightarrow,from=D2,to=R2," \Gamma\times 1",shorten <=9mm,shorten >=12mm]
    \end{tikzcd}}
\end{center}
equals the pasting
\begin{center}
\adjustbox{scale=0.9}{
    \begin{tikzcd}
    TA_1\times TA_2\times TA_3\times TA_4\ar[r,"\omega\times 1"]\ar[d,"\cong\times 1"']&|[alias=D1]|T(A_1\times TA_2)\times TA_3\times T A_4\ar[r,"\omega"]&T(A_1\times A_2\times A_3\times A_4)\\
    |[alias=R1]|TA_2\times TA_1\times TA_3\times TA_4\ar[rd,"1\times \omega"']\ar[r,"\omega\times 1"']\ar[dd,"1\times \cong"']&T(A_2\times A_1)\times TA_3\times TA_4\ar[r,"\omega"]\ar[u,"T\cong \times 1"']&T(A_2\times A_1\times A_3\times A_4)\ar[u,"T\cong"']\\[-10pt]
    &|[alias=D2]|TA_2\times TA_1\times T(A_3\times A_4)\ar[ur,"\omega"']&\\
    |[alias=R2]|TA_2\times TA_1\times TA_4\times TA_3\ar[r,"1\times \omega"']&TA_2\times TA_1\times T(A_4\times A_3)\ar[r,"\omega"']\ar[u,"1\times T\cong"]&T(A_2\times A_1\times A_4\times A_3).\ar[uu,"T\cong"']
    \arrow[Rightarrow,from=D1,to=R1,"\Gamma\times 1",shorten <=9mm,shorten >=12mm]
        \arrow[Rightarrow,from=D2,to=R2,"1\times \Gamma",shorten <=9mm,shorten >=12mm]
    \end{tikzcd}}
\end{center}
\end{lem}
\begin{proof}
    Both pastings are equal to the pasting
    \begin{center}
    \adjustbox{scale=0.9}{
        \begin{tikzcd}
            TA_1\times TA_2\times TA_3\times TA_4\ar[d,"\cong\times \cong"']\ar[r,"\omega\times \omega"]&|[alias=D1]|T(A_1\times A_2)\times T(A_3\times A_4)\ar[r,"\omega"]&T(A_1\times A_2\times A_3\times A_4)\\
            |[alias=R1]|TA_2\times TA_1\times TA_4\times TA_3\ar[r,"\omega\times\omega"']&T(A_2\times A_1)\times T(A_4\times A_3)\ar[r,"\omega"']\ar[u,"T\cong\times T\cong"']&T(A_1\times A_2\times A_3\times A_4).\ar[u,"T\cong"]
            \arrow[Rightarrow,from=D1,to=R1,"\Gamma_{A_1,A_2}\times \Gamma_{A_3,A_4}"',shorten <=9mm,shorten >=12mm]
        \end{tikzcd}}
    \end{center}   
\end{proof}
When $T$ is symmetric, a slight generalization of the previous lemma can be interpreted as follows. To save space we will write $\overline{TA}=TA_1\times \cdots\times TA_n$ and $\overline{TA}\sigma=TA_{\sigma(1)}\times \cdots\times TA_{\sigma(n)}$ when $\sigma\in \Sigma_n$ and $A_1,\dots,A_n$ are objects of $\K .$
\begin{lem}
Suppose that $(T,\eta,\mu,t,\Gamma)$ is a symmetric, pseudo commutative, strong 2-monad, $n\geq 3$ and $1\leq i<i+2\leq j\leq n-1.$ Let $A_1,\dots,A_n,C$ be objects of $\K.$  Then, the pasting
    \begin{center}
        \begin{tikzcd}
            \K (\overline{TA},C)\ar[r,"T"]\ar[d,"\sigma_i"']&|[alias=D1]|\talg (\la TA\ra;C)\ar[d,"\sigma_i"]\\
            |[alias=R1]|\K (\overline{TA}\sigma_i,C)\ar[d,"\sigma_j"']\ar[r,"T"]&|[alias=D2]|\talg (\la TA\ra\sigma_i;C)\ar[d,"\sigma_j"]\\
            |[alias=R2]|\K (\overline{TA}\sigma_i\sigma_j,C)\ar[r,"T"']&\talg (\la TA\ra\sigma_i\sigma_j;C),

            \arrow[Rightarrow,from=D1,to=R1,"T\sigma_i"',shorten <=5mm,shorten >=5mm]
            \arrow[Rightarrow,from=D2,to=R2,"T\sigma_j"',shorten <=5mm,shorten >=5mm]
        \end{tikzcd} 
    \end{center}
    equals the pasting
    \begin{center}
        \begin{tikzcd}
            \K (\overline{TA},C)\ar[r,"T"]\ar[d,"\sigma_j"']&|[alias=D1]|\talg (\la TA\ra;C)\ar[d,"\sigma_j"]\\
            |[alias=R1]|\K (\overline{TA}\sigma_j,C)\ar[d,"\sigma_i"']\ar[r,"T"]&|[alias=D2]|\talg (\la TA\ra\sigma_j;C)\ar[d,"\sigma_i"]\\
            |[alias=R2]|\K (\overline{TA}\sigma_j\sigma_i,C)\ar[r,"T"']&\talg (\la TA\ra\sigma_j\sigma_i;C).

            \arrow[Rightarrow,from=D1,to=R1,"T\sigma_j"',shorten <=5mm,shorten >=5mm]
            \arrow[Rightarrow,from=D2,to=R2,"T\sigma_i"',shorten <=5mm,shorten >=5mm]
        \end{tikzcd} 
    \end{center}
\end{lem}
Next we focus on the Yang-Banxter equation. First we prove the following lemma that we will also need later. We don't require symmetry.
\begin{lem}\label{newlemma}
Let $(T,\eta,\mu,t,\Gamma)$ be a pseudo commutative, strong 2-monad, and $A_1,A_2,A_3$ objects of $\K$. Then, 
\begin{enumerate}
    \item\label{newlem10} The pasting diagram
\begin{center}
\begin{equation}\label{newlem1}
    \adjustbox{scale=0.9}{\begin{tikzcd}
    TA_1\times TA_2\times TA_3\ar[d,"\cong\times 1"']\ar[r," \omega\times 1"]&|[alias=D1]|T(A_1\times A_2)\times TA_3\ar[r,"\omega"]&T(A_1\times A_2\times A_3)\\
    |[alias=R1]|TA_2\times TA_1\times TA_3\ar[rd,"1\times \omega"']\ar[r,"\omega\times 1"']\ar[dd,"1\times \cong"']&T(A_2\times A_1)\times TA_3\ar[r,"\omega"']\ar[u,"T(\cong\times 1)"]&T(A_2\times A_1\times A_3)\ar[u,"T(\cong\times 1)"']\\[-10pt]
    &|[alias=D2]|TA_2\times T(A_1\times A_3)\ar[ru,"\omega"']\\
    |[alias=R2]|TA_2\times TA_3\times TA_1\ar[r,"1\times \omega"']&TA_2\times T(A_3\times A_1)\ar[r,"\omega"']\ar[u,"1\times T\cong"']&T(A_2\times A_3\times A_1)\ar[uu,"T(1\times \cong )"']
    \arrow[Rightarrow,from=D1,to=R1,"\Gamma\times 1"',shorten <=6mm,shorten >=7mm]
    \arrow[Rightarrow,from=D2,to=R2,"1\times \Gamma"',shorten <=5mm,shorten >=7mm]
    \end{tikzcd}}
\end{equation}
\end{center}
equals the whiskering
\begin{center}
\begin{tikzcd}
	{TA_1\times TA_2\times TA_3} & {TA_1\times T(A_2\times A_3)} & {T(A_1\times A_2\times A_3).}
	\arrow["{1\times \omega}", from=1-1, to=1-2]
	\arrow[""{name=0, anchor=center, inner sep=0}, "{\omega }", curve={height=-12pt}, from=1-2, to=1-3]
	\arrow[""{name=1, anchor=center, inner sep=0}, "{\omega '}"', curve={height=12pt}, from=1-2, to=1-3]
	\arrow["\Gamma", between={0.2}{0.8}, Rightarrow, from=0, to=1]
\end{tikzcd}
\end{center}
\item\label{newlem9} The pasting diagram
\begin{center}
\begin{equation}\label{newlem2}
    \adjustbox{scale=0.9}{\begin{tikzcd}
    TA_1\times TA_2\times TA_3\ar[r,"1\times \omega"]\ar[d,"1\times \cong"']&|[alias=D1]|TA_1\times T(A_2\times A_3)\ar[r,"\omega"]&T(A_1\times A_2\times A_3)\\
    |[alias=R1]|TA_1\times TA_3\times TA_2\ar[rd,"\omega\times 1"']\ar[dd,"\cong\times 1"']\ar[r,"1\times \omega"']&TA_1\times T(A_3\times A_2)\ar[r,"\omega"']\ar[u,"1\times T\cong"']&T(A_1\times A_3\times A_2)\ar[u,"T(1\times \cong )"']\\[-10pt]
    &|[alias=D2]|T(A_1\times A_3)\times TA_2\ar[ru,"\omega"']\\
    |[alias=R2]|TA_3\times TA_1\times TA_2\ar[r,"\omega\times 1"']&T(A_3\times A_1)\times TA_2\ar[u," T\cong\times 1"']\ar[r,"\omega"']&T(A_3\times A_1\times A_2)\ar[uu,"T(\cong\times 1)"']
    \arrow[Rightarrow,from=D1,to=R1,"1\times \Gamma"',shorten <=3mm,shorten >=5mm]
    \arrow[Rightarrow,from=D2,to=R2,"\Gamma\times 1"',shorten <=3mm,shorten >=5mm]
    \end{tikzcd}}
\end{equation}
\end{center}
equals the whiskering
\begin{center}
\begin{tikzcd}
	{TA_1\times TA_2\times TA_3} & {T(A_1\times A_2)\times TA_3} &[15pt] {T(A_1\times A_2\times A_3).}
	\arrow["{\omega\times 1}", from=1-1, to=1-2]
	\arrow[""{name=0, anchor=center, inner sep=0}, "\omega", curve={height=-12pt}, from=1-2, to=1-3]
	\arrow[""{name=1, anchor=center, inner sep=0}, "{\omega '}"', curve={height=12pt}, from=1-2, to=1-3]
	\arrow["\Gamma", between={0.2}{0.8}, Rightarrow, from=0, to=1]
\end{tikzcd}
\end{center}
\end{enumerate}
\end{lem}
\begin{proof}
For part (1) we start from 

\begin{center}
\adjustbox{scale=0.9}{
\begin{tikzcd}
	{TA_1\times TA_2\times TA_3} && {TA_1\times T^2(A_2\times A_3)} & {TA_1\times T(A_2\times A_3)} & {T(A_1\times A_2\times A_3)}
	\arrow["{(1\times Tt_2)\circ (1\times t_1)}", from=1-1, to=1-3]
	\arrow["{1\times \mu}", from=1-3, to=1-4]
	\arrow[""{name=0, anchor=center, inner sep=0}, "\omega", curve={height=-12pt}, from=1-4, to=1-5]
	\arrow[""{name=1, anchor=center, inner sep=0}, "{\omega '}"', curve={height=12pt}, from=1-4, to=1-5]
	\arrow["\Gamma", between={0.2}{0.8}, Rightarrow, from=0, to=1]
\end{tikzcd}}
\end{center}

By (\ref{axiommu1gamma}) in \Cref{pseudocomm},  the previous whiskering equals the pasting diagram
\begin{center}
\begin{equation}\label{newlem3}
\adjustbox{scale=0.9}{
 \begin{tikzcd}
        TA_1\times TA_2\times TA_3\ar[r,"(1\times Tt_2)\circ (1\times t_1)"]&[10pt] TA_1\times T^2(A_2\times A_3)\ar[d,"t_2"']\ar[r,"t_1"]&[-10pt]T(A_1\times T^2(A_2\times A_3))\ar[r,"Tt_2"]&[-15pt]|[alias=D1]|T^2(A_1\times T(A_2\times A_3))\ar[d,"\mu"]\\
       & |[alias=R1]|T(TA_1\times T(A_2\times TA_3))\ar[d,"Tt_2"']\ar[r,"Tt_1"']&|[alias=D2]|T^2(A_1\times T(A_2\times A_3))\ar[d,"T^2t_2"]\ar[r,"\mu"']&T(A_1\times T(A_2\times A_3))\ar[d,"Tt_2"]\\
        &T^2(TA_1\times A_2\times A_3)\ar[d,"T^2t_1"]&T^3(A_1\times A_2\times A_3)\ar[r,"\mu"]\ar[d,"T\mu"]&T^2(A_1\times A_2\times A_3)\ar[d,"\mu"]\\
        &|[alias=R2]|T^3(A_1\times A_2\times A_3)\ar[r,"T\mu"']&T^2(A_1\times A_2\times A_3)\ar[r,"\mu"']&T(A_1\times A_2\times A_3).
        \arrow[Rightarrow,from=D1,to=R1,"\Gamma"',shorten <=15mm,shorten >=15mm]
        \arrow[Rightarrow,from=D2,to=R2,"T\Gamma"',shorten <=12mm,shorten >=12mm]
    \end{tikzcd}}
\end{equation}
\end{center}
First we will prove that the whiskering
\begin{center}
    \begin{equation}\label{newlem4}
        \adjustbox{scale=0.9}{\begin{tikzcd}
	{TA_1\times TA_2\times TA_3} &[15pt] {TA_1\times T^2(A_2\times A_3)} & {T(A_1\times T(A_2\times A_3))} & {T(A_1\times A_2\times A_3)}
	\arrow["{(1\times Tt_2)\circ (1\times t_1)}", from=1-1, to=1-2]
	\arrow[""{name=0, anchor=center, inner sep=0}, "\omega"', curve={height=12pt}, from=1-2, to=1-3]
	\arrow[""{name=1, anchor=center, inner sep=0}, "\omega", curve={height=-12pt}, from=1-2, to=1-3]
	\arrow["{\mu\circ Tt_2}", from=1-3, to=1-4]
	\arrow["\Gamma"', between={0.2}{0.8}, Rightarrow, from=1, to=0]
\end{tikzcd}}
    \end{equation}
\end{center}
coming from the previous diagram equals the whiskering
\begin{center}
\adjustbox{scale=0.9}{
\begin{tikzcd}
	{TA_1\times TA_2\times TA_3} &[15pt] {T(A_1\times A_2)\times TA_3} & {T(A_1\times A_2\times TA_3)} &[10pt] {T(A_1\times A_2\times A_3)}
	\arrow[""{name=0, anchor=center, inner sep=0}, "{\omega '\times 1}"', curve={height=12pt}, from=1-1, to=1-2]
	\arrow[""{name=1, anchor=center, inner sep=0}, "{\omega\times 1}", curve={height=-12pt}, from=1-1, to=1-2]
	\arrow["{t_1}", from=1-2, to=1-3]
	\arrow["{\mu\circ (Tt_3)}", from=1-3, to=1-4]
	\arrow["{\Gamma\times 1}", between={0.2}{0.8}, Rightarrow, from=1, to=0]
\end{tikzcd}}
\end{center}
coming from \Cref{newlem1}. By (\ref{axiomstrength3}) in  \Cref{pseudocomm},  the previous whiskering equals
\begin{center}
\adjustbox{scale=0.8}{ \begin{tikzcd}
        TA_1\times TA_2\times TA_3\ar[r,"1\times t_1"]&[-10pt]TA_1\times T(A_2\times TA_3)\ar[r,bend left,"\omega"name=D2]\ar[r,bend right,"\omega'"{swap,name=R2}]\ar[d,"1\times Tt_2"']&T(A_1\times A_2\times TA_3)\ar[d,"T(1\times t_2)"']\ar[dr,"Tt_3"]&[-10pt]\\[10pt]
        &TA_1\times T^2(A_2\times A_3)\ar[r,"\omega '"']&T(A_1\times T(A_2\times A_3))\ar[r,"Tt_2"']&T^2(A_1\times A_2\times A_3)\ar[r,"\mu"]&[-10pt]T(A_1\times A_2\times A_3).
        \arrow[Rightarrow,from=D2,to=R2,near start,"\Gamma ",shorten <=2mm,shorten >=2mm]
    \end{tikzcd}}
\end{center}
Since $\Gamma$ is a modification, the previous diagram equals \Cref{newlem4}. To finish part (1), we will prove that the whiskering 
\begin{center}
\begin{equation}\label{newlem5}
\adjustbox{scale=0.85}{
\begin{tikzcd}
	{TA_1\times TA_2\times TA_3} &[25pt] {T(TA_1\times T(A_2\times A_3))} & {T^2(A_1\times A_2\times A_3)} &[-15pt] {T(A_1\times A_2\times A_3)}
	\arrow["{t_2\circ (1\times Tt_2)\circ (1\times t_1)}", from=1-1, to=1-2]
	\arrow[""{name=0, anchor=center, inner sep=0}, "{T\omega '}"', curve={height=12pt}, from=1-2, to=1-3]
	\arrow[""{name=1, anchor=center, inner sep=0}, "{T\omega}", curve={height=-18pt}, from=1-2, to=1-3]
	\arrow["\mu", from=1-3, to=1-4]
	\arrow["T\Gamma",between={0.2}{0.8}, Rightarrow, from=1, to=0]
\end{tikzcd}}
\end{equation}
\end{center}
coming from \Cref{newlem3} equals the whiskering
\begin{center}
\adjustbox{scale=0.9}{
    \begin{tikzcd}
        TA_1\times TA_2\times TA_3\ar[r,"\cong\times 1"]
       &[-15pt] TA_2\times TA_1\times TA_3\ar[r,bend left=15,"1\times \omega"name=D2]\ar[r,bend right=15,"1\times \omega'"{swap,name=R2}]\ar[d,"t_1"]&TA_2\times T\ar[d,"t_1"](A_1\times A_3)\\
       & T(A_2\times TA_1\times TA_3)\ar[r,"T(1\times \omega ')"']&T(A_2\times T(A_1\times A_3))\ar[r,"T(\cong \times 1)\circ \mu \circ Tt_2"]&[15pt]T(A_1\times A_2 \times A_3)
       \arrow[Rightarrow,from=D2,to=R2,"1\times \Gamma ",shorten <=1mm,shorten >=1mm]
    \end{tikzcd}}
\end{center}
coming from \Cref{newlem1}. By 2-naturality of $t_1$, the previous diagram equals
\begin{center}
    \adjustbox{scale=0.75}{
\begin{tikzcd}
	{TA_1\times TA_2\times TA_3} &[-10pt] {T(A_2\times TA_1\times TA_3)} &[20pt] {T(A_2\times T(A_1\times A_3))} &[-15pt] {T^2(A_2\times A_1\times A_3)} &[-10pt] {T(A_1\times A_2\times A_3).}
	\arrow["{t_1\circ(\cong\times 1)}", from=1-1, to=1-2]
	\arrow[""{name=0, anchor=center, inner sep=0}, "{T(1\times \omega)}", curve={height=-12pt}, from=1-2, to=1-3]
	\arrow[""{name=1, anchor=center, inner sep=0}, "{T(1\times \omega')}"', curve={height=12pt}, from=1-2, to=1-3]
	\arrow["{Tt_2}", from=1-3, to=1-4]
	\arrow["{\mu\circ (\cong\times 1)}", from=1-4, to=1-5]
	\arrow["T(1\times \Gamma)"between={0.2}{0.8}, Rightarrow, from=0, to=1]
\end{tikzcd}}
\end{center}
By  (\ref{axiomstrength1}) in \Cref{pseudocomm}, this whiskering equals
\begin{center}
\adjustbox{scale=0.75}{
    \begin{tikzcd}
        TA_1\times TA_2\times TA_3\ar[r,"t_2"]\ar[d,"\cong\times 1"']&[-10pt]T(A_1\times A_2\times TA_3)\ar[d,"T(\cong\times 1)"] \ar[r,"T(t_1\times 1)"]&[-5pt]T(T(A_1\times A_2)\times TA_3)\ar[r,"T\omega"]\ar[d,"T(T\cong \times 1)"']&[-5pt]T^2(A_1\times A_2\times A_3)\ar[d,"T^2(\cong\times 1)"]\ar[r,"\mu"]&[-10pt]T(A_1\times A_2\times A_3)\ar[d,"T(\cong\times 1)"]\\[10pt]
        TA_2\times TA_1\times TA_3\ar[r,"t_1"']&T(A_2\times TA_1\times TA_3)\ar[r,"T(t_2\times 1)"']&T(T(A_2\times A_1)\times TA_3)\ar[r,bend left,"T \omega"name=D2]\ar[r,bend right,"T \omega' "{swap,name=R2}]&T^2(A_2\times A_1\times A_3)\ar[r,"\mu"]&T(A_2\times A_1\times A_3)\ar[d,"T(\cong \times 1)"]\\[-10pt]
        &&&&T(A_1\times A_2\times A_3).
        \arrow[Rightarrow,from=D2,to=R2,near start,"T\Gamma",shorten <=2mm,shorten >=2mm]
    \end{tikzcd}}
\end{center}
Since $\Gamma$ is a modificiation, we can write the previous whiskering as
\begin{center}
\adjustbox{scale=0.75}{
\begin{tikzcd}
	{TA_1\times TA_2\times TA_3} &[-15pt] {T(TA_1\times A_2\times TA_3)} &[-5pt] {T(T(A_1\times A_2)\times TA_3)} &[20pt] {T^2(A_1\times A_2\times A_3)} &[-15pt] {T(A_1\times A_2\times A_3).}
	\arrow["{t_2}", from=1-1, to=1-2]
	\arrow["{T(t_1\times 1)}", from=1-2, to=1-3]
	\arrow[""{name=0, anchor=center, inner sep=0}, "{T\omega}", curve={height=-12pt}, from=1-3, to=1-4]
	\arrow[""{name=1, anchor=center, inner sep=0}, "{T\omega '}"', curve={height=12pt}, from=1-3, to=1-4]
	\arrow["\mu", from=1-4, to=1-5]
	\arrow["{T\Gamma}", between={0.2}{0.8}, Rightarrow, from=0, to=1]
\end{tikzcd}}
\end{center}
By (\ref{axiomstrength2}) in \Cref{pseudocomm}, we get
\begin{center}
\adjustbox{scale=0.9}{
    \begin{tikzcd}
        TA_1\times TA_2\times TA_3\ar[r,"1\times t_1"']\ar[rr,bend left=10,"t_2"]&[-10pt]TA_1\times T(A_2\times TA_3)\ar[r,"t_2"]\ar[d,"1\times Tt_2"']&[-10pt]T(TA_1\times A_2\times TA_3)\ar[d,"T(1\times t_2)"']\\
        &TA_1\times T^2(A_2\times A_3)\ar[r,"t_2"']&T(TA_1\times T(A_2\times A_3))\ar[r,bend left,"T \omega"name=D2]\ar[r,bend right,"T \omega' "{swap,name=R2}]&T^2(A_1\times A_2\times A_3)\ar[d,"\mu"]\\[-10pt]
        &&&T(A_1\times A_2\times A_3),
        \arrow[Rightarrow,from=D2,to=R2,"T\Gamma",shorten <=1mm,shorten >=1mm]
    \end{tikzcd}}
\end{center}
which is precisely \Cref{newlem5}. We have proven part (\ref{newlem10}). Part (\ref{newlem9}) can be proven in a similar fashion.
\end{proof}
The next Lemma is the Yang-Baxter equation for pseudo commutative, strong 2-monads. Part (\ref{associativiyeqn}) is called the Associativity Equation in \cite{HP02}. Symmetry is not required.
\begin{lem}\label{yangbaxter}
Let $(T,\eta ,\mu ,t,\Gamma )$ be a pseudo commutative, strong, 2-monad. Then:
\begin{enumerate}
    \item The pasting
    \begin{center}
    \begin{equation}\label{wii+1}
        \begin{tikzcd}
    TA_1\times TA_2\times TA_3\ar[r,"1\times \omega"]\ar[d,"1\times \cong"']&|[alias=D1]|TA_1\times T(A_2\times A_3)\ar[r,"\omega"]&T(A_1\times A_2\times A_3)\\
    |[alias=R1]|TA_1\times TA_3\times TA_2\ar[rd,"\omega\times 1"']\ar[dd,"\cong\times 1"']\ar[r,"1\times \omega"']&TA_1\times T(A_3\times A_2)\ar[r,"\omega"']\ar[u,"1\times T\cong"']&T(A_1\times A_3\times A_2)\ar[u,"T(1\times\cong)"']\\[-10pt]
    &|[alias=D2]|T(A_1\times A_3)\times TA_2\ar[ru,"\omega"']\\
    |[alias=R2]|TA_3\times TA_1\times TA_2\ar[rd,"1\times \omega"']\ar[r,"\omega\times 1"']\ar[dd,"1\times \cong"']&T(A_3\times A_1)\times TA_2\ar[u," T\cong\times 1"']\ar[r,"\omega"']&T(A_3\times A_1\times A_2)\ar[uu,"T(\cong\times 1)"']\\[-10pt]
    &|[alias=D3]|TA_3\times T(A_1\times A_2)\ar[ru,"\omega"']\\
    |[alias=R3]|TA_3\times TA_2\times TA_1\ar[r,"1\times \omega"']&TA_3\times T(A_2\times A_1)\ar[r,"\omega"']\ar[u,"1\times T\cong"']&T(A_3\times A_2\times A_1)\ar[uu,"T(1\times\cong)"']
    \arrow[Rightarrow,from=D1,to=R1,"1\times \Gamma"',shorten <=3mm,shorten >=5mm]
    \arrow[Rightarrow,from=D2,to=R2,"\Gamma\times 1"',shorten <=3mm,shorten >=5mm]
    \arrow[Rightarrow,from=D3,to=R3,"1\times \Gamma"',shorten <=3mm,shorten >=5mm]
    \end{tikzcd}
    \end{equation}
\end{center}
equals the horizontal composite
\begin{center}
\begin{tikzcd}
	{TA_1\times TA_2\times TA_3} &[15pt] {TA_1\times T(A_2\times A_3)} & {T(A_1\times A_2\times A_3).}
	\arrow[""{name=0, anchor=center, inner sep=0}, "{1\times \omega}", curve={height=-12pt}, from=1-1, to=1-2]
	\arrow[""{name=1, anchor=center, inner sep=0}, "{1\times \omega'}"', curve={height=12pt}, from=1-1, to=1-2]
	\arrow[""{name=2, anchor=center, inner sep=0}, "\omega", curve={height=-12pt}, from=1-2, to=1-3]
	\arrow[""{name=3, anchor=center, inner sep=0}, "{\omega '}"', curve={height=12pt}, from=1-2, to=1-3]
	\arrow["{1\times \Gamma}", between={0.2}{0.8}, Rightarrow, from=0, to=1]
	\arrow["\Gamma", between={0.2}{0.8}, Rightarrow, from=2, to=3]
\end{tikzcd}
\end{center}
\item The pasting
\begin{center}
\begin{equation}
    \begin{tikzcd}
    TA_1\times TA_2\times TA_3\ar[d,"\cong\times 1"']\ar[r," \omega\times 1"]&|[alias=D1]|T(A_1\times A_3)\times TA_2\ar[r,"\omega"']&T(A_1\times A_2\times A_3)\\
    |[alias=R1]|TA_2\times TA_1\times TA_3\ar[rd,"1\times \omega"']\ar[r,"\omega\times 1"']\ar[dd,"1\times \cong"']&T(A_2\times A_1)\times TA_3\ar[r,"\omega"']\ar[u,"T\cong\times 1"]&T(A_2\times A_1\times A_3)\ar[u,"T(\cong\times 1)"']\\[-10pt]
    &|[alias=D2]|TA_2\times T(A_1\times A_3)\ar[ru,"\omega"']\\
    |[alias=R2]|TA_2\times TA_3\times TA_1\ar[r,"1\times \omega"']\ar[dd,"\cong\times 1"']\ar[rd,"\omega\times 1"']&TA_2\times T(A_3\times A_1)\ar[r,"\omega"]\ar[u,"1\times T\cong"']&T(A_2\times A_3\times A_1)\ar[uu,"T(1\times\cong)"']\\[-10pt]
    &|[alias=D3]|T(A_2\times A_3)\times TA_1\ar[ru,"\omega"']\\
    |[alias=R3]|TA_3\times TA_2\times TA_1\ar[r,"\omega\times 1"']&T(A_3\times A_2)\times TA_1\ar[u,"T\cong \times 1"']\ar[r,"\omega"']&T(A_3\times A_2\times A_1)\ar[uu,"T(\cong\times 1)"']
    \arrow[Rightarrow,from=D1,to=R1," \Gamma\times 1"',shorten <=3mm,shorten >=5mm]
    \arrow[Rightarrow,from=D2,to=R2,"1\times\Gamma"',shorten <=3mm,shorten >=5mm]
    \arrow[Rightarrow,from=D3,to=R3,"\Gamma\times 1"',shorten <=3mm,shorten >=5mm]
    \end{tikzcd}
    \end{equation}
\end{center}
equals the horizontal composite
\begin{center}
\begin{equation}\label{wii+2}
\begin{tikzcd}
	{TA_1\times TA_2\times TA_3} &[15pt] {T(A_1\times A_2)\times TA_3} & {T(A_1\times A_2\times A_3).}
	\arrow[""{name=0, anchor=center, inner sep=0}, "{\omega\times 1}", curve={height=-12pt}, from=1-1, to=1-2]
	\arrow[""{name=1, anchor=center, inner sep=0}, "{\omega' \times 1}"', curve={height=12pt}, from=1-1, to=1-2]
	\arrow[""{name=2, anchor=center, inner sep=0}, curve={height=-12pt}, from=1-2, to=1-3]
	\arrow[""{name=3, anchor=center, inner sep=0}, curve={height=12pt}, from=1-2, to=1-3]
	\arrow["{\Gamma\times 1}", between={0.2}{0.8}, Rightarrow, from=0, to=1]
	\arrow["\Gamma", between={0.2}{0.8}, Rightarrow, from=2, to=3]
\end{tikzcd}
\end{equation}
\end{center}
\item\label{associativiyeqn} The pastings and horizontal composites in  (1) and (2) are equal.
\end{enumerate}
\end{lem}
\begin{proof}
For (1), notice that by the \Cref{newlemma}, the pasting diagram
\begin{center}
\adjustbox{scale=0.85}{
    \begin{tikzcd}
     TA_1\times TA_2\times TA_3\ar[d,"1\times \cong"']&&T(A_1\times A_2\times A_3)\\
     TA_1\times TA_3\times TA_2\ar[d,"\cong\times 1"']\ar[r,"\omega\times 1"]&|[alias=D1]|T(A_1\times A_3)\times TA_2\ar[r,"\omega"]&T(A_1\times A_3\times A_2)\ar[u,"T(1\times \cong)"]\\
     |[alias=R1]|TA_3\times TA_1\times TA_2\ar[dd,"1\times \cong"']\ar[dr,"1\times \omega"']\ar[r,"\omega\times 1"]&T(A_3\times A_1)\times TA_2\ar[r,"\omega"]\ar[u,"T\cong\times 1"]&T(A_3\times A_1\times A_2)\ar[u,"T(\cong\times 1)"]\\
     [-15pt]
     &|[alias=D2]|TA_3\times T(A_1\times A_2)\ar[ru,"\omega"']&\\
     |[alias=R2]|TA_3\times TA_2\times TA_1\ar[r,"1\times \omega"']&TA_3\times T(A_2\times A_1)\ar[u,"1\times T\cong"]\ar[r,"\omega"']&T(A_3\times A_2\times A_1)\ar[uu,"T(1\times \cong )"']
     \arrow[Rightarrow,from=D1,to=R1,near start," \Gamma\times 1"',shorten <=6mm,shorten >=7mm]
        \arrow[Rightarrow,from=D2,to=R2,near start,"1\times\Gamma"',shorten <=6mm,shorten >=8mm]
    \end{tikzcd}}
\end{center}
equals the whiskering
\begin{center}
\adjustbox{scale=0.85}{
    \begin{tikzcd}
        TA_1\times TA_2\times TA_3\ar[d,"1\times \cong"']\ar[r,"1\times \omega '"]&TA_1\times T(A_2\times A_3)\ar[r,"\omega"]&T(A_1\times A_2\times A_3)\\
        TA_1\times TA_3\times TA_2\ar[r,"1\times \omega"']&TA_1\times T(A_3\times A_2)\ar[u,"1\times T\cong"]\ar[r,bend left,"\omega "{name=D1}]\ar[r,bend right,"\omega '"{swap,name=R1}]&T(A_3\times A_2\times A_1).\ar[u,"T(1\times \cong)"']
        \arrow[Rightarrow,from=D1,to=R1,near start,"\Gamma"',shorten <=2mm,shorten >=2mm]
    \end{tikzcd}}
\end{center}
Since $\Gamma$ is a modification, the last whiskering equals
\begin{center}
\adjustbox{scale=0.85}{
\begin{tikzcd}
	{TA_1\times TA_2\times TA_3} & {TA_1\times T(A_2\times A_3)} & {T(A_1\times A_2\times A_3).}
	\arrow["{1\times \omega '}", from=1-1, to=1-2]
	\arrow[""{name=0, anchor=center, inner sep=0}, "\omega", curve={height=-12pt}, from=1-2, to=1-3]
	\arrow[""{name=1, anchor=center, inner sep=0}, "{\omega '}"', curve={height=12pt}, from=1-2, to=1-3]
	\arrow["\Gamma", between={0.2}{0.8}, Rightarrow, from=0, to=1]
\end{tikzcd}}
\end{center}
Part (1) follows from this and part (2) is proven similarly. To prove part (3) we will prove that diagrams (\ref{wii+1}), and (\ref{wii+2}) are equal. We are done by (\ref{newlem9}) in \Cref{newlemma} since the whiserkings
\begin{center}
\adjustbox{scale=0.85,center}{\begin{tikzcd}
	{TA_1\times TA_2\times TA_3} &[-10pt] {TA_3\times TA_1\times TA_2} &[10pt] {TA_3\times T(A_1\times A_2),} &[-10pt] {T(A_3\times A_1\times A_2)} &[-10pt] {T(A_1\times A_2\times A_3)}
	\arrow["\cong", from=1-1, to=1-2]
	\arrow[""{name=0, anchor=center, inner sep=0}, "{1\times \omega '}"', curve={height=12pt}, from=1-2, to=1-3]
	\arrow[""{name=1, anchor=center, inner sep=0}, "{1\times \omega}", curve={height=-12pt}, from=1-2, to=1-3]
	\arrow["\omega", from=1-3, to=1-4]
	\arrow["{T\cong}", from=1-4, to=1-5]
	\arrow["1\times \Gamma",between={0.2}{0.8}, Rightarrow, from=1, to=0]
\end{tikzcd}}
\end{center}
and
\begin{center}
\adjustbox{scale=0.85}{
\begin{tikzcd}
	{TA_1\times TA_2\times TA_3} &[10pt] {T(A_1\times A_2)\times TA_3} & {T(A_1\times A_2\times A_3)}
	\arrow[""{name=0, anchor=center, inner sep=0}, "{\omega\times 1}", curve={height=-12pt}, from=1-1, to=1-2]
	\arrow[""{name=1, anchor=center, inner sep=0}, "{\omega'\times 1}"', curve={height=12pt}, from=1-1, to=1-2]
	\arrow["{\omega'}"', from=1-2, to=1-3]
	\arrow["{\Gamma\times 1}", between={0.2}{0.8}, Rightarrow, from=0, to=1]
\end{tikzcd}}
\end{center}
are equal. 
\end{proof}
In the presence of symmetry, we can give (a slight generalization of) the previous lemma the following interpretation.
\begin{lem}
Let $(T,\eta,\mu,t,\Gamma )$ be a symmetric, pseudo commutative, strong 2-monad. Then, the pasting diagram
\begin{center}
\adjustbox{scale=0.9}{
    \begin{tikzcd}
        \K(A_1\times\cdots \times A_n,C)\ar[r,"T"]\ar[d,"\sigma_i"']&[-10pt]|[alias=R1]|\talg (TA_1,\dots, TA_n;TC)\ar[d,"\sigma_i"]\\
        |[alias=D1]|\K(A_1\times \cdots\times A_{i+1}\times A_i\times \cdots\times A_n,C)\ar[r,"T"{swap}]\ar[d,"\sigma_{i+1}"']&|[alias=R2]|\talg (TA_1,\dots,TA_{i+1},TA_{i},\dots,TA_n;TC)\ar[d,"\sigma_{i+1}"]\\
         |[alias=D2]|\K(A_1\times\cdots\times A_{i+1}\times A_{i+2}\times A_i\times \cdots\times A_n,C)\ar[r,"T"']\ar[d,"\sigma_i"']&|[alias=R3]|\talg (TA_1,\dots,TA_{i+1},TA_{i+2},TA_i,\dots , TA_n;TC)\ar[d,"\sigma_i"]\\
        |[alias=D3]|\K(A_1\times\cdots\times A_{i+2}\times A_{i+1}\times A_i\times \cdots\times A_n,C)\ar[r,"T"']&\talg (TA_1,\dots,TA_{i+2},TA_{i+1},TA_i,\dots , TA_n;TC),
        \arrow[Rightarrow,from=D1,to=R1,"T_{\sigma_i}",shorten <=11mm,shorten >=12mm]
        \arrow[Rightarrow,from=D2,to=R2,"T_{\sigma_{i+1}}"',shorten <=11mm,shorten >=12mm]
        \arrow[Rightarrow,from=D3,to=R3,"T_{\sigma_{i}}"',shorten <=11mm,shorten >=12mm]
    \end{tikzcd}}
\end{center}
equals the pasting diagram
\begin{center}
\adjustbox{scale=0.9}{
    \begin{tikzcd}
        \K(A_1\times\cdots\times  A_n,C)\ar[r,"T"]\ar[d,"\sigma_{i+1}"']&[-10pt]|[alias=R1]|\talg (TA_1,\dots, TA_n;TC)\ar[d,"\sigma_{i+1}"]\\
        |[alias=D1]|\K(A_1\times \cdots\times A_{i+2}\times A_{i+1}\times \cdots\times A_n,C)\ar[r,"T"{swap}]\ar[d,"\sigma_{i}"']&|[alias=R2]|\talg (TA_1,\dots,TA_{i+2},TA_{i+1},\dots,TA_n;TC)\ar[d,"\sigma_{i}"]\\
         |[alias=D2]|\K(A_1\times\cdots\times A_{i+2}\times A_{i}\times A_{i+1}\times \cdots\times A_n,C)\ar[r,"T"']\ar[d,"\sigma_{i+1}"']&|[alias=R3]|\talg (TA_1,\dots,TA_{i+2},TA_{i},TA_{i+1},\dots , TA_n;TC)\ar[d,"\sigma_{i+1}"]\\
        |[alias=D3]|\K(A_1\times\cdots\times A_{i+2}\times A_{i+1}\times A_i\times \cdots\times A_n,C)\ar[r,"T"']&\talg (TA_1,\dots,TA_{i+2},TA_{i+1},TA_i,\dots , TA_n;TC).
        \arrow[Rightarrow,from=D1,to=R1,"T_{\sigma_{i+1}}",shorten <=11mm,shorten >=12mm]
        \arrow[Rightarrow,from=D2,to=R2,"T_{\sigma_{i}}"',shorten <=11mm,shorten >=12mm]
        \arrow[Rightarrow,from=D3,to=R3,"T_{\sigma_{i+1}}"',shorten <=11mm,shorten >=12mm]
    \end{tikzcd}}
\end{center}
\end{lem}
The three previous lemmas give us the following.
\begin{thm}
Suppose that $(T,\eta,\mu,t,\Gamma )$ is a symmetric, pseudo commutative strong  2-monad and let $A_1,\dots, A_n,C$ be objects of $\K$. The transformations $T_{\sigma_i}$ for $1\leq i\leq n-1$ assemble together to give, for each $\sigma \in \Sigma_n,$ a unique transformation
\begin{center}
    \begin{tikzcd}
        \K (A_1\times \cdots\times A_n ,C)\ar[d,"\sigma"']\ar[r,"T"]&|[alias=R1]|\talg (TA_1,\dots,TA_n;TC)\ar[d,"\sigma"]\\
        |[alias=D1]|\K(A_{\sigma(1)}\times \cdots\times A_{\sigma(n)},C)\ar[r,"T"']&\talg (TA_{\sigma(1)},\dots,TA_{\sigma(n)};TC).
        \arrow[Rightarrow,from=D2,to=R2,"T_{\sigma}"',shorten <=11mm,shorten >=12mm]
    \end{tikzcd}
\end{center}
These satisfy the unit and the product permutation axiom in \Cref{pseudosymm}.
\end{thm}
We are just missing the top and bottom equivariance axioms to prove that our functor $T:\K\to\talg$ is pseudo symmetric. When $T$ is a pseudo commutative, strong 2-monad that fails to be symmetric, we can still give \Cref{yangbaxter} an interpretation using the Bruhat order of the symmetric group $\Sigma_n$ on generators $\sigma_i$ for $1\leq i \leq n-1.$
\begin{defi}
    Let $\Sigma_n$ be the symmetric group with generators $\{\sigma_i\}_{1\leq i<n}$ and presentation:
    \begin{itemize}
        \item $\sigma_i \sigma_i =1,$
        \item $\sigma_i\sigma_j=\sigma_j\sigma_i$ if $|i-j|\geq 2$
        \item $\sigma_i\sigma_{i+1}\sigma_i=\sigma_{i+1}\sigma_i\sigma_{i+1}.$
    \end{itemize}
    The length of a permutation $\sigma\in\Sigma_n,$ $\ell (\sigma)$, is the number of inversions of $\sigma,$ i.e., the number of couples $(i,j)$ such that $1\leq i<j\leq n$ and $\sigma(i)<\sigma(j).$ This agrees with the length of a minimal word for $\sigma$ in the previous presentation \cite[Prop. 1.5.2.]{BB05}. The weak right order on $\Sigma_n$ \cite[Def. 3.1.1.]{BB05}is the partial order on $\Sigma_n$ generated by declaring that $\sigma<\sigma\sigma_i$ when $\ell(\sigma)< \ell(\sigma\sigma_i)$ \cite[p. 66]{BB05}. This only happens when none of the reduced words for $\sigma$ end in $\sigma_i.$ The bottom of this order is the identity and the top is the reverse order permutation.
\end{defi}
\begin{rmk}\label{hpcoherence}
    Let $(T,\eta,\mu,t,\Gamma)$ be a pseudo commutative, strong 2-monad and $A_1,\dots,A_n$ objects of $\K$. We have the 1-cell
    \begin{center}
        \begin{tikzcd}
            \omega\colon A_1\times \cdots\times A_n\ar[r]&T(A_1\times \cdots\times A_n).
        \end{tikzcd}
    \end{center}
    Although we don't have a symmetric $\mathbf{Cat}$-multicategory, we still have a 1-cell $\omega_\sigma$ (it is called $t_\sigma$ on \cite{HP02}):
    \begin{center}
        \begin{tikzcd}
            TA_1\times \cdots\times TA_n\ar[r,"\omega_\sigma"]\ar[d,"\sigma^{-1}"']&T(A_1\times \cdots\times A_n)\\
            TA_{\sigma(1)}\times \cdots\times TA_{\sigma(n)}\ar[r,"\omega"']&T(A_{\sigma(1)}\times \cdots\times A_{\sigma(n)}).\ar[u,"T\sigma"']
        \end{tikzcd}
    \end{center}
    When $\ell(\sigma)<\ell(\sigma \sigma_i ),$ we can define a 2-cell $\omega_\sigma\to\omega_{\sigma\sigma_i}$ as
 \begin{center}
 \begin{equation}\label{bruhat}\adjustbox{scale=0.7}{
\begin{tikzcd}
	{TA_1\times \cdots\times TA_n} & {TA_{\sigma(1)}\times\cdots\times TA_{\sigma(n)}} & {TA_{\sigma(1)}\times\cdots\times TA_{\sigma(i+1)}\times TA_{\sigma(i)}\times \cdots\times TA_{\sigma(n)}} \\
	& {TA_{\sigma(1)}\times\cdots\times T\left( A_{\sigma(i)}\times A_{\sigma (i+1)}\right) \times \cdots\times TA_{\sigma(n)}} & {TA_{\sigma(1)}\times \cdots\times T\left(A_{\sigma(i+1)}\times A_{\sigma(i)}\right)\times \cdots\times TA_{\sigma(n)}} \\
	{T(A_1\times\cdots\times A_n)} & {T\left(A_{\sigma(1)}\times \cdots\times A_{\sigma(n)}\right)} & {T\left(A_{\sigma(1)}\times\cdots\times A_{\sigma(i+1)}\times A_{\sigma(i)}\times\cdots\times A_{\sigma(n)}\right)}
	\arrow["{{\sigma^{-1}}}", from=1-1, to=1-2]
	\arrow["{\omega_\sigma}"', from=1-1, to=3-1]
	\arrow["{{\sigma_{i}^{-1}}}", from=1-2, to=1-3]
	\arrow["{{1\times\omega\times 1}}"', from=1-2, to=2-2]
	\arrow["{{1\times \omega \times 1}}", from=1-3, to=2-3]
	\arrow["{{1\times \Gamma\times 1}}", Rightarrow, from=2-2, to=1-3]
	\arrow["\omega"', from=2-2, to=3-2]
	\arrow["{{1\times T\cong \times 1}}", from=2-3, to=2-2]
	\arrow["\omega", from=2-3, to=3-3]
	\arrow["{{T\sigma}}", from=3-2, to=3-1]
	\arrow["{{T\sigma_i}}", from=3-3, to=3-2]
\end{tikzcd}
}
\end{equation}
 \end{center}
    Thus we have a 2-cell $\omega_\sigma\to\omega_{\sigma '}$ when $\sigma<\sigma '$ in the weak right order.
    
    Notice that our definition gives a 2-cell $\omega_{\sigma}\to \omega_{\sigma\sigma_i}$ even when $\sigma <\sigma\sigma_i$ is false in the weak right order, but we avoid considering these cells since, in the absence of symmetry, $\omega_{\sigma}\to\omega_{\sigma\sigma_i}\to \omega_{\sigma\sigma_i\sigma_i}$ may not be the identity.
    
    By \Cref{yangbaxter2,yangbaxter}, there is functor $\Omega:B_n\to \K(TA_1\times \cdots\times TA_n,T(A_1\times \cdots\times A_n)),$ 
    with $\Omega(1)=\omega,$ $\Omega(\sigma)=\omega_\sigma,$ and such that when $\sigma<\sigma\sigma_i$ in $B_n,$ $\Omega(\sigma<\sigma\sigma_i)$ is the 2-cell \Cref{bruhat}. We believe this to be the coherence theorem that Hyland and Power refer to in \cite{HP02}.
\end{rmk}
To finish proving our coherence theorem, that is, that $T:\K\to \talg$ is pseudo symmetric we need to prove the top and bottom invariance axioms for $T$ in \Cref{pseudosymm}. First we will prove top equivariance for $\sigma_i,$ for which we will need the following.

\begin{notn}In what follows we will use $\overline{A_i}$ to denote $A_{i,1}\times \cdots\times A_{i,k_i}$ and $\overline{A}$ to denote $\overline{A_1}\times\cdots\times \overline{A_n}$ when $k_i$ and $n$ are clear from the context.  We will also use $\overline{A_{>i}}=\overline{A_{i+1}}\times\cdots\times \overline{A_{n}}$ when $n$ is clear for the context, and similarly  $\overline{A_{<i}}=\overline{A_1}\times\cdots \times \overline{A_{i-1}}.$
\end{notn}
\begin{lem}\label{topequivarianceT}
    Let $(T,\eta,\mu,t,\Gamma)$ be a symmetric, pseudo commutative, strong 2-monad. Suppose $k_i\geq 1 $ for $1\leq i \leq n$, and let $A_{i,1},\dots, A_{i,{k_i}}$ for $1\leq i \leq n$ and $C$ be objects of $\K$. The component of the natural transformation $T_{\sigma_i\la \id_{k_{\sigma_i(j)}}\ra_{j=1}^n},$ fitting in the diagram
    \begin{center}
        \adjustbox{scale=0.9}{\begin{tikzcd}
            \K(\overline{A_1}\times \cdots \times \overline{A_n},C)\ar[d,"\sigma_i \la\id_{{k_{\sigma_i(j)}}}\ra_{j=1}^n "']\ar[r,"T"name=D1]&\talg (\la TA_1\ra,\dots,\la TA_n\ra,TC)\ar[d,"\sigma_i \la\id_{{k_{\sigma_i(j)}}}\ra_{j=1}^n "]\\
            |[alias=R1]|\K (\overline{A_{<i}}\times\overline{A_{i+1}}\times\overline{A_i}\times \overline{A_{>i+1}},C)\ar[r,"T"']&\talg (\la\la TA_{j}\ra\ra_{j<i},\la TA_{i+1}\ra,\la TA_i\ra,\la\la TA_j\ra\ra_{j>i+1};C),
            \arrow[Rightarrow,from=R1,to=D1,"T_{\sigma_i \la\id_{{k_{\sigma_i(j)}}}\ra_{j=1}^n }"',shorten <=8mm,shorten >=8mm]
        \end{tikzcd}}
    \end{center}
at $h\colon \overline{A_1}\times \cdots\times \overline{A_n}\to C$ in $\K$ is  
\begin{center}\adjustbox{scale=.9 ,center}{
\begin{tikzcd}
	{\overline{TA_1}\times\cdots\times \overline{TA_{i+1}}\times \overline{TA_i}\times \cdots \times \overline{TA_n}} && {\overline{TA_1}\times\cdots\times \overline{TA_{i+1}} \times \overline{TA_i}\times \cdots \times\overline{TA_n}} \\
	{\overline{TA_1}\times\cdots\times T\left(\overline{A_{i+1}}\right)\times T\left(\overline{A_i}\right)\times \cdots \times \overline{TA_n}} && {\overline{TA_1}\times\cdots\times T\left(\overline{A_i}\right)\times T\left(\overline{A_{i+1}}\right)\times\cdots \times \overline{TA_n}} \\
	{\overline{TA_1}\times\cdots\times T\left(\overline{A_{i+1}}\times\overline{A_i}\right)\times \cdots \times \overline{TA_n}} && {\overline{TA_1}\times\cdots\times T\left(\overline{A_{i+1}}\times\overline{A_i}\right)\times \cdots \times \overline{TA_n}} \\[-10pt]
	{T(\overline{A_1}\times \cdots\times \overline{A_{i+1}}\times\overline{A_i}\times \cdots\times \overline{A_n})} & {T(\overline{A_1}\times\cdots\times \overline{A_n})} & TC.
	\arrow["{1\times \cong\times 1}", from=1-1, to=1-3]
	\arrow["{1\times \omega\times \omega\times 1}"', from=1-1, to=2-1]
	\arrow["{1\times \omega\times \omega\times 1}", from=1-3, to=2-3]
	\arrow["{1\times \cong \times 1}", from=2-1, to=2-3]
	\arrow["{1\times \omega\times 1}"', from=2-1, to=3-1]
	\arrow["{1\times \omega\times 1}", from=2-3, to=3-3]
	\arrow["{1\times \Gamma\times 1}", Rightarrow, from=3-1, to=2-3]
	\arrow["\omega"', from=3-1, to=4-1]
	\arrow["{1\times T(\cong)\times 1}", from=3-3, to=3-1]
	\arrow["{T(1\times \cong\times 1)}"', from=4-1, to=4-2]
	\arrow["h"', from=4-2, to=4-3]
\end{tikzcd}}
\end{center}
\end{lem}
\begin{proof}
We prove this by induction on $k_i$ and $k_{i+1}.$ For $k_i=k_{i+1}=1$ this is just \Cref{pseudosymmisos}. Next we induct on $k_{i+1}$ assuming $k_i=1.$ In this case we can write $\sigma_i\la\id_{k_{\sigma_i(j)}}\ra$ as the composition
\begin{center}
\adjustbox{scale=0.9}{
\begin{tikzcd}
\overline{A_1}\times\cdots\times \overline{A_{i-1}}\times A_{i+1,1}\times A_{i+1,2}\times\cdots\times A_{i+1,k_{i+1}}\times A_i\times \overline{A_{i+2}}\times\cdots\times \overline{A_n}\ar[d,"\sigma_{i+1}\la\id_{k_1}{,}\dots {,}\id_{k_{i-1}}{,}\id_1{,}\id_{k_{i+1}-1}{,}\id_1{,}\id_{k_{i+2}}{,}\dots{,}\id_{k_n}\ra"]\\
\overline{A_1}\times\cdots\times \overline{A_{i-1}}\times A_{i+1,1}\times A_i\times A_{i+1,2}\times \cdots\times A_{i+1,k_{i+1}}\times \overline{A_{i+2}}\times \cdots \times \overline{A_n}\ar[d,"\sigma_{1+\sum_{t=1}^{i-1}k_t}"]\\
\overline{A_1}\times \cdots\times \overline{A_{i-1}}\times A_i\times \overline{A_{i+1}}\times \overline{A_{i+2}}\times\cdots\times \overline{A_n}.
\end{tikzcd}}
\end{center}
After applying the inductive hypothesis to  $\sigma_{i+1}\la\id_{k_1}{,}\dots {,}\id_{k_{i-1}}{,}\id_1{,}\id_{k_{i+1}-1}{,}\id_1{,}\id_{k_{i-2}}{,}\dots{,}\id_{k_n}\ra,$
\Cref{pseudosymmisos} to $\sigma_{1+\sum_{t=1}^{i-1}k_t},$ and the product axiom, we get the result for $\sigma_i\la\id_{k_{\sigma_i(j)}}\ra$ by an application of (\ref{newlem10}) in \Cref{newlemma}. By induction, the result holds for any $k_{i+1}$ and $k_i=1.$\\
We  finish by induction on $k_i,$ proving that the result holds for all $k_{i+1}.$ We have proven that this is true for $k_i=1.$ For the inductive step we can write $\sigma_i\la\id_{k_{\sigma_i(j)}}\ra$ as
    \begin{center}
    \adjustbox{scale=0.9}{
    \begin{tikzcd}
        \overline{A_1}\times\cdots\times \overline{A_{i-1}}\times\overline{A_{i+1}}\times A_{i,1}\times \cdots\times A_{i,{k_i}-1}\times A_{i,k_i}\times \overline{A_{i+2}}\times \cdots\times \overline{A_n } \ar[d,"\sigma_i\la\id_{k_1}{,}\dots {,}\id_{k_{i-1}}{,}\id_{k_{i+1}}{,}\id_{{k_i}-1}{,}\id_1{,}\id_{k_{i+2}}{,}\dots{,}\id_{k_n}\ra"]\\
        \overline{A_1}\times \cdots\times \overline{A_{i-1}}\times A_{i,1}\times\cdots\times A_{i,{k_i}-1}\times \overline{A_{i+1}}\times A_{i,{k_i}}\times \overline{A_{i+2}}\times \cdots\times \overline{A_n}\ar[d,"\sigma_{i+1}\la\id_{k_1}{,}\dots{,}\id_{k_{i-1}}{,}\id_{{k_i}-1}{,}\id_{k_{i+1}}{,}\id_1{,}\id_{k_{i+2}}{,}\dots{,}\id_{k_n}\ra"]\\
        \overline{A_1}\times \cdots\times \overline{A_{i}}\times \overline{A_{i+1}}\times\cdots\times \overline{A_i}.
    \end{tikzcd}}
    \end{center}
After applying the inductive hypothesis to
$\sigma_i\la\id_{k_1}{,}\dots {,}\id_{k_{i-1}}{,}\id_{k_{i+1}}{,}\id_{{k_i}-1}{,}\id_1{,}\id_{k_{i+2}}{,}\dots{,}\id_{k_n}\ra,$
the already proven to 
$\sigma_{i+1}\la\id_{k_1}{,}\dots{,}\id_{k_{i-1}}{,}\id_{{k_i}-1}{,}\id_{k_{i+1}}{,}\id_1{,}\id_{k_{i+2}}{,}\dots{,}\id_{k_n}\ra,$
and the product axiom, we get our result by an application of (\ref{newlem9}) in \Cref{newlemma}.
\end{proof}
\begin{lem}\label{topeqT}
Suppose $(T,\eta,\mu,t,\Gamma)$ is a symmetric, pseudo commutative, strong 2-monad. Let $n\geq 2$ and $1\leq i\leq n-1,$ and consider the $\Cat$-multifunctor $T:\K\to \talg$. Then, the top equivariance axiom in \Cref{pseudosymm} holds for $\sigma_i\la \id_{k_{\sigma_i(j)}}\ra_{j=1}^n.$ That is, for every $C\in\ob(\K),$  $\la B \ra =\la B_j\ra_{j=1}^n\in \ob(\K)^n,$ $k_j\geq 0$ for $1\leq j\leq n,$   and $\la A_j\ra=\la A_{j,i}\ra_{i=1}^{k_j}\in \ob(\K)^{k_j}$ for $1\leq j \leq n,$  the pasting diagram
    \begin{center}
    \adjustbox{scale=0.9}{
        \begin{tikzcd}
            \K \left(\overline{B};C\right)\times\prod\limits_j\K\left(\overline{A_i},B_j\right)\ar[rr,"\prod T"]\ar[d,"\gamma"{swap,name=h}]&[-10pt]&[-10pt]\talg(\la TB\ra;TC)\times\prod\limits_j\talg\left(\langle TA_{j}\rangle;TB_j\right)\ar[d,"\gamma"name=g]\\
            
            \K\left(\overline{A},C\right)\ar[d,"\sigma_i\la\id_{k_{\sigma_i(j)}}\ra"swap]\ar[rr,"T"]&&|[alias=N1]|\talg\left(\la\la TA_j\ra\ra;TC\right)\ar[d,"\sigma_i\la\id_{k_{\sigma_i(j)}}\ra"]\\
            |[alias=M1]|\K\left(\overline{A_{<i}}\times \overline{A_{i+1}}\times \overline{A_i}\times\overline{A_{>i+1}},C\right)\ar[rr,"T"swap]&&\talg\left(\la\la TA_j\ra\ra_{j<i},\la TA_{i+1}\ra,\la TA_{i}\ra,\la \la TA_j\ra\ra_{j>i+1};TC\right)
            \arrow[Rightarrow,from=M1, to=N1,shorten=6mm,"T_{\sigma_i\la\id_{k_{\sigma_i(j)}}\ra}"swap]
            \end{tikzcd}}
            \end{center}
            equals the pasting diagram
\begin{center}\adjustbox{scale=.85}{
\begin{tikzcd}
	{\K\left(\overline{B},C\right)\times \prod\limits_j\K\left(\overline{A_j},B_j\right)} &[-10pt] { {\talg(\la TB\ra;TC)\times \prod\limits_j \talg (\la TA_j\ra;TB_j)}} \\
	{\K(B_{<i}\times B_{i+1}\times B_i\times B_{>i+1},C)\times \prod\limits_j\left(\overline{A_{\sigma_i(j)}},B_{\sigma_i(j)}\right)} & {\talg (\la TB\ra \sigma_i;TC)\times \prod\limits_j \talg \left(\la TA_{\sigma_i(j)}\ra;TB_{\sigma_i(j)}\right)} \\
	{\K\left(\overline{A_i}\times \overline{A_{i+1}}\times \overline{A_i}\times \overline{A_{>i+1}},C\right)} & {\talg \left(\la \la TA_j\ra\ra_{j<i} ,\la TA_{i+1}\ra,\la TA_{i}\ra ,\la\la TA_j\ra\ra_{j>i+1};TC\right).}
	\arrow["{{\prod T}}", from=1-1, to=1-2]
	\arrow["{{\sigma_i\times \sigma_i^{-1}}}"', from=1-1, to=2-1]
	\arrow["{{\sigma_i\times \sigma_i^{-1}}}", from=1-2, to=2-2]
	\arrow["{{T_{\sigma_i}\times 1}}"', between={0.2}{0.8}, Rightarrow, from=2-1, to=1-2]
	\arrow["{{\prod T}}"', from=2-1, to=2-2]
	\arrow["\gamma"', from=2-1, to=3-1]
	\arrow["\gamma", from=2-2, to=3-2]
	\arrow["T"', from=3-1, to=3-2]
\end{tikzcd}}
\end{center}
\end{lem}
\begin{proof}
The lemma follows at once from \Cref{topequivarianceT}, \Cref{pseudosymmisos},  and (\ref{newlem10}) in \Cref{newlemma}.
\end{proof}
\begin{lem}\label{bottomeqT}
Suppose $(T,\eta,\mu,t,\Gamma)$ is a symmetric, pseudo commutative, strong 2-monad. Let $n\geq 1$ and $1\leq i\leq n-1,$ and consider the $\Cat$-multifunctor $T:\K\to \talg$. Then, the bottom equivariance axiom in \Cref{pseudosymm} holds for $\id_n\la\id_{k_1},\dots,\sigma_i,\dots,\id_{k_n}\ra$ that is, For every $C\in\ob(\K),$ $\la B \ra =\la B_j\ra_{j=1}^n\in \ob(K)^n,$ $k_j\geq 0$ for $1\leq j\leq n,$   and $\la A_j\ra=\la A_{j,l}\ra_{l=1}^{k_j}\in \ob(\K)^{k_j}$ for $1\leq j \leq n,$ the pasting diagram
\begin{center}
\adjustbox{scale=0.9}{
    \begin{tikzcd}
        \K\left(\overline{B},C\right)\times \prod\limits_j \K\left(\overline{A_j},B_j\right)\ar[d,"\gamma"']\ar[rr,"\prod T"]&[-20pt]&[-20pt]\talg\left(\la TB\ra;TC\right)\times \prod\limits_j \talg\left(\la TA_j\ra;TB_j\right)\ar[d,"\gamma"]\\
\K\left(\overline{A},C\right)\ar[rr,"T"name=R1]\ar[d,"\id_n\la \id_{k_1}{,}\dots{,}\sigma_i{,}\dots{,}\id_{k_n}\ra"']&&|[alias=R2]|\talg\left(\la\la TA_j\ra\ra;TC\right)\ar[d,"\id_n\la \id_{k_1}{,}\dots{,}\sigma_i{,}\dots{,}\id_{k_n}\ra"]\\[10pt]
    |[alias=D1]|\K\left(\overline{A_{<i}}\times \prod\limits_j A_{i,\sigma_i(j)}\times \overline{A_{>i}},C\right)\ar[rr,"T"{swap,name=D2}]&&\talg\left(\la \la TA_j\ra\ra_{j<i},\la TA_{i,\sigma_i(j)}\ra_j,\la \la TA_j\ra\ra_{j>i} ;TC\right)
    \arrow[Rightarrow,from=D1,to=R2," T_{\id_n\la \id_{k_1}{,}\dots{,}\sigma_i{,}\dots{,}\id_{k_n}\ra}"',near start,shorten <=8mm,shorten >=8mm]
    \end{tikzcd}}
\end{center}
is equal to the pasting
\begin{center}\adjustbox{scale=.7,center}{\begin{tikzcd}[column sep=-2in]
	{\K\left(\overline{B},C\right)\times \prod\limits_j\K\left(\overline{A_j},B_j\right)} & {\talg(\la TB\ra ;TC)\times \prod\limits_j \left(\la TA_j\ra ;B_j\right)} \\
	{\K\left(\overline{B},C\right)\times\prod\limits_{j<i}\K\left(\overline{A_j},B_j\right)\times \K\left(\prod\limits_j A_{i,\sigma_i(j)},B_i\right)\times \prod\limits_{j>i}\K\left(\overline{A_j},B_j\right)} \\
	& {\talg (\la TB\ra ;TC)\times \prod\limits_{j<i}\talg (\la TA_j\ra ;B_j)\times \talg\left(\la TA_{i,\sigma_i(j)}\ra ;TB_i\right)\times \prod\limits_{j>i}\talg(\la TA_j\ra ;TB_j)} \\
	{\K\left(\overline{A_{<i}}\times \prod\limits_j A_{i,\sigma_i(j)}\times \overline{A_{>i}},C\right)} & {\talg\left( \la TA_j\ra_{j<i},\la TA_{i,\sigma_i(j)}\ra,\la TA\ra_{j>i}\ra ;C\right).}
	\arrow["{\prod T}", from=1-1, to=1-2]
	\arrow["{\id\times \id_{k_1}\times\cdots\times\sigma_i\times \cdots\times \id_{k_n}}"', from=1-1, to=2-1]
	\arrow["{\id\times \id_{k_1}\times\cdots\times \sigma_i\times\cdots\times \id_{k_n}}", from=1-2, to=3-2]
	\arrow["{1\times T_{\sigma_i}\times 1}"',between={0.2}{0.8}, Rightarrow, from=2-1, to=1-2]
	\arrow["{\prod T}"', from=2-1, to=3-2]
	\arrow["\gamma"', from=2-1, to=4-1]
	\arrow["\gamma", from=3-2, to=4-2]
	\arrow["T"', from=4-1, to=4-2]
\end{tikzcd}}
\end{center}
\end{lem}
\begin{proof}
The lemma follows at once from \Cref{pseudosymmisos},  and (\ref{newlem9}) in \Cref{newlemma}.
\end{proof}
Finally we arrive at the proof of our main theorem.
\begin{thm}\label{maintheoremchapter2}
    Suppose $(T,\eta,\mu,t,\Gamma)$ is a symmetric, pseudo commutative, strong 2-monad. Then, the free algebra $\Cat$-multifunctor $T\colon \K\to \talg$ is pseudo symmetric.
\end{thm}
\begin{proof}
We just need to prove that the bottom and top equivariance axioms hold for $T.$ For the top equivariance axiom we notice that given $\sigma,\tau\in \Sigma_n$, and $k_1,\dots,k_n,$ we can write $\sigma\tau\la\id_{k_{\sigma\tau(1)}},\dots,\id_{k_{\sigma\tau(n)}}\ra$ as the composition
\begin{center}
    \begin{tikzcd}[column sep=0.5in]
        \overline{A_{\sigma\tau(1)}}\times\cdots\times \overline{A_{\sigma\tau(n)}}\ar[r,"\tau\la\id_{k_{\sigma\tau(i)}}\ra"]&
        \overline{A_{\sigma(1)}}\times\cdots\times \overline{A_{\sigma(n)}}\ar[r,"\sigma\la\id_{k_{\sigma(i)}}\ra"]&
        \overline{A_1}\times\cdots\times \overline{A_n}.
    \end{tikzcd}
\end{center}
By an application of the product axiom, if $\sigma\la \id_{k_{\sigma(i)}}\ra$ and $\tau\la {\id_{k_{\sigma\tau(i)}}}\ra$ satisfy the top invariance axiom, then so does $\sigma\tau\la\id_{k_{\sigma\tau(1)}},\dots,\id_{k_{\sigma\tau(n)}}\ra.$ We are done by Lemma \Cref{topeqT}.

Similarly, for the bottom equivariance axiom. Given $n,$ $k_1,\dots k_n$ and $\sigma ,\tau \in\Sigma_{k_i}.$ If the bottom equivariance axiom holds for $\id_n\la \id_{k_1},\dots,\tau,\dots\id_{k_n}\ra,$ and $\id_n\la \id_{k_1},\dots,\sigma,\dots,\id_{k_n}\ra,$ then it also holds for $\id_n \la\id_{k_1},\dots,\sigma\tau ,\dots,\id_{k_n}\ra$ by  the product axiom. By \Cref{bottomeqT}, we get the bottom equivariance axiom for $\id_n\la \id_{k_1},\dots,\sigma ,\dots , \id_{k_n}\ra$ for any $\sigma\in \Sigma_{k_i}.$ On the other hand, if the bottom equivariance axiom holds for $\id_n\la \sigma_1,\dots,\sigma_n\ra,$ and $\id_n\la\tau_1,\dots,\tau_n\ra,$ where $\sigma_i,\tau_i\in \Sigma_{k_i},$ then it also holds for $\id_n\la \sigma_1\tau_1,\dots,\sigma_n\tau_n\ra$ by another application of the product axiom. Thus $T$ saatisfies the bottom equivariance axiom.
\end{proof}
Since the free functor associated to a pseudo commutative operad is a symmetric, pseudo commutative strong 2-monad, the free functors of the pseudo commutative operads defined in \cite{CG13,GMMO21} and considered as well in \cite{Y24} are pseudo symmetric.
\begin{rmk}
We can get a version of \Cref{hpcoherence} for the case were $(T,\eta,\mu,t,\Gamma)$ is a symmetric, pseudo commutative, strong 2-monad. In this case, the Bruhat order is replaced by $E\Sigma_n ,$ i.e., the category with objects $\Sigma_n$ and a unique morphism between each pair of objects. We notice that the $E\Sigma_n$'s assemble to give a $\Cat$-operad known as the Barrat-Eccles operad. That is, there is a multicategory $E\Sigma_*$ with a single object and such that the $n$-multilinear maps are given by $$E\Sigma_*(\underbrace{*,\dots, *}_{n\text{ times}};*)=E\Sigma_n,$$
and the composition is defined for $\sigma\in \Sigma_n$ and $\tau_i\in \Sigma_{k_i}$ as $\gamma(\sigma;\tau_1,\dots, \tau_n)=\sigma\la\tau_1,\dots ,\tau_n\ra,$(see \Cref{gam}). The composition for 2-cells is forced by the previous. By the coherence theorem for pseudo symmetric multifunctors \cite{M23} we can rigidify the multifunctor $T\colon \K\to \talg$ and turn it into a symmetric multifunctor $\phi(T)\colon \K\times E\Sigma_*\to \talg .$ For objects, $\phi(T)(A,*)=\phi(T)(A)=TA,$ for 1-cells $\phi(T)(f,\sigma)=T(f\sigma^{-1})\sigma,$ for 2-cells $\phi(T)(\alpha,1_\sigma)=T(\alpha\sigma^{-1})\sigma,$ and also $\phi(T)(1_f,\sigma\to \tau)=(T_{\tau\sigma^{-1};f\sigma^{-1}})\sigma.$ Let's focus on the 1-cell $1_{A_1\times\cdots\times A_n}.$ It's easy to check that for $\omega_\sigma$ defined in \Cref{hpcoherence}, we have that $\omega_\sigma=\phi(T)(1_{A_1\times \cdots\times A_n},\sigma^{-1}).$ Thus, we get a map
$$E\Sigma_n\to \talg (TA_1,\dots, TA_n;T(A_1\times \cdots\times A_n))$$
which sends $\sigma$ to $\phi(T)(1_{A_1\times \cdots\times A_n},\sigma^{-1})=\omega_\sigma,$ and $\sigma\to\tau$ to $\phi(T)\left(1_{A_1\times \cdots\times A_n},\sigma^{-1}\to \tau^{-1}\right).$
\end{rmk}
\appendix
\section{}\label{appendix}
Here are the definitions of $C$-multicategory, (symmetric) $C$-multifunctor, and pseudo symmetric $\Cat$-multifunctor. 
\begin{notn}\label{gam} If $\sigma \in \Sigma_n$ and $\tau_i\in \Sigma_{k_{i}}$ for $1\leq i \leq n,$ we will denote by $\sigma\langle \tau_1,\dots, \tau_n\rangle\in\Sigma_{k_1+\cdots+k_n}$ the permutation that swaps $n$ blocks of lengths $k_1,\dots, k_n$ according to $\sigma$ and each block of length $k_i$ according to $\tau_i.$  
\end{notn}
\begin{defi}\label{multicat}
    If $(C,\otimes ,1,\rho ,\lambda )$ is a symmetric monoidal category, a $C$-\textit{multicategory} $(\M,\gamma,1)$ consists of the following data:
    \begin{itemize}
        \item A class of objects $\ob(\M).$
        \item For every $n\geq 0,$ $\la a\ra=\la a_i\ra_{i=1}^n\in\ob(\M)^n$  and $b\in \ob(\M),$ an object in $C$ denoted by
        $$\M(\la a\ra;b)=\M(a_1,\dots,a_n;b).$$
        We will write $\la a\ra$ instead of $\la a_i\ra_{i=1}^n$ when $n$ is clear from the context or irrelevant. [In the case $C=\Cat,$ an object $f$ of $\M(\la a\ra;b)$ will be called an $n$-ary 1-cell with input $\la a\rangle$ and output $b$ and will be denoted as $f\colon \la a \ra \to b.$ Similarly, we will call $\alpha\colon f\to g$ in $\M(\la a\ra;b)(f,g)$ an $n$-ary 2-cell.]
        \item For each $n\geq 0,$  $\la a\ra\in\ob(\M)^n,$ $b\in\ob(\M),$ and $\sigma\in \Sigma_n,$ a $C$-isomorphism
        \begin{center}
            \begin{tikzcd}
                \M(\la a\ra;b)\ar[r,"\sigma","\cong"swap]&\M(\la a\ra \sigma;b)
            \end{tikzcd}
        \end{center}
        called the right $\sigma$ action or the symmetric group action. Here
        $$\la a\ra\sigma=\la a_1,\dots,a_n\ra \sigma=\la a_{\sigma(1)},\dots,a_{\sigma(n)}\ra.$$ [In the case $C=\Cat$ we write $f\sigma$ for the image of an $n$-ary 1-cell $f\colon \la a\ra \to b$ in $\M$ and similarly for 2-cells.]
        \item For each object $a\in\ob(\M),$ a morphism
        \begin{center}
            \begin{tikzcd}
                1\ar[r,"1_a"]&\M(a;a)
            \end{tikzcd}
        \end{center}
        called the $a$-unit. In the case $C=\Cat$ we notice that  if $a\in\ob(\M),$ $1_a\colon a\to a$ is a 1-ary 1-cell while if $f\colon \la a\ra\to b$ is an $n$-ary 1-cell, then $1_f\colon f\to f$ is an $n$-ary 2-cell in $\M(\la a\ra;b)(f,f)$ so our notation is unambiguous.
        \item For every $c\in\ob(\M),$ $n\geq 0,$ $\la b \ra =\la b_j\ra_{j=1}^n\in \ob(M)^n,$ $k_j\geq 0$ for $1\leq j\leq n,$   and $\la a_j\ra=\la a_{j,i}\ra_{i=1}^{k_j}\in \ob(\M)^{k_j}$ for $1\leq j \leq n,$   a morphism in $C,$
        \begin{center}
            \begin{tikzcd}
                \M(\la b\ra;c)\otimes \bigotimes\limits_{j=1}^n\M(\la a_j\ra;b_j)\ar[r,"\gamma"]&\M(\la a\ra;c),
            \end{tikzcd}
        \end{center}
        where we adopt the convention that $\la a\ra \in \ob(\M)^k,$ where $k=\sum_{i=1}^n k_j,$ denotes the concatenation of the varying $a_j$'s for $j=1,\dots,n$. That is,
        $$\la a\ra =\la a_1,\dots,a_n\ra=\la\la a_j\ra\ra_{j=1}^n= \la a_{1,1},\dots, a_{1,k_1},a_{2,1,}\dots,a_{n-1,k_{n-1}}a_{n,1},\dots,a_{n,k_n}\ra.$$
    \end{itemize}
    The previous data are required to satisfy the following axioms.
    \begin{itemize}
        \item \textbf{Symmetric group action}: For every $n\geq 0,$ $\la a\ra\in \ob(\M),$  $b\in\ob(\M),$ and $\sigma,\tau$ in  $\Sigma_n$ the following diagram commutes in $C:$
        \begin{center}
\begin{tikzcd}
	{\M(\la a\ra ;b)} & {\M(\la a\sigma\ra ;b)} & {\M(\la a\sigma\tau\ra ;b).}
	\arrow["\sigma", from=1-1, to=1-2]
	\arrow["{\sigma\tau}"', curve={height=12pt}, from=1-1, to=1-3]
	\arrow["\tau", from=1-2, to=1-3]
\end{tikzcd}
        \end{center}
        \item \textbf{Identity:} the identity permutation $\id_n\in\Sigma_n$  act as the identity morphism on $\M(\la a\ra;b).$
        \item \textbf{Associativity:} For every $d\in\ob(\M),$ $  n\geq 1,$ $\la c\ra=\la c_j\ra_{j=1}^n\in\ob(\M)^n,$ $k_j\geq 0$ for $1\leq j\leq n$ with $k_j\geq 1$ for at least one $j,$ $\la b_j\ra=\la b_{j,i}\ra_{i=1}^{k_j}\in\ob(\M)^{k_j}$ for $1\leq j\leq n,$ $l_{i,j}\geq 0$ for $1\leq j\leq n $ and $1\leq i\leq k_j,$   and $\la a_{j,i}\ra=\langle a_{j,i,p}\rangle_{p=1}^{l_{i,j}}\in\ob(\M)^{l_{i,j}}$ for $1\leq j \leq n$ and $1\leq i\leq k_j,$ the following \textit{associativity diagram} commutes in $C$:
        \begin{center}
          \adjustbox{scale=0.75}{
\begin{tikzcd}
	{\M (\la c\ra;d)\otimes\left(\bigotimes\limits_{j=1}^n \M(\la b_j\ra;c_j)\right)\otimes \bigotimes\limits_{j=1}^n\left(\bigotimes\limits_{i=1}^{k_j}\M (\la a_{j,i}\ra ;b_{j,i})\right)} & {\M(\la \la b_j\ra\ra_{j=1}^n,d)\otimes\bigotimes\limits_{j=1}^n\left(\bigotimes\limits_{i=1}^{k_j}\M (\la a_{j,i}\ra ;b_{j,i})\right)} \\
	{\M (\la c\ra;d)\otimes \bigotimes\limits_{j=1}^n\left(\M(\la b_j;c_j\ra)\bigotimes\limits_{i=1}^{k_j}\M (\la a_{j,i}\ra ;b_{j,i})\right)} \\
	{\M (\la c\ra;d)\otimes \bigotimes\limits_{j=1}^n \M(\la a_j\ra ;c_j)} & {\M (\la a\ra ;b).}
	\arrow["{\gamma\times1}", from=1-1, to=1-2]
	\arrow["\cong"', from=1-1, to=2-1]
	\arrow["\gamma", from=1-2, to=3-2]
	\arrow["{1\otimes \bigotimes_{j=1}^n\gamma}"', from=2-1, to=3-1]
	\arrow["\gamma"', from=3-1, to=3-2]
\end{tikzcd}}
        \end{center}
        \item \textbf{Unity:} Suppose $b\in \ob(\M)$ and $\la a\ra=\la a_j\ra_{j=1}^n\in\ob(\M)^n,$ then the diagrams
        \begin{center}
\adjustbox{scale=0.85}{
\begin{tikzcd}
	{\M (\la a\ra ;b)\otimes \bigotimes\limits_{j=1}^n1} &&[-20pt]&[-20pt] {1\otimes\M (\la a\ra ;b)} \\
	{\M (\la a\ra ;b)\otimes \bigotimes\limits_{j=1}^n\M (a_j;a_j)} & {\M(\la a\ra ;b),} & {\text{and}} & {\M(b;b)\otimes \M(\la a\ra;b)} & {\M( \la a\ra ;b)}
	\arrow["{1\otimes \bigotimes_{j=1}^n1_{a_j}}"', from=1-1, to=2-1]
	\arrow["\rho\circ\dots\circ\rho", from=1-1, to=2-2]
	\arrow["{1_b\otimes 1}"', from=1-4, to=2-4]
	\arrow["\lambda", from=1-4, to=2-5]
	\arrow["\gamma"', from=2-1, to=2-2]
	\arrow["\gamma"', from=2-4, to=2-5]
\end{tikzcd}}
            \end{center}
        commute.
        \item \textbf{Top equivariance:} For every $c\in\ob(\M),$ $n\geq 1,$ $\la b \ra =\la b_j\ra_{j=1}^n\in\ob(\M)^n,$ $k_j\geq 0$ for $1\leq j\leq n,$ $\la a_j\ra=\la a_{j,i}\ra_{i=1}^{k_j}\in \ob(\M)^{k_j}$ for $1\leq j \leq n,$ and $\sigma\in \Sigma_n,$ the following diagram commutes:
        \begin{center}
            \begin{tikzcd}
                \M(\la b\ra;c)\otimes\bigotimes\limits_{j=1}^n\M(\la a_j\ra;b_j)\ar[r,"\sigma\otimes \sigma^{-1}"]\ar[d,"\gamma"swap]&[20pt] \M(\la b\ra\sigma;c)\otimes\bigotimes\limits_{j=1}^n\M(\la a_{\sigma(j)}\ra;b_{\sigma(j)})\ar[d,"\gamma"]\\
                \M(\la a_1\ra,\dots,\la a_n\ra;c)\ar[r,"\sigma\bigl\langle\id_{k_{\sigma(1)}}{,}\dots{,}\id_{k_{\sigma(n)}}\bigr\rangle"swap]&\M(\la a_{\sigma(1)}\ra,\dots,\la a_{\sigma(n)}\ra;c).
            \end{tikzcd}
            \end{center}
            Here $\sigma^{-1}$ is the unique isomorphism in $C,$ given by the coherence theorem for symmetric monoidal categories, that permutes the factors $\M(\la a_j\ra,b_j)$ according to $\sigma^{-1}.$
        \item \textbf{Bottom equivariance:} For $\la a_j\ra,\la b\ra$ and $c$ as in Top equivariance, the following diagram commutes:
        \begin{center}
            \begin{tikzcd}
                \M(\la b\ra;c)\otimes\bigotimes\limits_{j=1}^n\M(\la a_j\ra;b_j)\ar[r,"\id\otimes \bigotimes\limits_{j=1}^n\tau_j"]\ar[d,"\gamma"swap]&[20pt] \M(\la b\ra,c)\otimes\bigotimes\limits_{j=1}^n\M(\la a_j\ra\tau_j;b_j)\ar[d,"\gamma"]\\
                \M(\la a_1\ra,\dots,\la a_n\ra;c)\ar[r,"\id_n\bigl\langle\tau_1{,}\dots{,}\tau_n\bigr\rangle"swap]&\M(\la a_1\ra\tau_1,\dots,\la a_n\ra\tau_n;c).
            \end{tikzcd}
            \end{center}
    \end{itemize}
    This concludes the definition of a $C$-multicategory.
\end{defi}
\begin{defi}\label{multifunc}
    A symmetric $C$-multifunctor $F\colon \M\to\Nn$ between $C$-multicategories $\M$ and $\Nn$ consists of the following data:
    \begin{itemize}
        \item An object assignment $F\colon \ob(\M)\to\ob(\Nn)$.
        \item For each $n\geq 0,$ $\la a\ra\in\ob(\M)^n$ and $b\in\ob(\M)$ a $C$ morphism
        \begin{center}
            \begin{tikzcd}
                \M(\la a\ra;b)\ar[r,"F"]&\Nn(\la Fa\ra;Fb).
            \end{tikzcd}
        \end{center}
    \end{itemize}
    These data are required to preserve units, composition, and the action of the symmetric group.
    \begin{itemize}
        \item\textbf{Units:} For each object $a\in\ob(\M),$ 
        $F(1_a)=1_{Fa},$ i.e., the following diagram commutes in $C:$
        \begin{center}
\begin{tikzcd}
	1 & {\M (a,a)} & {\Nn (Fa,Fa).}
	\arrow["{1_a}", from=1-1, to=1-2]
	\arrow["{1_{Fa}}"', curve={height=12pt}, from=1-1, to=1-3]
	\arrow["F", from=1-2, to=1-3]
\end{tikzcd}
        \end{center}
        \item \textbf{Composition:} For every $c\in\ob(\M),$ $n\geq 0,$ $\la b \ra =\la b_j\ra_{j=1}^n\in \ob(M)^n,$ $k_j\geq 0$ for $1\leq j\leq n,$   and $\la a_j\ra=\la a_{j,i}\ra_{i=1}^{k_j}\in \ob(\M)^{k_j}$ for $1\leq j \leq n$ and $1\leq i \leq n,$ the following diagram commutes in $C:$
        \begin{center}
                \begin{tikzcd}
                    \M (\la b\ra;c)\otimes\bigotimes\limits_{j=1}^n\M(\langle a_j\rangle;b_j)\ar[r,"F\otimes\bigotimes\limits_{j=1}^n F"]\ar[d,"\gamma"{swap,name=h}]&[5pt]\Nn(\la Fb\ra;Fc)\otimes\bigotimes\limits_{j=1}^n\Nn(\langle Fa_j\rangle;Fb_j)\ar[d,"\gamma"name=g]\\
            \M(\la a\ra;c)\ar[r,"F"swap]&|[alias=N1]|\Nn(\la Fa\ra;Fc).
                \end{tikzcd}
        \end{center}
        \item\textbf{Symmetric Group Action:} For each $\la a \ra \in\ob(\M)^n$ and $b\in\ob(\M)$ the following diagram commutes in $C:$
        \begin{center}
                \begin{tikzcd}
                    \M(\la a\ra;b)\ar[r,"F"]\ar[d,"\sigma","\cong"swap]&\Nn(\la Fa\ra;Fb)\ar[d,"\sigma","\cong"swap]\\
                    \M(\la a\ra\sigma;b)\ar[r,"F"swap]&\N(\la Fa\ra\sigma;Fb).
                \end{tikzcd}
        \end{center}
    \end{itemize}
    \end{defi}
    Yau defines, in \cite{Y23}, a 2-category, $C$-$\mathbf{Multicat}$  with 0-cells $C$-multicategories, 1-cells symmetric $C$-multifunctors, and 2-cells $C$-multinatural transformations.
 \begin{defi}{\cite[Def. 4.1.1]{Y23}}\label{pseudosymm} Suppose that $\mathcal{M},\mathcal{N}$ are $\Cat$-multicategories. A  \textit{pseudo symmetric $\Cat$-multifunctor} $F\colon \mathcal{M}\to \mathcal{N}$ consists of the following data:
\begin{itemize}
    \item A  function on object sets $F\colon \ob(\mathcal{M})\to\ob(\mathcal{N}).$
    \item For each $\la a\ra\in \ob(\mathcal{M})^n$ and $b\in\ob(\mathcal{M}),$ a component functor
    \begin{center}
\begin{tabular}{c}
    \xymatrix{\mathcal{M}(\langle a\rangle;b)\ar[r]^-{F}&\mathcal{N}(\langle Fa\rangle;Fb).}
\end{tabular}
\end{center}
\item For each $\sigma\in \Sigma_n,$ $\langle a\rangle\in\ob(\M)^n,$ $b\in \ob(\M),$  a natural isomorphism $F_{\sigma,\langle a\rangle,b}$
\begin{center}
\begin{tikzcd}
  \mathcal{M}(\langle a\rangle;b)\ar[r,"F"]\ar[d,"\sigma"swap]&|[alias=N]|\mathcal{N}(\langle Fa\rangle;Fb)\ar[d,"\sigma"]\\
  |[alias=M]|\mathcal{M}(\langle a\rangle\sigma;b)\ar[r,"F"swap]&\mathcal{N}(\langle Fa\rangle\sigma;Fb).
  \arrow[Rightarrow,from=M, to=N,shorten >=4mm,shorten <=4mm,"F_{\sigma,\langle a\rangle,b}"swap,"\cong"]
\end{tikzcd}
\end{center}
When $\langle a\rangle$ and $b$ are clear from the context we write simply $F_\sigma,$ and if $f\colon \la a\ra\to b $ we will denote by $F_{\sigma,\la a\ra,b;f}=F_{\sigma;f}\colon F(f\sigma)\to F(f)\sigma$ the 2-cell in $\Nn(\la Fa\ra\sigma;Fb)$ corresponding to the component of $F_\sigma$ at $f.$  
\end{itemize}
These data are subject to the same axioms of unit and composition preservation as a symmetric $\mathbf{Cat}$-multifunctor, but we replace the symmetric group action preservation axiom by the following four axioms.
\begin{itemize}
    \item \textbf{Unit permutation:} Let $n\geq 0,$ $\la a\ra\in\ob(\M)^n$ and $b\in\ob(\M),$ then
       \[ F_{\text{id}_n,\langle a\rangle,b}=1_F.\]
    \item \textbf{Product permutation:} Let $n\geq 0,$ $\la a\ra\in\ob(\M)^n,$ $b\in\ob(M)$ and $\sigma,\tau \in \Sigma_n.$ Then, the following pasting digrams are equal.
    \begin{center}
\begin{tikzcd}
	{\M (\la a\ra ;b)} & {\Nn (\la Fa\ra ;Fb)} &[-10pt]&[-20pt] {\M (\la a\ra ;b)} & {\Nn (\la Fa\ra ;Fb)} \\
	{\M (\la a\ra\sigma ;b)} & {\Nn (\la Fa\ra\sigma ;Fb)} & {=} \\
	{\M (\la a\ra\sigma\tau ;b)} & {\Nn (\la Fa\ra\sigma\tau ;Fb)} && {\M (\la a\ra\sigma\tau ;b)} & {\Nn (\la Fa\ra\sigma\tau ;Fb).}
	\arrow["F", from=1-1, to=1-2]
	\arrow["\sigma"', from=1-1, to=2-1]
	\arrow["\sigma", from=1-2, to=2-2]
	\arrow["F", from=1-4, to=1-5]
	\arrow["{\sigma\tau}"', from=1-4, to=3-4]
	\arrow["{\sigma\tau}", from=1-5, to=3-5]
	\arrow["{F_\sigma}"', between={0.3}{0.7}, Rightarrow, from=2-1, to=1-2]
	\arrow["F"', from=2-1, to=2-2]
	\arrow["\tau"', from=2-1, to=3-1]
	\arrow["\tau", from=2-2, to=3-2]
	\arrow["{F_{\tau}}"', between={0.3}{0.7}, Rightarrow, from=3-1, to=2-2]
	\arrow["F"', from=3-1, to=3-2]
	\arrow["{F_{\sigma\tau}}"', between={0.4}{0.6}, Rightarrow, from=3-4, to=1-5]
	\arrow["F"', from=3-4, to=3-5]
\end{tikzcd}
    \end{center}
    \item \textbf{Top equivariance:} For every $c\in\ob(\M),$ $n\geq 0,$ $\la b \ra =\la b_j\ra_{j=1}^n\in \ob(M)^n,$ $k_j\geq 0$ for $1\leq j\leq n,$   and $\la a_j\ra=\la a_{j,i}\ra_{i=1}^{k_j}\in \ob(\M)^{k_j}$ for $1\leq j \leq n,$ and $\sigma\in \Sigma_n,$ the pasting
\begin{center}
\begin{tikzcd}
	{\M (\la b\ra ;c)\times \prod\limits_{j=1}^n\M (\la a_j\ra ;b_j)} & {\Nn (\la Fb\ra ;Fc)\times \prod\limits_{j=1}^n\Nn (\la Fa_j\ra ;Fb_j)} \\
	{\M (\la\la a_j\ra\ra_{j=1}^n ;c)} & {\Nn (\la\la Fa_j\ra\ra_{j=1}^n ;Fc)} \\
	{\M (\la\la a_{\sigma(j)}\ra\ra_{j=1}^n; c)} & {\Nn (\la\la Fa_{\sigma(j)}\ra\ra_{j=1}^n; Fc)}
	\arrow["{\prod F}", from=1-1, to=1-2]
	\arrow["\gamma"', from=1-1, to=2-1]
	\arrow["\gamma", from=1-2, to=2-2]
	\arrow["F"', from=2-1, to=2-2]
	\arrow["{\sigma\la \id_{k_{\sigma(j)}}\ra}"', from=2-1, to=3-1]
	\arrow["{\sigma\la \id_{k_{\sigma(j)}}\ra}", from=2-2, to=3-2]
	\arrow["{F_{\sigma\la\id_{k_{\sigma(j)}}\ra}}"', between={0.3}{0.7}, Rightarrow, from=3-1, to=2-2]
	\arrow["F"{description}, from=3-1, to=3-2]
\end{tikzcd}
\end{center}
equals
    \begin{center}
\begin{tikzcd}
	{\M (\la b\ra ;c)\times \prod\limits_{j=1}^n\M (\la a_j\ra ;b_j)} & {\Nn (\la Fb\ra ;Fc)\times \prod\limits_{j=1}^n\Nn (\la Fa_j\ra ;Fb_j)} \\
	{\M (\la b\ra \sigma ;c)\times \prod\limits_{j=1}^n\M (\la a_{\sigma(j)}\ra ;b_{\sigma(j)})} & {\Nn (\la Fb\ra \sigma ;Fc)\times \prod\limits_{j=1}^n\Nn (\la Fa_{\sigma(j)}\ra ;Fb_{\sigma(j)})} \\
	{\M (\la\la a_{\sigma(j)}\ra\ra_{j=1}^n; c)} & {\Nn (\la\la Fa_{\sigma(j)}\ra\ra_{j=1}^n; Fc).}
	\arrow["{\prod F}", from=1-1, to=1-2]
	\arrow["{\sigma\times \sigma^{-1}}"', from=1-1, to=2-1]
	\arrow["{\sigma\times \sigma^{-1}}", from=1-2, to=2-2]
	\arrow["{F_\sigma \times 1}"', between={0.3}{0.7}, Rightarrow, from=2-1, to=1-2]
	\arrow["{\prod F}"', from=2-1, to=2-2]
	\arrow["\gamma"', from=2-1, to=3-1]
	\arrow["\gamma", from=2-2, to=3-2]
	\arrow["F"', from=3-1, to=3-2]
\end{tikzcd}
\end{center}
Here $\sigma\la \id_{k_{\sigma(j)}}\ra=\sigma\la\id_{k_{\sigma(1)}},\dots,\id_{k_\sigma(n)}\ra.$  
\item \textbf{Bottom Equivariance:} For every $c\in\ob(\M),$ $n\geq 0,$ $\la b \ra =\la b_j\ra_{j=1}^n\in \ob(M)^n,$ $k_j\geq 0$ for $1\leq j\leq n,$   and $\la a_j\ra=\la a_{j,i}\ra_{i=1}^{k_j}\in \ob(\M)^{k_j}$ for $1\leq j \leq n$ and $1\leq i \leq k_j,$ and $\tau_j \in \Sigma_{k_j},$ the following two diagram
\begin{center}
\begin{tikzcd}
	{\M (\la b\ra ;c)\times \prod\limits_{j=1}^n\M (\la a_j\ra ;b_j)} & {\Nn (\la Fb\ra ;Fc)\times \prod\limits_{j=1}^n\Nn (\la Fa_j\ra ;Fb_j)} \\
	{\M (\la\la a_j\ra\ra_{j=1}^n ;c)} & {\Nn (\la\la Fa_j\ra\ra_{j=1}^n ;Fc)} \\
	{\M (\la\la a_j\ra\tau_j\ra_{j=1}^n; c)} & {\Nn (\la\la Fa_j\ra\tau_j\ra_{j=1}^n; Fc)}
	\arrow["{\prod F}", from=1-1, to=1-2]
	\arrow["\gamma"', from=1-1, to=2-1]
	\arrow["\gamma", from=1-2, to=2-2]
	\arrow["F"', from=2-1, to=2-2]
	\arrow["{\id_n \la \tau_j\ra}"', from=2-1, to=3-1]
	\arrow["{\id_n\la \tau_j\ra}", from=2-2, to=3-2]
	\arrow["{F_{\id_n \la \tau_j\ra}}"', between={0.3}{0.7}, Rightarrow, from=3-1, to=2-2]
	\arrow["F"', from=3-1, to=3-2]
\end{tikzcd}
\end{center}
equals
\begin{center}
\begin{tikzcd}
	{\M (\la b\ra ;c)\times \prod\limits_{j=1}^n\M (\la a_j\ra ;b_j)} & {\Nn (\la Fb\ra ;Fc)\times \prod\limits_{j=1}^n\Nn (\la Fa_j\ra ;Fb_j)} \\
	{\M (\la b\ra ;c)\times \prod\limits_{j=1}^n\M (\la a_j\ra\tau_j ;b_j)} & {\Nn (\la Fb\ra ;Fc)\times \prod\limits_{j=1}^n\Nn (\la Fa_j\ra\tau_j ;Fb_j)} \\
	{\M (\la\la a_j\ra\tau_j\ra_{j=1}^n; c)} & {\Nn (\la\la Fa_j\ra\tau_j\ra_{j=1}^n; Fc).}
	\arrow["{\prod F}", from=1-1, to=1-2]
	\arrow["{1\times \prod\tau_j}"', from=1-1, to=2-1]
	\arrow["{1\times \prod\tau_j}", from=1-2, to=2-2]
	\arrow["{1\times\prod F_{\tau_j}}"', between={0.3}{0.7}, Rightarrow, from=2-1, to=1-2]
	\arrow["F"', from=2-1, to=2-2]
	\arrow["\gamma"', from=2-1, to=3-1]
	\arrow["\gamma", from=2-2, to=3-2]
	\arrow["F"', from=3-1, to=3-2]
\end{tikzcd}
\end{center}
The domain and codomain of the pasting diagrams in the Top and Bottom Equivariance axioms are equal by top and bottom equivariance for $\M$ and $\N.$ The preservation of $\gamma$ by $F$ guarantees that the empty squares commute.
\end{itemize}
\end{defi}
\printbibliography

\end{document}